\documentclass[12pt,reqno]{amsart}
\usepackage{mathrsfs,amssymb,graphicx,verbatim}
\usepackage{paralist}
\usepackage[breaklinks,pdfstartview=FitH]{hyperref}
\usepackage{upgreek}
\renewcommand{\gamma}{\upgamma}

\newcommand{\sign}{{\mathrm{sign}}}

\renewcommand{\d}{\delta}
\newcommand{\D}{{\mathcal{D}}}

\newcommand{\e}{\varepsilon}
\newcommand{\R}{\mathbb R}
\newcommand{\1}{\mathbf 1}
\newcommand{\eqdef}{\stackrel{\mathrm{def}}{=}}
\newcommand{\conv}{\mathrm{conv}}
\renewcommand{\le}{\leqslant}
\renewcommand{\ge}{\geqslant}
\renewcommand{\leq}{\leqslant}

\newtheorem{theorem}{Theorem}[section]
\newtheorem{proposition}[theorem]{Proposition}
\newtheorem{lemma}[theorem]{Lemma}
\newtheorem{corollary}[theorem]{Corollary}

\theoremstyle{remark}

\newtheorem{remark}[theorem]{Remark}
\newtheorem{question}[theorem]{Question}
\theoremstyle{definition}

\newcommand{\n}{\{1,\ldots,n\}}
\newcommand{\m}{\{1,\ldots,m\}}

\renewcommand{\subset}{\subseteq}

\newcommand{\E}{\mathbb{E}}
\newcommand{\F}{\mathbb F}
\newcommand{\N}{\mathbb N}
\newcommand{\Z}{\mathbb Z}

\renewcommand{\setminus}{\smallsetminus}
\newcommand{\BMW}{\mathrm{BMW}}

\newcommand{\Lip}{\mathrm{Lip}}

\newcommand{\f}{\phi}

\DeclareMathOperator{\diam}{diam}

\begin{document}

%\title{Comparison inequalities for nonlinear spectral gaps}
\title{Comparison of metric spectral gaps}
\thanks{Supported by
NSF grant CCF-0832795, BSF grant 2010021, the Packard Foundation and
the Simons Foundation.}
%Part of this work was completed when the author was visiting Universit\'e Pierre et
%Marie Curie, Paris, France.}

\author{Assaf Naor}
\address{Courant Institute\\ New York University}
\email{naor@cims.nyu.edu}

\date{}

\begin{abstract}
Let $A=(a_{ij})\in M_n(\R)$ be an $n$ by $n$ symmetric stochastic matrix.  For $p\in [1,\infty)$ and a metric space $(X,d_X)$, let $\gamma(A,d_X^p)$ be the infimum over those $\gamma\in (0,\infty]$ for which every $x_1,\ldots,x_n\in X$ satisfy
 $$
 \frac{1}{n^2} \sum_{i=1}^n\sum_{j=1}^n d_X(x_i,x_j)^p\le \frac{\gamma}{n}\sum_{i=1}^n\sum_{j=1}^n a_{ij} d_X(x_i,x_j)^p.
 $$
 Thus $\gamma(A,d_X^p)$ measures the magnitude of the {\em nonlinear spectral gap} of the matrix $A$ with
 respect to the kernel $d_X^p:X\times X\to [0,\infty)$. We study pairs of metric spaces $(X,d_X)$ and
 $(Y,d_Y)$ for which there exists $\Psi:(0,\infty)\to (0,\infty)$ such that
 $\gamma(A,d_X^p)\le \Psi\left(\gamma(A,d_Y^p)\right)$ for every symmetric stochastic $A\in M_n(\R)$
 with $\gamma(A,d_Y^p)<\infty$. When $\Psi$ is linear a complete geometric characterization is obtained.

 Our estimates on nonlinear spectral gaps yield new embeddability results as well as new nonembeddability results. For example, it is shown that if $n\in \N$ and $p\in (2,\infty)$
 then for every $f_1,\ldots,f_n\in L_p$ there exist $x_1,\ldots,x_n\in L_2$ such that
 \begin{equation}\label{eq:p factor}
 \forall\, i,j\in \{1,\ldots,n\},\quad \|x_i-x_j\|_2\lesssim p\|f_i-f_j\|_p,
 \end{equation}
 and
 $$
 \sum_{i=1}^n\sum_{j=1}^n \|x_i-x_j\|_2^2=\sum_{i=1}^n\sum_{j=1}^n \|f_i-f_j\|_p^2.
 $$
 This statement is impossible for $p\in [1,2)$, and the asymptotic dependence on $p$ in~\eqref{eq:p factor}
 is sharp. We also obtain the best known lower bound on the $L_p$ distortion of Ramanujan graphs,
 improving over the work of Matou\v{s}ek. Links to Bourgain--Milman--Wolfson type and a conjectural nonlinear
Maurey--Pisier theorem are studied.
\end{abstract}


\maketitle

\vspace{-0.3in}

%\newpage

\tableofcontents

\vspace{-0.33in}

\section{Introduction}

 The decreasing rearrangement of the eigenvalues of an $n$ by $n$ symmetric stochastic matrix
 $A=(a_{ij})\in M_n(\R)$ is denoted below by
 $$
 1=\lambda_1(A)\ge \lambda_2(A)\ge\ldots\ge \lambda_n(A).
 $$
We also set
\begin{equation}\label{eq:def lambda}
\lambda(A)\eqdef \max_{i\in \{2,\ldots,n\}}|\lambda_i(A)|=\max\left\{\lambda_2(A),-\lambda_n(A)\right\}.
\end{equation}
 For $p\in [1,\infty)$ and a metric space $(X,d_X)$, we denote by $\gamma(A,d_X^p)$ the infimum over those $\gamma\in (0,\infty]$ for which every $x_1,\ldots,x_n\in X$ satisfy
 \begin{equation}\label{eq:def gamma}
 \frac{1}{n^2} \sum_{i=1}^n\sum_{j=1}^n d_X(x_i,x_j)^p\le \frac{\gamma}{n}\sum_{i=1}^n\sum_{j=1}^n a_{ij} d_X(x_i,x_j)^p.
 \end{equation}
 Thus for $p=2$ and $X=(\R,d_\R)$, where we denote $d_\R(x,y)\eqdef |x-y|$ for every $x,y\in \R$, we have
 \begin{equation}\label{eq:R case}
 \gamma(A,d_\R^2)=\frac{1}{1-\lambda_2(A)}.
 \end{equation}
 For this reason one thinks of $\gamma(A,d_X^p)$ as measuring the magnitude
 of the {\em nonlinear spectral gap} of the matrix $A$ with respect
 to the kernel $d_X^p:X\times X\to [0,\infty)$. See~\cite{MN-towards} for more information on the topic of nonlinear spectral
gaps.

Suppose that $(X,d_X)$ is a metric space with $|X|\ge 2$ and fix
distinct points $a,b\in X$. Fix also $p\in [1,\infty)$, $n\in \N$,
and an $n$ by $n$ symmetric stochastic matrix $A=(a_{ij})$. By
considering  $x_1,\ldots,x_n\in \{a,b\}$ in~\eqref{eq:def gamma}, we
see that
$$
\gamma(A,d_X^p)\ge \max_{\emptyset\neq S\subsetneq \n}\frac{|S|(n-|S|)}{n\sum_{(i,j)\in S\times (\n\setminus S)}a_{ij}}.
$$
It therefore follows from Cheeger's inequality~\cite{Che70} (in our context,
see~\cite{JS88} and~\cite{LS88}) that
\begin{equation}\label{eq:cheeger}
\gamma(A,d_X^p)\gtrsim \frac{1}{\sqrt{1-\lambda_2(A)}}.
\end{equation}
Consequently, any finite upper bound on $\gamma(A,d_X^p)$
immediately implies a spectral gap estimate for the matrix $A$. The
ensuing discussion always tacitly assumes that metric spaces
contain at least two points.

In~\eqref{eq:cheeger}, and in what follows, the notations $U\lesssim
V$ and $V\gtrsim U$ mean that $U\le CV$ for some universal constant
$C\in (0,\infty)$. If we need to allow $C$ to depend on parameters,
we indicate this by subscripts, thus e.g. $U \lesssim_{\beta} V$
means that $U \leq C(\beta) V$ for some $C(\beta)\in (0,\infty)$
which is allowed to depend only on the parameter $\beta$. The
notation $U\asymp V$ stands for $(U\lesssim V) \wedge  (V\lesssim
U)$, and correspondingly the notation $U\asymp_\beta V$ stands for
$(U\lesssim_\beta V) \wedge (V\lesssim_\beta U)$.

A simple application of the triangle inequality
(see~\cite[Lem.~2.1]{MN-towards}) shows that $\gamma(A,d_X^p)$ is
finite if and only if $\lambda_2(A)<1$ (equivalently, the matrix $A$
is ergodic, or the graph on $\{1,\ldots,n\}$  whose edges are the
pais $\{i,j\}$ for which $a_{ij}>0$ is connected).

It is often quite difficult to obtain good estimates on nonlinear
spectral gaps. This difficulty is exemplified by several problems in
metric geometry that can be cast as estimates on nonlinear spectral
gaps: see~\cite{Laff08,Laf09,MN-towards,MN12-random,Lia13,deSal13}
for some specific examples, as well as the ensuing discussion on
nonlinear type. In this general direction, here we investigate the
following basic ``meta-problem."
\begin{question}[Comparison of nonlinear spectral gaps]\label{Q:meta}
Given $p\in [1,\infty)$, characterize those pairs of metric spaces $(X,d_X)$ and $(Y,d_Y)$ for which there exists an increasing function $\Psi=
\Psi_{X,Y}:(0,\infty)\to (0,\infty)$ such that for every $n\in \N$ and every ergodic symmetric stochastic $A\in M_n(\R)$ we have
\begin{equation}\label{eq:Q meta}
\gamma\!\left(A,d_X^p\right)\le \Psi\left(\gamma\!\left(A,d_Y^p\right)\right).
\end{equation}
\end{question}

The case $p=2$ and $(Y,d_Y)=(\R,d_\R)$  of Question~\ref{Q:meta}  is especially important, so
we explicitly single it out as follows. See also the closely related question that Pisier posed as Problem~3.1 in~\cite{Pis10-interpolation}.
\begin{question}[Bounding nonlinear gaps by linear gaps]\label{Q:l2 case}
Characterize those metric spaces $(X,d_X)$ for which there exists an increasing function $\Psi=\Psi_X:(0,\infty)\to (0,\infty)$ such that for every $n\in \N$ and every ergodic symmetric stochastic $A\in M_n(\R)$ we have
\begin{equation}\label{eq:Psi euclidean}
\gamma\!\left(A,d_X^p\right)\le \Psi\left(\frac{1}{1-\lambda_2(A)}\right).
\end{equation}
\end{question}

Question~\ref{Q:meta} and Question~\ref{Q:l2 case} seem to be difficult, and they might not have a useful simple-to-state
answer. As an indication of this, in~\cite{MN12-random} it is shown
that there exists a $CAT(0)$ metric space $(X,d_X)$,
%(a simply connected
%geodesic metric space that it nonpositively curved in the sense of
%Aleksandrov; see~\cite{BH-book}),
and for each $k\in \N$ there
exist $n_k\in \N$ with $\lim_{k\to\infty} n_k=\infty$ such that there are symmetric stochastic matrices $A_k,B_k\in
M_{n_k}(\R)$
%(in fact, $A_k,B_k$ are normalized adjacency matrices
%of $3$-regular graphs)
 with $$\sup_{k\in \N} \lambda_2(A_k)<1\qquad \mathrm{and}\qquad \sup_{k\in
\N}\lambda_2(B_k)<1,$$
%(equivalently, the corresponding graphs
%are expanders; see~\cite{HLW}),
yet $$\sup_{k\in \N}
\gamma(A_k,d_X^2)<\infty\qquad \mathrm{and}\qquad \sup_{k\in \N}
\gamma(B_k,d_X^2)=\infty.$$ Such questions are difficult even in the
setting of Banach spaces: a well-known open question (see
e.g.~\cite{Pis10-interpolation}) asks whether~\eqref{eq:Psi
euclidean} holds true when $X$ is a uniformly convex Banach space.
If true, this would yield a remarkable ``linear to nonlinear
transference principle for spectral gaps," establishing in
particular the existence of super-expanders (see~\cite{MN-towards})
with logarithmic girth, a result which could then be used in
conjunction with Gromov's random group
construction~\cite{Gromov-random-group} to rule out the success of
an approach to the Novikov conjecture that was discovered by
Kasparov and Yu~\cite{KY06}. There is little evidence, however, that
every uniformly convex Banach space admits an inequality such
as~\eqref{eq:Psi euclidean}, and we suspect that~\eqref{eq:Psi
euclidean} fails for some uniformly convex Banach spaces.

When the function $\Psi$ appearing in~\eqref{eq:Q meta} can be taken
to be linear, i.e., $\Psi(t)=Kt$ for some $K\in (0,\infty)$,
Theorem~\ref{thm:duality} below provides the following geometric
answer to Question~\ref{Q:meta}. Given $p\in [1,\infty)$ and a
metric space $(Y,d_Y)$, for every $m\in \N$ we denote by
$\ell_p^m(Y)$ the metric space $Y^m$ equipped with the metric
$$
\forall\, x,y\in Y^m,\qquad d_{\ell_p^m(Y)}(x,y)\eqdef \left(\sum_{i=1}^m d_Y(x_i,y_i)^p\right)^{\frac{1}{p}}.
$$
Recall also the standard notation $\ell_p^m=\ell_p^m(\R)$. A
coordinate-wise application of~\eqref{eq:def gamma} shows that every
symmetric stochastic $A$ satisfies
\begin{equation}\label{eq:lift to ellp}
\forall\, m\in \N,\qquad \gamma\!\left(A,d_Y^p\right)=\gamma\!\left(A,d_{\ell_p^m(Y)}^p\right).
\end{equation}

Fixing $D\in [1,\infty)$, suppose that $(X,d_X)$ is a metric space such that for every $n\in \N$ and $x_1,\ldots,x_n\in X$ there exists $m\in \N$ and a non-constant mapping $f:\{x_1,\ldots,x_n\}\to \ell_p^m(Y)$ such that
\begin{equation}\label{eq:average dist D}
\sum_{i=1}^n\sum_{j=1}^n d_{\ell_p^m(Y)}(f(x_i),f(x_j))^p\ge \frac{\|f\|_{\Lip}^p}{D}\sum_{i=1}^n\sum_{j=1}^n d_X(x_i,x_j)^p,
\end{equation}
where $\|f\|_{\Lip}$ denotes the Lipschitz constant of $f$, i.e.,
\begin{equation*}\label{eq:def lip const}
\|f\|_{\Lip}\eqdef \max_{\substack{x,y\in \{x_1,\ldots,x_n\}\\x\neq y}} \frac{d_{\ell_p^m(Y)}(f(x),f(y))}{d_X(x,y)}.
\end{equation*}
Then for every symmetric stochastic matrix $A\in M_n(\R)$,
\begin{eqnarray}\label{eq:duality trivial dir}
\nonumber\frac{1}{n^2}\sum_{i=1}^n\sum_{j=1}^n d_X(x_i,x_j)^p&\stackrel{\eqref{eq:average dist D}}{\le}& \frac{D}{n^2\|f\|_{\Lip}^p}\sum_{i=1}^n\sum_{j=1}^n d_{\ell_p^m(Y)}(f(x_i),f(x_j))^p\\&\stackrel{\eqref{eq:lift to ellp}}{\le}&\nonumber
\frac{D\gamma(A,d_Y^p)}{n}\sum_{i=1}^n\sum_{j=1}^n a_{ij}\frac{d_{\ell_p^m(Y)}(f(x_i),f(x_j))^p}{\|f\|_{\Lip}^p}\\&\le& \frac{D\gamma(A,d_Y^p)}{n}\sum_{i=1}^n\sum_{j=1}^n a_{ij} d_X(x_i,x_j)^p.
\end{eqnarray}
Consequently, $\gamma(A,d_X^p)\le D\gamma(A,d_Y^p)$.

The above simple
argument, combined with standard metric embedding methods, already
implies that a variety of metric spaces satisfy the spectral
inequality~\eqref{eq:Psi euclidean} with $\Psi$ linear. This holds
in particular when $(X,d_X)$ belongs to one of the following classes
of metric spaces:  doubling metric spaces, compact Riemannian
surfaces, Gromov hyperbolic spaces of bounded local geometry,
Euclidean buildings, symmetric spaces, homogeneous Hadamard
manifolds, and forbidden-minor (edge-weighted) graph families; this
topic is treated in Section~\ref{sec:average}. We will also see
below that the same holds true for certain Banach spaces, including
$L_p(\mu)$ spaces for $p\in [2,\infty)$.

The following theorem asserts that the above reasoning using metric
embeddings is the only way to obtain an inequality such
as~\eqref{eq:Q meta} with $\Psi$ linear. Its proof is a duality
argument that is inspired by the proof of a lemma of K.
Ball~\cite[Lem.~1.1]{Ball} (see also~\cite[Lem~5.2]{MN13-ext}).

\begin{theorem}\label{thm:duality}
Fix $n\in \N$ and $p,K\in [1,\infty)$. Given two metric spaces  $(X,d_X)$ and $(Y,d_Y)$, the following assertions are equivalent.
\begin{enumerate}
\item For every symmetric stochastic $n$ by $n$ matrix $A$ we have \begin{equation}\label{eq:linear bound}
\gamma(A,d_X^p)\le K\gamma(A,d_Y^p).
\end{equation}
\item For every $D\in (K,\infty)$ and every $x_1,\ldots,x_n\in X$ there exists $m\in \N$ and  a nonconstant mapping $f:\{x_1,\ldots,x_n\}\to \ell_p^m(Y)$ that satisfies
$$
\sum_{i=1}^n\sum_{j=1}^n d_{\ell_p^m(Y)}(f(x_i),f(x_j))^p\ge \frac{\|f\|_{\Lip}^p}{D}\sum_{i=1}^n\sum_{j=1}^n d_X(x_i,x_j)^p.
$$
\end{enumerate}
\end{theorem}

It is worthwhile to single out the special case of Theorem~\ref{thm:duality} that corresponds to Question~\ref{Q:l2 case}, in which case the embeddings are into $\ell_2$.

\begin{corollary}\label{coro:linear gap}
Fix $n\in \N$ and  $K\in [1,\infty)$. Given a metric space $(X,d_X)$, the following assertions are equivalent.
\begin{enumerate}
\item For every symmetric stochastic matrix $A\in M_n(\R)$ we have $$\gamma(A,d_X^2)\le \frac{K}{1-\lambda_2(A)}.$$
\item For every $D\in (K,\infty)$ and every $x_1,\ldots,x_n\in X$ there exists a nonconstant mapping $f:\{x_1,\ldots,x_n\}\to \ell_2$ that satisfies
$$
\sum_{i=1}^n\sum_{j=1}^n \|f(x_i)-f(x_j)\|_{\ell_2}^2\ge \frac{\|f\|_{\Lip}^2}{D}\sum_{i=1}^n\sum_{j=1}^n d_X(x_i,x_j)^2.
$$
\end{enumerate}
\end{corollary}

In Section~\ref{sec:use markov type} below we prove the following
theorem.

\begin{theorem}\label{thm:p bound gamma}
For every $p\in [2,\infty)$, every $n\in \N$ and every symmetric
stochastic matrix $A\in M_n(\R)$ we have
\begin{equation}\label{eq:p^2 therem}
\gamma\!\left(A,\|\cdot\|_{\ell_p}^2\right)\lesssim \frac{p^2}{1-\lambda_2(A)}.
\end{equation}
\end{theorem}
Inequality~\eqref{eq:p^2 therem} is proved in Section~\ref{sec:use
markov type} via a direct argument using analytic and probabilistic
techniques, but once this inequality is established one can use
duality through Corollary~\ref{coro:linear gap} to deduce  the
following new Hilbertian embedding result for arbitrary finite subsets
of $\ell_p$.
\begin{corollary}\label{coro:p av embed}
If $n\in \N$ and $p\in (2,\infty)$ then for every $x_1,\ldots,x_n\in
\ell_p$ there exist $y_1,\ldots,y_n\in \ell_2$ such that
 \begin{equation}\label{eq:p yaya}
 \forall\, i,j\in \{1,\ldots,n\},\qquad \|y_i-y_j\|_{\ell_2}\lesssim p\|x_i-x_j\|_{\ell_p},
 \end{equation}
 and
 $$
 \sum_{i=1}^n\sum_{j=1}^n \|y_i-y_j\|_{\ell_2}^2=\sum_{i=1}^n\sum_{j=1}^n \|x_i-x_j\|_{\ell_p}^2.
 $$
\end{corollary}
The dependence on $p$ in~\eqref{eq:p yaya} is sharp up to the
implicit universal constant, and the conclusion of
Corollary~\ref{coro:p av embed} is false for $p\in [1,2)$ even if
one allows any dependence on $p$ in~\eqref{eq:p yaya} (provided it
is independent of $n$): for the former statement see
Lemma~\ref{lem:p sharp} below and for the latter statement see
Lemma~\ref{lem:no cube} below. It would be interesting to find a
constructive proof of Corollary~\ref{coro:p av embed}, i.e., a
direct proof that does not rely on duality to show that the desired
embedding exists.

\begin{remark}
Questions in the spirit of Theorem\ref{thm:p bound gamma} have been
previously studied by Matou\v{s}ek~\cite{Mat97}, who proved that
there exists a universal constant $C\in (0,\infty)$ such that for
every $n\in \N$ and every $n$ by $n$ symmetric stochastic matrix $A\in M_n(\R)$,
\begin{equation}\label{eq:matousek extrapolation}
p\in [2,\infty)\implies
\gamma\!\left(A,\|\cdot\|_{\ell_p}^p\right)\le \frac{(Cp)^p}{\left(1-\lambda_2(A)\right)^{p/2}}.
\end{equation}
See~\cite[Lem.~5.5]{BLMN05} and~\cite[Lem.~4.4]{NS11} for the
formulation and proof of this fact (which is known as Matou\v{s}ek's
extrapolation lemma for Poincar\'e inequalities) in the form stated
in~\eqref{eq:matousek extrapolation}. In order to obtain an
embedding result such as Corollary~\ref{coro:p av embed} one needs to
bound $\gamma (A,\|\cdot\|_{\ell_p}^2)$ rather than
$\gamma(A,\|\cdot\|_{\ell_p}^p)$ by a quantity that grows linearly
with $1/(1-\lambda_2(A))$. We do not see how to use Matou\v{s}ek's
approach in~\cite{Mat97} to obtain such an estimate, and we
therefore use an entirely different method (specifically, complex
interpolation and Markov type) to prove Theorem~\ref{thm:p bound
gamma}.
\end{remark}

As stated above, when $p\in [1,2)$ the analogue of Theorem~\ref{thm:p bound gamma} can
hold true only if we allow the right hand side of~\eqref{eq:p^2
therem} depend on $n$. Specifically, we ask the following question.

\begin{question}\label{Q:p ARV}
Is it true that for every $p\in [1,2]$, every $n\in \N$ and every
$n$ by $n$ symmetric stochastic matrix $A$ we have
$$
\gamma(A,\|\cdot\|_{\ell_p}^2)\lesssim \frac{(\log n)^{\frac{2}{p}-1}}{1-\lambda_2(A)}\ ?
$$
\end{question}

Due to Theorem~\ref{thm:duality}, an affirmative answer to
Question~\ref{Q:p ARV} is equivalent to the assertion that if $p\in
[1,2]$ then for every $x_1,\ldots,x_n\in \ell_p$ there exist
$y_1,\ldots,y_n\in \ell_2$ that satisfy
$$
\sum_{i=1}^n\sum_{j=1}^n \|y_i-y_j\|_2^2=\sum_{i=1}^n\sum_{j=1}^n \|x_i-x_j\|_p^2,
$$
and
$$
\forall\, i,j\in \{1,\ldots,n\},\qquad \|y_i-y_j\|_2\lesssim (\log n)^{\frac1{p}-\frac12}\|x_i-x_j\|_p.
$$
We conjecture that the answer to Question~\ref{Q:p ARV} is positive.
As partial motivation for this conjecture, we note that by an
important theorem of Arora, Rao and Vazirani~\cite{ARV09} the answer
is positive when $p=1$. A possible approach towards proving this
conjecture for $p\in (1,2)$ is to investigate whether the quantity
$\gamma(A,\|\cdot\|_X^2)$ behaves well under interpolation. In our
case one would write
$\frac{1}{p}=\frac{\theta}{2}+\frac{1-\theta}{1}$ for an appropriate
$\theta\in (0,1)$ and ask whether or not it is true that for every
$n\in \N$ every $n$ by $n$ symmetric stochastic matrix $A$ satisfies
\begin{equation}\label{eq:interpolation hope}
\gamma(A,\|\cdot\|_p^2)\lesssim \gamma(A,\|\cdot\|_2^2)^\theta\cdot
\gamma(A,\|\cdot\|_1^2)^{1-\theta}.
\end{equation}
 Investigating the possible
validity such interpolation inequalities for nonlinear spectral gaps
is interesting in its own right; in Section~\ref{sec:interpolation}
we derive a weaker interpolation inequality in this spirit.

\subsection{Applications to bi-Lipschitz
embeddings}\label{sec:embedding application} The (bi-Lipschitz)
distortion of a metric space $(X,d_X)$ in a metric space $(Y,d_Y)$,
denoted $c_Y(X)$, is define to be the infimum over those $D\in
[1,\infty)$ for which there exists $s\in (0,\infty)$ and a mapping
$f:X\to Y$ that satisfies
$$
\forall\, x,y\in X,\qquad sd_X(x,y)\le d_Y(f(x),f(y))\le Dsd_X(x,y).
$$
Set $c_Y(x)=\infty$ if no such $D$ exists. When $Y=\ell_p$ for some $p\in [1,\infty]$ we use the simpler notation $c_p(X)=c_{\ell_p}(X)$. The parameter $c_2(X)$ is known in the literature as the Euclidean distortion of $X$.

The  Fr\'echet--Kuratowski embedding  (see e.g.~\cite{Hei01}) shows
that $c_\infty(X)=1$ for every separable metric space $X$. We
therefore define $p(X)$ to be the infimum over those $p\in
[2,\infty]$ such that $c_p(X)< 10$. The choice of the number $10$
here is arbitrary, and one can equally consider any fixed number
bigger than $1$ in place of $10$ to define the parameter $p(X)$; we
made this arbitrary choice rather than adding an additional
parameter only for the sake of notational simplicity.

For $n\in \N$ and $d\in \{3,\ldots,n-1\}$ let $p(n,d)$ be the
expectation of  $p(G)$ when $G$ is distributed uniformly at random
over all connected $n$-vertex $d$-regular graphs, equipped with the
shortest-path metric.
Thus, if $p=p(n,d)$ then in expectation a
connected $n$-vertex $d$-regular graph $G$ satisfies $c_{p}(G)\le
10$. In~\cite{Mat97} Matou\v{s}ek evaluated  the asymptotic
dependence of the largest possible distortion of an $n$-point metric
space in $\ell_p$, yielding the estimate $p(n,d)\gtrsim \log_d n$,
which is an asymptotically sharp bound if $d=O(1)$. As a consequence
of our proof of Theorem~\ref{thm:p bound gamma}, it turns out that
Matou\v{s}ek's bound is not sharp as a function of the degree $d$.
Specifically, in Section~\ref{sec:p index} we prove the following
result.

\begin{proposition}\label{prop:better than matousek pnd}
For every $n\in \N$ and $d\in \{3,\ldots,n-1\}$ we have
\begin{equation}\label{eq:d small 2/3}
d\le e^{(\log n)^{2/3}}\implies p(n,d)\gtrsim \frac{\log n}{\sqrt{\log d}},
\end{equation}
and
\begin{equation}\label{eq:d big 2/3 intro}
d\ge e^{(\log n)^{2/3}}\implies p(n,d)\gtrsim \left(\log_d n\right)^2.
\end{equation}
\end{proposition}
%Additional results along these lines are presented in
%Section~\ref{sec:p index}.
The significance of the quantity $\log_d n$ appearing in
Matou\v{s}ek's bound is that up to universal constant factors it is
the typical diameter of a uniformly random connected $n$-vertex
$d$-regular graph~\cite{BF82}. Thus, in~\eqref{eq:d small 2/3} there
must be some restriction on the size of $d$ relative to $n$ since
when $d$ is at least a constant power of $n$ the typical diameter of
$G$  is $O(1)$, and therefore $c_p(G)\le c_2(G)=O(1)$.
Both~\eqref{eq:d small 2/3} and~\eqref{eq:d big 2/3 intro}  assert
that $p(n,d)$ tends to $\infty$ faster than the typical diameter of
$G$ when $n^{o(1)}=d\to \infty$ (note also  that~\eqref{eq:d big 2/3
intro} is consistent with the fact that $p(n,d)$ must become bounded
when $d$ is large enough). While we initially expected
Matou\v{s}ek's bound to be sharp, Proposition~\eqref{prop:better
than matousek pnd} indicates that the parameter $p(n,d)$, and more
generally the parameter $p(X)$ for a finite metric space $(X,d_X)$,
deserves further investigation. In particular, we make no claim that
Proposition~\eqref{prop:better than matousek pnd} is sharp.

The link between Theorem~\ref{thm:p bound gamma} and Proposition~\ref{prop:better than matousek pnd} is that our proof of Theorem~\ref{thm:p bound gamma} yields the following bound, which holds for every $n\in \N$, every $n$ by $n$ symmetric stochastic matrix $A$, and every $p\in [2,\infty)$.
\begin{equation}\label{lambda p bound in intro}
\gamma\!\left(A,\|\cdot\|_{\ell_p}^2\right)\lesssim \frac{p}{1-\lambda(A)^{2/p}},
\end{equation}
where we recall that $\lambda(A)$ was defined in~\eqref{eq:def
lambda}. Proposition~\ref{prop:better than matousek pnd} is deduced
in Section~\ref{sec:p index} from~\eqref{lambda p bound in intro}
through a classical argument of Linial, London and
Rabinovich~\cite{LLR}. The bounds appearing in Proposition~\ref{prop:better than matousek pnd} hold true with $p(n,d)$ replaced by $p(G)$ when $G$ is an $n$-vertex $d$-regular Ramanujan graph~\cite{LPS,Mar88}, and it
is an independently interesting open question to evaluate $c_p(G)$
up to universal constant factors when $G$ is Ramanujan. Some
estimates in this direction are obtained in Section~\ref{sec:p
index}, where a similar question is also studied for Abelian
Alon--Roichman graphs~\cite{AR94}.

\subsection{Ozawa's localization argument for Poincar\'e inequalities} Theorem~\ref{thm:ozawa-ineq} below provides a partial answer to
Question~\ref{Q:meta} when $X$ and $Y$ are certain Banach spaces.
Its proof builds on an elegant idea of Ozawa~\cite{Ozawa} that was
used in~\cite{Ozawa} to rule out coarse embeddings of expanders into
certain Banach spaces.  While Ozawa's original argument did not
yield a nonlinear spectral gap inequality in the form that we
require here, it can be modified so as to yield
Theorem~\ref{thm:ozawa-ineq}; this is carried out in
Section~\ref{sec:ozawa}.

Throughout this article the unit ball of a Banach space
$(X,\|\cdot\|_X)$ is denoted by $$ B_X\eqdef \{x\in X:\ \|x\|_X\le
1\}.$$

\begin{theorem}\label{thm:ozawa-ineq}
Let $(X,\|\cdot\|_X)$ and $(Y,\|\cdot\|_Y)$ be Banach spaces. Suppose that $\alpha,\beta:[0,\infty)\to [0,\infty)$ are increasing functions, $\beta$ is concave, and $$\lim_{t\to 0} \alpha(t)=\lim_{t\to 0}\beta(t)=0.$$ Suppose also that there exists a mapping $\phi:B_X\to Y$ that satisfies
\begin{equation}\label{eq;alpha beta}
\forall\, x,y\in B_X,\quad \alpha\left(\|x-y\|_X\right)\le \|\f(x)-\f(y)\|_Y\le \beta\left(\|x-y\|_X\right).
\end{equation}
Then for every $q\in [1,\infty)$, every $n\in \N$, and every
symmetric stochastic matrix $A\in M_n(\R)$ we have
\begin{equation}\label{eq:ozawa ineq}
\gamma\!\left(A,\|\cdot\|_X^q\right)\le 8^{q+1}\gamma\!\left(A,\|\cdot\|_Y^q\right)+\frac{8^q}{\beta^{-1}\left(\frac{\alpha(1/4)}{8\gamma\!\left(A,\|\cdot\|_Y^q\right)^{1/q}}\right)^q}.
\end{equation}
\end{theorem}
The key point of Theorem~\ref{thm:ozawa-ineq} is that one can use local information such as~\eqref{eq;alpha beta} in order to deduce a Poincar\'e-type inequality such as~\eqref{eq:def gamma}.

Certain classes of Banach spaces $X$ are known to satisfy the
assumptions of  Theorem~\ref{thm:ozawa-ineq} when $Y$ is a Hilbert
space. These include: $L_p(\mu)$ spaces for $p\in [1,\infty)$, as
shown by Mazur~\cite{Maz29} (see~\cite[Ch.~9, Sec.~1]{BL}); Banach
spaces of finite cotype with an unconditional basis, as shown by
Odell and Schlumprecht~\cite{OS94}; more generally, Banach lattices
of finite cotype, as shown by Chaatit~\cite{Cha94} (see~\cite[Ch.~9,
Sec.~2]{BL}); Schatten classes of finite cotype (and more general
noncommutative $L_p$ spaces), as shown by Raynaud~\cite{Ray02}. For
these classes of Banach spaces Theorem~\ref{thm:ozawa-ineq}
furnishes a positive answer to Question~\ref{Q:l2 case} with $\Psi$
given by the right hand side of~\eqref{eq:ozawa ineq}.

Mazur proved~\cite{Maz29} that for every $p\in [1,\infty)$, if $X=\ell_p$ and $Y=\ell_2$ then  there
exists a mapping $\f:X\to Y$ that
satisfies~\eqref{eq;alpha beta} with $\alpha(t)=2(t/2)^{p/2}$ and
$\beta(t)=pt$ if $p\in [2,\infty)$ and $\alpha(t)=t/3$ and
$\beta(t)=2t^{p/2}$ if $p\in [1,2]$ (these estimates are recalled in
Section~\ref{sec:duality}). Therefore, it follows from
Theorem~\ref{thm:ozawa-ineq} that for every $p\in [1,\infty)$, every
$n\in \N$ and every symmetric stochastic matrix $A\in M_n(\R)$ we
have
\begin{equation}\label{eq:p<2}
1\le p\le 2\implies \gamma\!\left(A,\|\cdot\|_{\ell_p}^2\right)\lesssim \frac{1}{(1-\lambda_2(A))^{2/p}},
\end{equation}
and
\begin{equation}\label{eq:p>2}
2\le p<\infty\implies \gamma\!\left(A,\|\cdot\|_{\ell_p}^2\right)\lesssim \frac{p^28^p}{1-\lambda_2(A)}.
\end{equation}

As explained in Section~\ref{sec:BMW intro}, both~\eqref{eq:p<2}
and~\eqref{eq:p>2} are sharp in terms of the asymptotic dependence
on $1-\lambda_2(A)$, but in terms of the dependence on $p$
 the bound~\eqref{eq:p>2} is exponentially worse than the (sharp)
bound~\eqref{eq:p^2 therem}. In Section~\ref{sec:no ozawa} we show
that this exponential loss is inherent to Ozawa's method, in the
sense that for $p\in (2,\infty)$, if $\f:\ell_p\to \ell_2$
satisfies~\eqref{eq;alpha beta} with $\beta(t)=Kt$ for some $K\in
(0,\infty)$ and every $t\in (0,\infty)$ (this is to ensure
that~\eqref{eq:ozawa ineq} yields an upper bound on
$\gamma(A,\|\cdot\|_{\ell_p}^2)$ that grows linearly with
$(1-\lambda_2(A))^{-1}$), then necessarily $K/\alpha(1/4)\gtrsim
2^{3p/2}$.
%This lower bound is a special case of various
%restrictions on the moduli $\alpha$ and $\beta$ of~\eqref{eq;alpha
%beta} that we obtain in Section~\ref{sec:no ozawa}, completing as a
%byproduct the works of Enflo~\cite{Enf-smirnov},
%Raynaud~\cite{Ray83} and L\"ovblom~\cite{Lov88}.

Note that~\eqref{eq:p>2} suffices via duality (i.e.,
Corollary~\ref{coro:linear gap}) to obtain an embedding result for
arbitrary subsets of $\ell_p$ as in Corollary~\ref{coro:p av embed},
with exponentially worse dependence on $p$. Yet another proof of
such an embedding statement appears in Section~\ref{sec:average},
though it also yields a bound in terms of $p$ that is exponentially
worse than Corollary~\ref{coro:p av embed}. We do not know how to
prove the sharp statement of Corollary~\ref{coro:p av embed} other
than through Theorem~\ref{thm:p bound gamma}, whose proof is not as
elementary as the above mentioned proofs that yield an exponential
dependence on $p$.

\subsection{Average distortion embeddings and nonlinear type}\label{sec:MP}
For $p,q\in (0,\infty)$ the {\em $(p,q)$-average distortion} of
$(X,d_X)$ into $(Y,d_Y)$, denoted $Av^{(p,q)}_Y(X)\in [1,\infty)$,
is the infimum over those $D\in [1,\infty]$ such that for every
$n\in \N$ and every $x_1,\ldots,x_n\in X$ there exists a nonconstant
Lipschitz function $f:\{x_1,\ldots,x_n\}\to Y$ that satisfies
$$
\left(\frac{1}{n^2}\sum_{i=1}^n\sum_{j=1}^n d_Y(f(x_i),f(x_j))^p\right)^{\frac{1}{p}}\ge
\frac{\|f\|_{\Lip}}{D}\left(\frac{1}{n^2}\sum_{i=1}^n\sum_{j=1}^n d_X(x_i,x_j)^q\right)^{\frac{1}{q}}.
$$
When $p=q$ we use the simpler notation $Av^{(p)}_Y(X)\eqdef
Av^{(p,p)}_Y(X)$.

The notion of average distortion and its relevance to approximation
algorithms was brought to the fore in the influential
work~\cite{Rab08} of Rabinovich. Parts of Section~\ref{sec:average}
below are inspired by Rabinovich's ideas in~\cite{Rab08}. Earlier
applications of this notion include the work of Alon, Boppana and
Spencer~\cite{ABS98} that related average distortion to
asymptotically sharp isoperimetric theorems on products spaces; see
Remark~\ref{rrem:spread constant} below. In the
linear theory of Banach spaces average distortion embeddings have
been studied in several contexts; see e.g. the work on random sign
embeddings in~\cite{Elt83,FJS88}.

With the above terminology, Theorem~\ref{thm:duality} asserts that
the linear dependence~\eqref{eq:linear bound} holds true if and only
if for every finite subset $S\subseteq X$ there exists $m\in \N$
such that
$$
Av_{\ell_p^m(Y)}^{(p)}(S,d_X)\le K^{1/p}.
$$
By Corollary~\ref{coro:p av embed}, if $p\in [2,\infty)$ then
\begin{equation}\label{eq:pAv}
 Av_{\ell_2}^{(2)}(\ell_p)\lesssim p.
\end{equation}
As stated earlier, the estimate~\eqref{eq:pAv} cannot be improved
(up to the implicit constant factor): this is a special case of the
following lemma.

\begin{lemma}\label{lem:p sharp} For every $p,q,r,s\in [1,\infty)$ with $2\le q\le p$ and every $n\in \N$ there
exist $x_1,\ldots,x_n\in \ell_p$ such that if $y_1,\ldots,y_n\in
\ell_q$ satisfy
\begin{equation}\label{eq:expander rs norm}
\left(\frac{1}{n^2}\sum_{i=1}^n\sum_{j=1}^n\|y_i-y_j\|_{\ell_q}^r\right)^{\frac{1}{r}}=
\left(\frac{1}{n^2}\sum_{i=1}^n\sum_{j=1}^n\|x_i-x_j\|_{\ell_p}^s\right)^{\frac{1}{s}},
\end{equation}
then there exist $i,j\in \n$ such that
$$
\|y_i-y_j\|_{\ell_q}\gtrsim \frac{p}{q+r}\cdot \|x_i-x_j\|_{\ell_p}>0.
$$
\end{lemma}
The proof of Lemma~\ref{lem:p sharp} is given in Section~\ref{sec:interpolation}. It suffices to say at this juncture that the points
$x_1,\ldots, x_n\in \ell_p$ are the images of the vertices of a
bounded degree expanding $n$-vertex regular graph under Matou\v{s}ek's $\ell_p$-variant~\cite{Mat97}
of Bourgain's embedding~\cite{Bourgain-embed}.

We also stated earlier that Corollary~\ref{coro:p av embed}  fails
for $p\in [1,2)$: this is a special case of the following lemma.
\begin{lemma}\label{lem:no cube}
Fix $p,q,r,s\in [1,\infty)$ with $p\in [1,2)$ and $q\in (p,\infty)$.
For arbitrarily large $n\in \N$ there exist $x_1,\ldots,x_n\in
\ell_p$ such that for every $y_1,\ldots,y_n\in \ell_q$ with
$$
\left(\frac{1}{n^2}\sum_{i=1}^n\sum_{j=1}^n\|y_i-y_j\|_{\ell_q}^r\right)^{\frac{1}{r}}=
\left(\frac{1}{n^2}\sum_{i=1}^n\sum_{j=1}^n\|x_i-x_j\|_{\ell_p}^s\right)^{\frac{1}{s}},
$$
there exist $i,j\in \{1,\ldots, n\}$ such that
\begin{equation}\label{eq:p+r}
 \|y_i-y_j\|_{\ell_q}\gtrsim \frac{(\log n)^{\frac{1}{p}-\frac{1}{\min\{q,2\}}}}{\sqrt{q+r}}\cdot \|x_i-x_j\|_{\ell_p}>0.
\end{equation}
\end{lemma}

\subsubsection{Bourgain--Milman--Wolfson type}\label{sec:BMW intro} The reason for the validity of the lower bound~\eqref{eq:p+r} is
best explained in the context of nonlinear type: a metric
invariant that furnishes an obstruction to the existence of average
distortion embeddings. Let $\F_2$ be the field of cardinality $2$
and for $n\in \N$ let $e_1,\ldots,e_n$ be the standard basis of
$\F_2^n$. We also write $e=e_1+\ldots+e_n$. Following Bourgain,
Milman and Wolfson~\cite{BMW}, given $p,T\in (0,\infty)$, a metric
space $(X,d_X)$ is said to have BMW type $p$ with constant $T$ if
for every $n\in \N$ every $f:\F_2^n\to X$ satisfies
\begin{equation}\label{eq:def BMW}
\sum_{x\in \F_2^n} d_Y(f(x),f(x+e))^2\le T^2n^{\frac{2}{p}-1}\sum_{i=1}^n\sum_{x\in\F_2^n}
d_Y\left(f(x),f(x+e_i)\right)^2.
\end{equation}
$(X,d_X)$ has BMW type $p$ if it has BMW type $p$ with constant $T$
for some $T\in (0,\infty)$; in this case the infimum over those
$T\in (0,\infty)$ for which~\eqref{eq:def BMW} holds true is denoted
$\BMW_p(X)$. For background on this notion we refer to~\cite{BMW},
as well as~\cite{Pisier-type,NS02,Nao12}. These references also
contain a description of the closely related important notion of Enflo
type~\cite{Enf76}, a notion whose definition is recalled below but
will not be further investigated here.

The simple proof of the following lemma appears in
Section~\ref{sec:type}.
\begin{lemma}\label{lem:av type}
Fix $p\in (0,\infty)$. For every two metric spaces $(X,d_X)$ and
$(Y,d_Y)$ we have
$$
\BMW_p(X)\le  2Av_Y^{(2)}(X)\cdot \BMW_p(Y).
$$
\end{lemma}
The case $r=s=2$ of Lemma~\ref{lem:no cube} follows  from
Lemma~\ref{lem:av type} and the computations of BMW type that
appear in the literature. The remaining cases of Lemma~\ref{lem:no
cube} are proved using similar ideas.

 The Hamming cube is the Cayely graph on $\F_2^n$ corresponding to
the set of generators $\{e_1,\ldots,e_n\}$. The shortest path metric
on this graph coincides with the $\ell_1^n$ metric under the
identification of $\F_2^n$ with $\{0,1\}^n\subset \R^n$. Let $H_n$
be the $2^n$ by $2^n$ symmetric stochastic matrix which is the
normalized adjacency matrix of the Hamming cube. Thus for $x,y\in
\F_2^n$ the $(x,y)$-entry of $H_n$ equals $0$ unless $x-y\in
\{e_1,\ldots,e_n\}$, in which case it equals  $1/n$. It is well
known (and easy to check) that $\lambda_2(H_n)=1-2/n$, so
$\gamma(H_n,d_\R^2)\asymp n$. A simple argument (that is explained
in Section~\ref{sec:type}) shows that the definition~\eqref{eq:def
BMW} is equivalent to the requirement that
$\gamma(H_n,d_X^2)\lesssim_X n^{2/p}$ for every $n\in \N$. The
notion of Enflo type $p$  that was mentioned above is equivalent to
the requirement that $\gamma(H_n,d_X^p)\lesssim_X n$ for every $n\in
\N$.

For $p\in [1,2]$, by considering the identity mapping of $\F_2^n$
into $\ell_p^n$ (observe that $\|x-y\|_{p}^p=\|x-y\|_{1}$ for
every $x,y\in \F_2^n$) one sees that
$\gamma(H_n,\|\cdot\|_{p}^2)\gtrsim n^{2/p}\asymp
1/(1-\lambda_2(H_n))^{2/p}$. Thus~\eqref{eq:p<2} is sharp.

\subsubsection{Towards a nonlinear Maurey--Pisier
theorem}\label{se:MP intro} Every metric space has BMW type $1$ and
no metric space has BMW type greater than $2$; see
Remark~\ref{rem:p<2 always} below. Thus, for a metric space
$(X,d_X)$ define
\begin{equation}\label{eq:def p_X BMW}
 p_X\eqdef
\sup\left\{p\in [1,2]:\ \BMW_p(X)<\infty\right\}.
\end{equation}
Maurey and Pisier~\cite{MP-type-cotype} associate a quantity $p_X$
to every Banach space $X$, which is defined analogously
to~\eqref{eq:def p_X BMW} but with BMW type replaced by Rademacher
type. The (linear) notion of Rademacher type is recalled in
Section~\ref{sec:no average} below; at this juncture we just want to
state, for the sake of readers who are accustomed to the standard
Banach space terminology, that despite the apparent conflict of
notation between~\eqref{eq:def p_X BMW} and~\cite{MP-type-cotype}, a
beautiful theorem of Bourgain, Milman and Wolfson~\cite{BMW} asserts
that actually the two quantities coincide.

The following theorem is due to Bourgain, Milman and Wolfson.
\begin{theorem}[\cite{BMW}]\label{thm:BMW1} Suppose that $(X,d_X)$ is a metric space with $p_X=1$.
Then $c_X(\F_2^n,\|\cdot\|_1)=1$ for every $n\in \N$.
%$$
%\forall\, n\in \N,\quad c_X(\F_2^n,\|\cdot\|_1)=1.
%$$
\end{theorem}
Thus, the only possible obstruction to a metric space $(X,d_X)$
having BMW type $p$ for some $p>1$ is the presence of arbitrarily
large Hamming cubes. This is a metric analogue of a classical
theorem of Pisier~\cite{Pisier-type-1} asserting that the only
obstruction to a Banach space having nontrivial Rademacher type is
the presence of $\ell_1^n$ for every $n\in \N$.

In light of the Maurey--Pisier theorem~\cite{MP-type-cotype} for
Rademacher type, it is natural to ask if a similar result holds true
for a general metric space $X$ even when $p_X>1$: Is it true that
for every metric space $(X,d_X)$ we have $$\sup_{n\in \N}
c_X(\F_2^n,\|\cdot\|_{p_X})<\infty?$$ The answer to this question is
negative if $p_X=2$. Indeed, we have $p_\R=2$ yet
$c_\R(\F_2^n,\|\cdot\|_2)$ tends to infinity exponentially fast as
$n\to \infty$ because there is an exponentially large subset $S$ of
$\F_2^n$ with the property that $\|x-y\|_2\asymp \sqrt{n}$ for every
distinct $x,y\in S$. If, however, $p_X\in (1,2)$ then the above
question, called the {\em Maurey--Pisier problem for BMW type},
remains open.

The Maurey--Pisier problem for BMW type was posed by Bourgain,
Milman and Wolfson in~\cite{BMW}, where they obtained a partial
result about it: they gave a condition on a metric space $(X,d_X)$
that involves its BMW type as well as an additional geometric
restriction that ensures that $\sup_{n\in \N}
c_X(\F_2^n,\|\cdot\|_{p_X})<\infty$; see Section~4 of~\cite{BMW}.

In Section~\ref{sec:MP actual} we prove the following theorem.

\begin{theorem}\label{thm:our MP intro} For every metric space $(X,d_X)$ and every $d\in \N$ there exists $N=N(d,X)\in \N$ such that
$$
 c_{\ell_2^N(X)}\left(\F_2^d,\|\cdot\|_{p_X}\right)\le \mathrm{BMW}_{p_X}(X)^2.
$$
\end{theorem}
Thus, if $\BMW_{p_X}(X)<\infty$, i.e., the supremum defining $p_X$
in~\eqref{eq:def p_X BMW} is attained, then for every $d\in \N$ one
can embed $(\F_2^d,\|\cdot\|_{p_X})$ into $\ell_2^N(X)$ for
sufficiently large $N\in \N$. Note that it follows immediately
from~\eqref{eq:def p_X BMW} that $\BMW_p(\ell_2^N(X))=\BMW_p(X)$ for
every $N\in \N$, so Theorem~\ref{thm:our MP intro} is a complete
metric characterization of the parameter $p_X$ when the supremum defining $p_X$
in~\eqref{eq:def p_X BMW} is attained. Note also that by passing to
$\ell_2^N(X)$ we overcome the issue that was described above if
$p_X=2$, since trivially $(\F^n_2,\|\cdot\|_2)$ is isometric to a
subset of $\ell_2^n(\R)$. This indicates why Theorem~\ref{thm:our MP
intro} is easier than the actual Maurey--Pisier problem for BMW
type,  whose positive solution would require using the assumption
$p_X<2$. We therefore ask whether or not it is true that for every metric
space $(X,d_X)$ and every $p\in (1,2)$,  if there is $K\in
(0,\infty)$ such that for every $d\in \N$ there exists $N\in \N$ for
which $(\F_2^d,\|\cdot\|_{p})$ embeds with distortion $K$ into
$\ell_2^N(X)$, then $\sup_{d\in \N}
c_X(\F_2^d,\|\cdot\|_{p})<\infty$? For $p=1$ the answer to this
question is positive due to Theorem~\ref{thm:BMW1}, and by virtue of
Theorem~\ref{thm:our MP intro} a positive answer to this question
would imply an affirmative solution of the Maurey--Pisier problem for BMW
type.

Recalling Lemma~\ref{lem:av type}, given the relation between BMW
type and average distortion embeddings, it is natural to study the
following weaker version of the Maurey--Pisier problem for BMW type:
Is it true that for every metric space $(X,d_X)$ we have
\begin{equation}\label{eq:Av Maurey Pisier}
p_X<2\implies \sup_{n\in \N} Av^{(2)}_X(\F_2^n,\|\cdot\|_{p_X})<\infty?
\end{equation}
In~\eqref{eq:Av Maurey Pisier} we restrict to $p_X<2$ because one
can show that
\begin{equation}\label{eq:av cube fourth root}
Av_\R^{(2)}(\F_2^n,\|\cdot\|_2)\asymp \sqrt[4]{n}.
\end{equation}
See Remark~\ref{rem:sqrt[4]n} below for the proof of~\eqref{eq:av cube
fourth root}.

By Theorem~\ref{thm:BMW1} and Lemma~\ref{lem:av type}, for every
metric space $(X,d_X)$,
\begin{equation}\label{eq:average imples bilip}
\sup_{n\in \N} Av^{(2)}_X\left(\F_2^n,\|\cdot\|_1\right)
<\infty\implies \sup_{n\in \N} c_X\left(\F_2^n,\|\cdot\|_1\right)=1.
\end{equation}
Given $p\in (1,\infty)$, it is therefore natural to ask whether or
not for every metric space $(X,d_X)$ we have
\begin{equation}\label{eq:average imples bilip}
\sup_{n\in \N} Av^{(2)}_X\left(\F_2^n,\|\cdot\|_p\right)
<\infty\implies \sup_{n\in \N} c_X\left(\F_2^n,\|\cdot\|_p\right)<\infty?
\end{equation}
If~\eqref{eq:average imples bilip} were true for every $p\in (1,2)$ then a positive answer to the question that appears in~\eqref{eq:Av Maurey Pisier} would imply a positive solution to the Maurey--Pisier problem for BMW type.

More generally, in light of the availability of results such as~\eqref{eq:average imples bilip}, it would be of interest to relate average distortion embeddings to bi-Lipschitz embeddings. For example, is it true that if a metric space $(X,d_X)$ satisfies $Av_{\ell_2}^{(2)}(X)<\infty$ then for every finite subset $S\subset X$ we have $c_2(S)=o_X(\log |S|)$? If the answer to this question is positive then Corollary~\ref{coro:p av embed} would imply that for $p\in (2,\infty)$ any $n$-point subset of $\ell_p$ embeds into Hilbert space with distortion $o_p(\log n)$. No such improvement over Bourgain's embedding theorem~\cite{Bourgain-embed} is known for finite subsets of $\ell_p$ if $p\in (2,\infty)$; for $p\in [1,2)$ see~\cite{CGR08,ALN08}.

\bigskip

\noindent{\bf Acknowledgements.} I thank Noga Alon, Manor Mendel, and especially Yuval Rabani for very helpful discussions.  I am also grateful for the hospitality of Universit\'e Pierre et
Marie Curie, Paris, France, where part of this work was completed.

\section{Absolute spectral gaps}  Fix $n\in \N$ and an $n$ by $n$ symmetric stochastic matrix $A=(a_{ij})$. Following~\cite{MN-towards}, for $p\in [1,\infty)$ and a metric space $(X,d_X)$, denote by $\gamma_+(A,d_X^p)$ the infimum over those $\gamma_+\in (0,\infty]$ for which every $x_1,\ldots,x_n,y_1,\ldots,y_n\in X$ satisfy
 \begin{equation}\label{eq:def gamma+}
 \frac{1}{n^2} \sum_{i=1}^n\sum_{j=1}^n d_X(x_i,y_j)^p\le \frac{\gamma_+}{n}\sum_{i=1}^n\sum_{j=1}^n a_{ij} d_X(x_i,y_j)^p.
 \end{equation}
Note that by definition $\gamma(A,d_X^p)\le \gamma_+(A,d_X^p)$.
Recalling~\eqref{eq:def lambda}, we have
 $$
  \gamma_+(A,d_\R^2)=\frac{1}{1-\lambda(A)}.
 $$
 For this reason one thinks of $\gamma_+(A,d_X^p)$ as measuring the magnitude
 of the nonlinear {\em absolute} spectral gap of the matrix $A$ with respect
 to the kernel $d_X^p:X\times X\to [0,\infty)$.

 The parameter $\gamma_+(A,d_X^p)$ is useful in various contexts (see~\cite{MN-towards}), and in particular it will be used in some of the ensuing arguments. It is natural to ask for the analogue of Question~\ref{Q:meta} with  $\gamma(\cdot,\cdot)$ replaced by $\gamma_+(\cdot,\cdot)$. However, it turns out that this question is essentially the same as Question~\ref{Q:meta}, as explained in Proposition~\ref{prop:gamma gamma+ the same} below.

\begin{proposition}\label{prop:gamma gamma+ the same} Fix $p\in [1,\infty)$ and  metric spaces $(X,d_X), (Y,d_Y)$. Suppose that there exists an increasing function $\Psi:(0,\infty)\to (0,\infty)$ such that for every $n\in \N$ and every $n$ by $n$ symmetric stochastic matrix $A$ we have
\begin{equation}\label{eq:assum Psi gamma}
\gamma\!\left(A,d_X^p\right)\le \Psi\left(\gamma\!\left(A,d_Y^p\right)\right).
\end{equation}
Then for every $n\in \N$ and every $n$ by $n$ symmetric stochastic matrix $A$ we also have
\begin{equation}\label{eq:deduce Psi gamma+}
\gamma_+\!\left(A,d_X^p\right)\le 2\Psi\left(2^p\gamma_+\!\left(A,d_Y^p\right)\right).
\end{equation}

Conversely, suppose that $\Phi:(0,\infty)\to (0,\infty)$ is an increasing function  such that for every $n\in \N$ and every $n$ by $n$ symmetric stochastic matrix $A$ we have
\begin{equation}\label{eq:assume Phi gamma+}
\gamma_+\!\left(A,d_X^p\right)\le \Phi\left(\gamma_+\!\left(A,d_Y^p\right)\right).
\end{equation}
Then for every $n\in \N$ and every $n$ by $n$ symmetric stochastic matrix $A$ we also have
\begin{equation}\label{eq:deduce Phi gamma}
\gamma\!\left(A,d_X^p\right)\le \frac12 \Phi\left(2^{2p+1}\gamma\!\left(A,d_Y^p\right)\right).
\end{equation}
\end{proposition}

Before passing to the (simple) proof of Proposition~\ref{prop:gamma gamma+ the same}, we record for future use some basic facts about nonlinear spectral gaps.

\begin{lemma}\label{lem:lower a priori}
Fix $p\in [1,\infty)$, $n\in \N$ and a metric space $(X,d_X)$. Then every symmetric stochastic matrix $A=(a_{ij})\in M_n(\R)$ satisfies
\begin{equation}\label{eq:lower bounds gammas}
\gamma\!\left(A,d_X^p\right)\ge 1-\frac{1}{n}\qquad\mathrm{and}\qquad \gamma_+(A,d_X^p)\ge 1.
\end{equation}
\end{lemma}
\begin{proof}
Fix distinct $a,b\in X$ and let $x_1,\ldots,x_n\in X$ be i.i.d. points, each of which is chosen uniformly at random from $\{a,b\}$.  Then every symmetric stochastic $A\in M_n(\R)$ satisfies
$$
\E\left[\frac{1}{n}\sum_{i=1}^n\sum_{j=1}^n a_{ij} d_X(x_i,x_j)^p\right]=\frac{d_X(a,b)^p}{2n}
\sum_{\substack{i,j\in \{1,\ldots,n\}\\i\neq j}} a_{ij}\le \frac{d_X(a,b)^p}{2},
$$
while
$$
\E\left[\frac{1}{n^2}\sum_{i=1}^n\sum_{j=1}^n d_X(x_i,x_j)^p\right]=\frac{n(n-1)}{2n^2}d_X(a,b)^p.
$$
the desired conclusion now follows from the definition~\eqref{eq:def gamma}. The rightmost inequality in~\eqref{eq:lower bounds gammas} follows by substituting $x_1=\ldots =x_n=a$ and $y_1=\ldots =y_n=b$ into~\eqref{eq:def gamma+}.
\end{proof}

\begin{comment}
\begin{lemma}\label{lem:gamma at least 1/2}
Fix $p\in [1,\infty)$ and a metric space $(X,d_X)$. Then for every $n\in \N$, every $n$ by $n$ symmetric stochastic matrix $A=(a_{ij})$ satisfies $\gamma(A,d_X^p)\ge \frac12$.
\end{lemma}

\begin{proof}
Fix distinct points $a,b\in X$. By considering  $x_1,\ldots,x_n\in \{a,b\}$ in~\eqref{eq:def gamma}, we see that
\begin{align}\label{eq:use cheeger}
\nonumber\gamma(A,d_X^p)&\ge \max_{\emptyset\neq S\subset \n}\frac{|S|(n-|S|)}{n\sum_{(i,j)\in S\times (\n\setminus S)}a_{ij}}\\&\ge \frac{1}{\sqrt{2(1-\lambda_2(A))}}\ge \frac12,
\end{align}
where the penultimate inequality in~\eqref{eq:use cheeger} is Cheeger's inequality~\cite{Che70,JS88,LS88}, and the final inequality in~\eqref{eq:use cheeger} uses the fact that, since $A$ is symmetric and stochastic, we have $\lambda_2(A)\ge -1$.
\end{proof}
\end{comment}

\begin{lemma}\label{lem:I+A}
Fix $p\in [1,\infty)$ and a metric space $(X,d_X)$. Then for every integer $n\ge 2$, every $n$ by $n$ symmetric stochastic matrix $A=(a_{ij})$ satisfies
\begin{equation}\label{eq:relation gamma gamma+ I+A}
2\gamma(A,d_X^p)\le \gamma_+\!\left(\frac{I+A}{2},d_X^p\right)\le 2^{2p+1}\gamma(A,d_X^p).
\end{equation}
\end{lemma}

\begin{proof}
Since the diagonal entries of $A$ play no role in the
definition of $\gamma(A,d_X^2)$, it
follows immediately from~\eqref{eq:def gamma} that
\begin{equation}\label{eq:average identity easy gamma}
\gamma\!\left(\frac{I+A}{2},d_X^p\right)=2\gamma(A,d_X^p).
\end{equation}
Because $\gamma_+\!\left(\frac{I+A}{2},d_X^p\right)\ge \gamma\!\left(\frac{I+A}{2},d_X^p\right)$, this implies the leftmost inequality in~\eqref{eq:relation gamma gamma+ I+A}.

Next, fix $x_1,\ldots,x_n,y_1,\ldots,y_n\in X$. By the triangle
inequality and the convexity of $t\mapsto |t|^p$, for every $i,j\in
\n$ we have
\begin{equation}\label{eq:xiyj option 1}
d_X(x_i,y_j)^p\le 2^{p-1}\left( d_X(x_i,x_j)^p+d_X(x_j,y_j)^p\right),
\end{equation}
and
\begin{equation}\label{eq:xiyj option 2}
d_X(x_i,y_j)^p\le 2^{p-1}\left( d_X(x_i,y_i)^p+d_X(y_i,y_j)^p\right).
\end{equation}
By averaging~\eqref{eq:xiyj option 1} and~\eqref{eq:xiyj option 2}, and then averaging the resulting inequality over $i,j\in \n$, we see that
\begin{align}\label{eq:first averaged triangle}
&\nonumber\frac{1}{n^2}\sum_{i=1}^n
\sum_{j=1}^nd_X(x_i,y_j)^p\\&\le \frac{2^{p}}{2n}
\sum_{i=1}^n d_X(x_i,y_i)^p+\frac{2^{p}}{4n^2}\sum_{i=1}^n\sum_{j=1}^n\left(d_X(x_i,x_j)^p+d_X(y_i,y_j)^p\right).
\end{align}
Now, by the definition of $\gamma(A,d_X^p)$ we have
\begin{multline}\label{eq:use gamma xi and yi}
\frac{1}{n^2}\sum_{i=1}^n\sum_{j=1}^n\left(d_X(x_i,x_j)^p+d_X(y_i,y_j)^p\right)\\\le
\frac{\gamma(A,d_X^p)}{n}\sum_{i=1}^n\sum_{j=1}^na_{ij}\left(d_X(x_i,x_j)^p+d_X(y_i,y_j)^p\right).
\end{multline}
Next, for every $i,j\in \n$ we have
\begin{equation}\label{eq:xixj option 1}
d_X(x_i,x_j)^p\le 2^{p-1}\left( d_X(x_i,y_j)^p+d_X(y_j,x_j)^p\right),
\end{equation}
\begin{equation}\label{eq:xixj option 2}
d_X(x_i,x_j)^p\le 2^{p-1}\left( d_X(x_i,y_i)^p+d_X(y_i,x_j)^p\right),
\end{equation}
\begin{equation}\label{eq:yiyj option 1}
d_X(y_i,y_j)^p\le 2^{p-1}\left( d_X(y_i,x_j)^p+d_X(x_j,y_j)^p\right),
\end{equation}
and
\begin{equation}\label{eq:yiyj option 2}
d_X(y_i,y_j)^p\le 2^{p-1}\left( d_X(y_i,x_i)^p+d_X(x_i,y_j)^p\right).
\end{equation}
By averaging~\eqref{eq:xixj option 1}, \eqref{eq:xixj option 2}, \eqref{eq:yiyj option 1} and~\eqref{eq:yiyj option 2} we see that
\begin{multline}\label{eq:to multiply aij}
d_X(x_i,x_j)^p+d_X(y_i,y_j)^p\\\le 2^{p-1}\left(d_X(x_i,y_j)^p+d_X(y_i,x_j)^p+d_X(x_i,y_i)^p+d_X(x_j,y_j)^p\right).
\end{multline}
By multiplying~\eqref{eq:to multiply aij} by $a_{ij}/n$ and summing over $i,j\in \n$ while using the fact that $A$ is symmetric and stochastic we conclude that
\begin{multline}\label{eq:second averaged triangle}
\frac{1}{n}\sum_{i=1}^n\sum_{j=1}^na_{ij}\left(d_X(x_i,x_j)^p+d_X(y_i,y_j)^p\right)\\
\le \frac{2^{p}}{n}\sum_{i=1}^n d_X(x_i,y_i)^p+\frac{2^{p}}{n}\sum_{i=1}^n\sum_{j=1}^n a_{ij}d_X(x_i,y_j)^p.
\end{multline}
A substitution of~\eqref{eq:second averaged triangle} into~\eqref{eq:use gamma xi and yi}, and then a substitution of the resulting inequality into~\eqref{eq:first averaged triangle}, yields the following estimate.
\begin{align}\label{eq:for gamma at least 1/2}
\nonumber &\frac{1}{n^2}\sum_{i=1}^n\sum_{j=1}^nd_X(x_i,y_j)^p\\\nonumber&\le \frac{2^{p-1}+2^{2p-2}\gamma(A,d_X^p)}{n}\sum_{i=1}^n d_X(x_i,y_i)^p\\\nonumber&\qquad +\frac{2^{2p-2}\gamma(A,d_X^p)}{n}\sum_{i=1}^n\sum_{j=1}^n a_{ij} d_X(x_i,y_j)^p\\
&\le \frac{2^p+2^{2p-1}\gamma(A,d_X^p)}{n}\sum_{i=1}^n\sum_{j=1}^n \left(\frac{I+A}{2}\right)_{ij} d_X(x_i,y_j)^p.
\end{align}
Since~\eqref{eq:for gamma at least 1/2} holds true for every $x_1,\ldots,x_n,y_1,\ldots,y_n\in X$, we deduce from the definition of $\gamma_+(\cdot,\cdot)$ in~\eqref{eq:def gamma+} that
\begin{equation}\label{eq:2p+1}
\gamma_+\!\left(\frac{I+A}{2},d_X^p\right)\le 2^p+2^{2p-1}\gamma(A,d_X^p).
\end{equation}
The rightmost inequality in~\eqref{eq:relation gamma gamma+ I+A} is a consequence of~\eqref{eq:2p+1} since by Lemma~\ref{lem:lower a priori} we have $\gamma(A,d_X^p)\ge \frac12$.
\end{proof}

\begin{proof}[Proof of Proposition~\ref{prop:gamma gamma+ the same}]
Fix $n\in \N$ and an $n$ by $n$ symmetric stochastic matrix $A$. Assuming the validity of~\eqref{eq:assume Phi gamma+} for the matrix $(I+A)/2$, we deduce~\eqref{eq:deduce Phi gamma}. Next, by~\cite[Lem.~2.3]{MN-towards} for every metric space $(W,d_W)$ we have
\begin{equation}\label{eq:anti diagonal}
\frac{2}{2^p+1}\gamma\!\left(\left ( \begin{smallmatrix}
  0 & A \\
   A & 0
   \end{smallmatrix} \right ),d_W^p\right)\le \gamma_+(A,d_W^p)\le 2\gamma\!\left(\left ( \begin{smallmatrix}
  0 & A \\
   A & 0
   \end{smallmatrix} \right ),d_W^p\right).
\end{equation}
Therefore, assuming the validity of~\eqref{eq:assum Psi gamma} for the matrix $\left ( \begin{smallmatrix}
  0 & A \\
   A & 0
   \end{smallmatrix} \right )\in M_{2n}(\R)$, the desired estimate~\eqref{eq:deduce Psi gamma+} follows from~\eqref{eq:anti diagonal}.
\end{proof}

\section{Duality}\label{sec:duality}

The implication $(2)\implies (1)$ of Theorem~\ref{thm:duality} was
already proved in~\eqref{eq:duality trivial dir}. Below we prove the
more substantial implication $(1)\implies (2)$.

\begin{proof}[Proof of Theorem~\ref{thm:duality}] Fix $D\in (K,\infty)$ and define $\e\in (0,\infty)$ by
\begin{equation*}\label{eq:def epsilon}(1+\e)(K+2\e)=D.\end{equation*}

Let $C\subseteq M_n(\R)$ be the set of $n$ by $n$ symmetric matrices $C=(c_{ij})$ for which there exists $y_1,\ldots,y_n\in Y$ with $|\{y_1,\ldots,y_n\}|\ge 2$ and
$$
\forall\, i,j\in \{1,\ldots,n\},\qquad c_{ij}=\frac{\sum_{r=1}^n\sum_{s=1}^n d_X(x_r,x_s)^p}{\sum_{r=1}^n\sum_{s=1}^n d_Y(y_r,y_s)^p}\cdot d_Y(y_i,y_j)^p.
$$
Letting $P\subseteq M_n(\R)$ be the set of all symmetric $n$ by $n$ matrices with nonnegative entries and vanishing diagonal, denote
$$
M\eqdef \conv\left(C+P\right).
$$

For every $i,j\in \{1,\ldots, n\}$ write $t_{ij}\eqdef (K+2\e) d_X(x_i,x_j)^p$. We first claim that the matrix $T=(t_{ij})\in M_n(\R)$ belongs to $M$. Indeed, if this were not the case then by the separation theorem (Hahn--Banach) there would exist a symmetric matrix $H=(h_{ij})\in M_n(\R)$ which has at least one nonzero off-diagonal entry and whose diagonal vanishes, satisfying
\begin{equation}\label{eq:use duality}
\inf_{B=(b_{ij})\in M} \sum_{i=1}^n\sum_{j=1}^n h_{ij}b_{ij}\ge (K+2\e)\sum_{i=1}^n\sum_{j=1}^n h_{ij}d_X(x_i,x_j)^p.
\end{equation}

Since $P\subseteq M$, the fact that the left hand side of~\eqref{eq:use duality} is bounded from below implies that $h_{ij}\ge 0$ for all $i,j\in \{1,\ldots,n\}$.
%Note that the diagonal entries of $H$ play no role in~\eqref{eq:use duality}. Consequently,
Fixing $\d\in (0,1)$, if we define
$$
\sigma\eqdef \max_{i\in \{1,\ldots,n\}} \sum_{j\in \{1,\ldots,n\}\setminus \{i\}} (h_{ij}+\d),
$$
and for every $i,j\in \{1,\ldots,n\}$,
$$
a_{ij}\eqdef \frac{1}{2\sigma}\left\{\begin{array}{ll}2\sigma-\sum_{r\in \{1,\ldots,n\}\setminus \{i\}} (h_{ir}+\d)& \mathrm{if\ }i=j,\\
h_{ij}+\d&\mathrm{if\ } i\neq j,\end{array}\right.
$$
then, provided $\d\in (0,1)$ is small enough, $A=(a_{ij})\in M_n(\R)$ is a symmetric stochastic matrix all of whose entries are positive  such that
\begin{eqnarray}\label{eq:use duality2}
\nonumber\inf_{B=(b_{ij})\in M} \sum_{i=1}^n\sum_{j=1}^n a_{ij}b_{ij}&\stackrel{\eqref{eq:use duality}}{\ge}& (K+\e)\sum_{i=1}^n\sum_{j=1}^n a_{ij}d_X(x_i,x_j)^p \\&\stackrel{\eqref{eq:linear bound}}{\ge}&\frac{K+\e}{n\gamma(A,d_X^p)}\sum_{i=1}^n\sum_{j=1}^n d_X(x_i,x_j)^p.
\end{eqnarray}
By the definition of $C\subseteq M_n(\R)$ and $\gamma(A,d_Y^p)$ we have
\begin{align}\label{eq:recall def C}
&\nonumber\inf_{B=(b_{ij})\in C}\sum_{i=1}^n\sum_{j=1}^n a_{ij}b_{ij}\\&\nonumber
=\inf_{\substack{y_1,\ldots,y_n\in Y\\|\{y_1,\ldots,y_n\}|\ge 2}} \frac{\sum_{r=1}^n\sum_{s=1}^n d_X(x_r,x_s)^p}{\sum_{r=1}^n\sum_{s=1}^n d_Y(y_r,y_s)^p}\cdot\sum_{i=1}^n \sum_{j=1}^n a_{ij} d_Y(y_i,y_j)^p\\&=\frac{1}{n\gamma(A,d_Y^p)}\sum_{r=1}^n\sum_{s=1}^n d_X(x_r,x_s)^p.
\end{align}
Because $M\supseteq C$ it follows from~\eqref{eq:use duality2} and~\eqref{eq:recall def C} that
\begin{multline}\label{eq:furnishes}
\frac{1}{n\gamma(A,d_Y^p)}\sum_{r=1}^n\sum_{s=1}^n d_X(x_r,x_s)^p
\ge \frac{K+\e}{n\gamma(A,d_X^p)}\sum_{i=1}^n\sum_{j=1}^n
d_X(x_i,x_j)^p
\\\ge \frac{K+\e}{Kn\gamma(A,d_Y^p)}\sum_{i=1}^n\sum_{j=1}^n
d_X(x_i,x_j)^p,
\end{multline}
Since all the entries of $A$ are positive, $\gamma(A,d_Y^p)\in
(0,\infty)$, and therefore~\eqref{eq:furnishes} furnishes the
desired contradiction.

Having proved that $T\in M$, it follows that there exist $N\in \N$ and $\mu_1,\ldots,\mu_N\in (0,1]$ with $\sum_{k=1}^N\mu_k=1$, and for every $k\in \{1,\ldots,N\}$ there are $y_1^k,\ldots,y_n^k\in Y$ with $|\{y_1^k,\ldots,y_n^k\}|\ge 2$, such that for every $i,j\in \{1,\ldots,n\}$ we have
\begin{equation}\label{eq:in hull}
(K+2\e)d_X(x_i,x_j)^p\ge  \sum_{k=1}^N \mu_k\frac{\sum_{r=1}^n\sum_{s=1}^n d_X(x_r,x_s)^p}{\sum_{r=1}^n\sum_{s=1}^n d_Y(y^k_r,y^k_s)^p}\cdot d_Y(y^k_i,y^k_j)^p.
\end{equation}
There are integers $q_1,\ldots,q_N,Q\in \N$ such that
\begin{equation}\label{eq:rational}
\forall\, k\in \{1,\ldots,N\},\quad \frac{q_k}{Q}\le \mu_k\frac{\sum_{r=1}^n\sum_{s=1}^n d_X(x_r,x_s)^p}{\sum_{r=1}^n\sum_{s=1}^n d_Y(y_r^k,y_s^k)^p}\le (1+\e)\frac{q_k}{Q}.
\end{equation}
Setting $q_0\eqdef 0$ and $m\eqdef \sum_{k=1}^Nq_k$, define $f:\{x_1,\ldots,x_n\}\to \ell_p^{m}(Y)$ by
\begin{equation}\label{eq:define f}
\forall\, k\in \{1,\ldots,N\},\ \forall\, u\in \left[1+\sum_{j=0}^{k-1}q_j,\sum_{j=0}^k q_j\right]\cap \N,\quad f(x_i)_u=y_i^k.
\end{equation}
Then for every $i,j\in \{1,\ldots,n\}$,
\begin{eqnarray*}
d_{\ell_p^m(Y)}(f(x_i),f(x_j))&\stackrel{\eqref{eq:define f}}{=}&\left(\sum_{k=1}^N q_k d_Y(y_i^k,y_j^k)^p\right)^{\frac{1}{p}}\\&\stackrel{\eqref{eq:in hull}\wedge\eqref{eq:rational}}{\le}& Q^{1/p}(K+2\e)^{1/p}d_X(x_i,x_j).
\end{eqnarray*}
Consequently,
\begin{equation}\label{eq:lip const}
\|f\|_{\Lip}\le Q^{1/p}(K+2\e)^{1/p}.
\end{equation}
Hence,
\begin{align*}
&\sum_{i=1}^n\sum_{j=1}^n d_Y(f(x_i),f(x_j))^p\stackrel{\eqref{eq:define f}}{=}\sum_{i=1}^n\sum_{j=1}^n
\sum_{k=1}^N q_k d_Y(y_i^k,y_j^k)^p\\
&\stackrel{\eqref{eq:rational}}{\ge} \frac{Q}{1+\e}\sum_{i=1}^n\sum_{j=1}^n d_X(x_i,x_j)^p\stackrel{\eqref{eq:lip const}}{\ge} \frac{\|f\|_{\Lip}^p}{(1+\e)(K+2\e)}\sum_{i=1}^n\sum_{j=1}^n d_X(x_i,x_j)^p.
\end{align*}
Since $(1+\e)(K+2\e)=D$, the proof of Theorem~\ref{thm:duality} is complete.
\end{proof}

\section{Interpolation and Markov type}\label{sec:use markov type}

Fix $p\in [1,\infty)$ and $m\in \N$.  Following K. Ball~\cite{Ball},
given a metric space $(X,d_X)$ define its {\em Markov type $p$
constant at time $m$}, denoted $M_p(X;m)$, to be the infimum over
those $M\in (0,\infty)$ such that for every $n\in \N$,  every
$x_1,\ldots,x_n\in X$ and every $n$ by $n$ symmetric stochastic
matrix $A=(a_{ij})\in M_n(\R)$ we have
\begin{equation}\label{eq:def mtype m}
\sum_{i=1}^n\sum_{j=1}^n (A^m)_{ij}d_X(x_i,x_j)^p\le M^pm\sum_{i=1}^n\sum_{j=1}^n a_{ij}d_X(x_i,x_j)^p.
\end{equation}
$X$ is said to have Markov type $p$ if $M_p(X)\eqdef \sup_{m\in \N}
M_p(X;m)<\infty$. Note that it follows from the triangle inequality
that
$$
\forall\, m\in \N,\quad M_p(X;m)\le m^{1-\frac{1}{p}}.
$$

\begin{remark}\label{rem:enflo}
In Section~\ref{sec:MP} we recalled the notions of BMW type and
Enflo type. The link between these notions and Ball's notion of
Markov type is that Markov type $p$ implies Enflo type $p$
(see~\cite{NS02}). One can also define natural variants of Markov
type that imply BMW type (see the inequalities appearing in
Theorem~4.4 of~\cite{NPSS06}). Recently Kondo proved~\cite{Kondo}
that there exists a Hadamard space (see e.g.~\cite{BH-book}) that
fails to have Markov type $p$ for any $p>1$, answering a question
posed in~\cite{NPSS06}. Since Hadamard spaces have Enflo type $2$
(see~\cite{Oht09}), this yields the only known example of a metric
space that has Enflo type $2$ but fails to have nontrivial Markov
type (observe that the notions of Enflo type $2$ and BMW type $2$
coincide).
\end{remark}

\begin{lemma}\label{lem:mtype lower decay gap}
Fix $p\in [1,\infty)$ and $m,n\in \N$. Let $(X,d_X)$ be a metric space and $A=(a_{ij})\in M_n(\R)$ be a symmetric stochastic matrix. Then
$$
\gamma\!\left(A,d_X^p\right)\le M_p(X;m)^pm\gamma\!\left(A^m,d_X^p\right).
$$
\end{lemma}

\begin{proof}
This is immediate from the definitions: for every $x_1,\ldots,x_n\in X$ we have
\begin{align*}
&\frac{1}{n^2}\sum_{i=1}^n\sum_{j=1}^n d_X(x_i,x_j)^p\stackrel{\eqref{eq:def gamma}}{\le} \frac{\gamma\!\left(A^m,d_X^p\right)}{n}\sum_{i=1}^n\sum_{j=1}^n (A^m)_{ij}d_X(x_i,x_j)^p\\
&\quad\quad\stackrel{\eqref{eq:def mtype m}}{\le} \frac{M_p(X;m)^pm\gamma\!\left(A^m,d_X^p\right)}{n}\sum_{i=1}^n\sum_{j=1}^n a_{ij}d_X(x_i,x_j)^p.\qedhere
\end{align*}
\end{proof}

The {modulus of uniform smoothness} of a Banach space $(X,\|\cdot\|_X)$ is defined for
$\tau\in (0,\infty)$ as
\begin{equation*}
\rho_X(\tau)\eqdef \sup\left\{\frac{\|x+\tau y\|_X+\|x-\tau
y\|_X}{2}-1:\ x,y\in B_X\right\}.
\end{equation*}
$X$ is said to be {uniformly smooth} if $\lim_{\tau\to
0}\rho_X(\tau)/\tau=0$. Furthermore, $X$ is said to have modulus of
smoothness of power type $q\in (0,\infty)$ if there exists a
constant $C\in (0,\infty)$ such that $\rho_X(\tau)\le C\tau^q$ for
all $\tau\in (0,\infty)$. It is straightforward to check that in
this case necessarily $q\in [1,2]$. It is shown in~\cite{BCL} that
$X$ has modulus of smoothness of power type $q$ if and only if there
exists a constant $S\in [1,\infty)$ such that for every $x,y\in X$
\begin{equation}\label{eq:two point smooth}
\frac{\|x+y\|_X^q+\|x-y\|_X^q}{2}\le \|x\|_X^q+S^q\|y\|_X^q.
\end{equation}
The infimum over those $S\in [1,\infty)$ for which \eqref{eq:two point smooth}
holds true is called the $q$-smoothness constant of $X$, and is denoted
$S_q(X)$. Observe that every Banach space satisfies $S_1(X)=1$.

The following theorem is due to~\cite{NPSS06}.
\begin{theorem}\label{thm:npss}
Fix $q\in [1,2]$ and $p\in [q,\infty)$. Suppose that $(X,\|\cdot\|_X)$ is a Banach space whose modulus of smoothness has power type $q$. Then
\begin{equation}\label{eq:mtype}
\forall\, m\in \N,\qquad M_p(X;m)\lesssim \left(p^{\frac{1}{q}}+S_q(X)\right)m^{\frac{1}{q}-\frac{1}{p}}.
\end{equation}
\end{theorem}
%\begin{theorem}\label{thm:npss}
%Fix $q\in (1,2]$ and $p\in [q,\infty)$. Suppose that $(X,\|\cdot\|_X)$ is a Banach space whose modulus of smoothness has power type $q$. Then for every $m,n\in \N$, every %symmetric stochastic $A=(a_{ij})\in M_n(\R)$, and every $x_1,\ldots,x_n\in X$ we have
%\begin{equation}\label{eq:mtype}
%\sum_{i=1}^n\sum_{j=1}^n (A^m)_{ij}\|x_i-x_j\|_X^p\le(Cpm)^{\frac{p}{q}}S_q(X)^p \sum_{i=1}^n\sum_{j=1}^n a_{ij}\|x_i-x_j\|_X^p,
%\end{equation}
%where $C\in [1,\infty)$ is a universal constant.
%\end{theorem}
The statement corresponding to Theorem~\ref{thm:npss}
in~\cite{NPSS06} (specifically, see Theorem~4.4 there), allows for a
multiplicative constant with unspecified dependence on $p$ and $q$,
while in~\eqref{eq:mtype} we stated an explicit dependence on these
parameters that will serve us later on several occasions. We shall
therefore proceed to sketch the proof of Theorem~\ref{thm:npss} so
as to explain why the dependence on $p$ and $q$ in~\eqref{eq:mtype}
is indeed valid.
\begin{proof}[Proof of Theorem~\ref{thm:npss} (sketch)]
For every measure space $(\Omega,\mu)$ we have
\begin{equation}\label{eq:restate S_q L_p}
S_q\left(L_p(\mu,X)\right)\lesssim p^{\frac{1}{q}}+S_q(X).
\end{equation}
The case $q=2$ of~\eqref{eq:restate S_q L_p} appears
in~\cite{Nao12-azuma}, and the proof for general $q\in [1,2]$
follows mutatis mutandis from the proof in~\cite{Nao12-azuma}. This
has been carried out explicitly in Lemma~6.3 of~\cite{MN-towards},
whose statement asserts the weaker bound
$S_q\left(L_p(\mu,X)\right)\lesssim p^{1/q}S_q(X)$, but the proof
of~\cite[Lem.~6.3]{MN-towards} without any change whatsoever
yields~\eqref{eq:restate S_q L_p}.

As explained in the proof of Theorem~4.4 in~\cite{NPSS06}, by a result
of~\cite{Lin63} (see also~\cite[Prop.~2.2]{Pisier-martingales}) it
follows from~\eqref{eq:restate S_q L_p}  that  every $X$-valued
martingale  $\{M_k\}_{k=0}^m$ satisfies
\begin{align}\label{eq:martingale with constant}
&\nonumber\E\left[\|M_m-M_0\|_X^p\right]\\\nonumber&\le
K^p\left(p^{\frac{p}{q}}+S_q(X)^p\right)
\left(\sum_{k=1}^m \left(\E\left[\|M_k-M_{k-1}\|_X^p\right]\right)^{\frac{q}{p}}\right)^{\frac{p}{q}}\\
&\le K^p\left(p^{\frac{p}{q}}+S_q(X)^p\right)m^{\frac{p}{q}-1}
\sum_{k=1}^m\E\left[\|M_k-M_{k-1}\|_X^p\right],
\end{align}
where $K\in (0,\infty)$ is a universal constant.  Now, substituting
the martingale inequality~\eqref{eq:martingale with constant} into
the proof of Theorem~2.3 in~\cite{NPSS06}, in place of the use of
Pisier's martingale inequality~\cite{Pisier-martingales},
yields~\eqref{eq:mtype}.
\end{proof}

We record for future use the following corollary, which is an immediate consequence of Lemma~\ref{lem:mtype lower decay gap} and Theorem~\ref{thm:npss}.
\begin{corollary}\label{cor:smoothness}
Fix $q\in (1,2]$ and $p\in [q,\infty)$. Suppose that $(X,\|\cdot\|_X)$ is a Banach space whose modulus of smoothness has power type $q$. Then for every  $m,n\in \N$ and every  symmetric stochastic $A=(a_{ij})\in M_n(\R)$,
\begin{equation}\label{eq:gamma comparison smoothness}
\gamma\!\left(A,\|\cdot\|_X^p\right)\lesssim  C^p\left(p^{\frac{p}{q}}+S_q(X)^p\right) m^{\frac{p}{q}}  \gamma\!\left(A^m,\|\cdot\|_X^p\right),
\end{equation}
where $C\in (0,\infty)$ is a universal constant.
\end{corollary}

%In what follows  we let $\lambda(A)$ denote the second largest absolute value of an eigenvalue of $A$, i.e.,
%\begin{equation}\label{eq:lambda}
%\lambda(A)\eqdef \max_{i\in \{2,\ldots,n\}} |\lambda_i(A)|=\max\left\{\lambda_2(A),|\lambda_n(A)|\right\}.
%\end{equation}

We refer to~\cite{BL76} for the background on complex
interpolation that is used below. We also recall the definition of
$\lambda(A)$ in~\eqref{eq:def lambda}. The following theorem is the
main result of this section.

\begin{theorem}\label{thm:smoothness mtype} Let $(H,Z)$ be a compatible pair of Banach spaces with
$H$ a Hilbert space. Suppose that $\theta\in [0,1]$ and consider the
complex interpolation space $X=[H,Z]_\theta$. Fix $q\in [1,2]$ and
suppose that $X$ has modulus of smoothness of power type $q$. Then
for every $n\in \N$ and every $n$ by $n$ symmetric stochastic matrix
$A\in M_n(\R)$ we have
\begin{equation}\label{eq:smoothness interpolation}
\gamma\!\left(A,\|\cdot\|_X^2\right)\lesssim \frac{S_q(X)^2}{\left(1-\lambda(A)^{\theta}\right)^{2/q}}.
\end{equation}
\end{theorem}

Before proving Theorem~\ref{thm:smoothness mtype}  we present some
of its immediate corollaries. First, since in the setting of
Theorem~\ref{thm:smoothness mtype} we always have $S_q(X)\lesssim 1$
for $q=2/(1+\theta)$ (see~\cite{Pis79,CR82}), the following
corollary is a special case of Theorem~\ref{thm:smoothness mtype}.

\begin{corollary}\label{cor:use interpolation with theta smoothness}
Under the assumptions of Theorem~\ref{thm:smoothness mtype} we have
$$
\gamma\!\left(A,\|\cdot\|_X^2\right)\lesssim \frac{1}{\left(1-\lambda(A)^{\theta}\right)^{1+\theta}}.
$$
\end{corollary}

In order for the above results to fit into the framework of
Question~\ref{eq:Psi euclidean}, we need to bound
$\gamma\!\left(A,\|\cdot\|_X^2\right)$ in terms of $\lambda_2(A)$
rather than $\lambda(A)$. This is the content of the next corollary.

\begin{corollary}\label{coro:pass to lambda2}
Under the assumptions of Theorem~\ref{thm:smoothness mtype} we have
\begin{equation}\label{eq:back to lambda 2}
\gamma\!\left(A,\|\cdot\|_X^2\right)\lesssim \frac{S_q(X)^2}{\theta^{2/q}}\cdot \frac{1}{(1-\lambda_2(A))^{2/q}}.
\end{equation}
\end{corollary}

\begin{proof}
Since $A$ is symmetric and stochastic, all of its eigenvalues are in the interval $[-1,1]$. Consequently, all the eigenvalues of the symmetric stochastic matrix $(I+A)/2$ are nonnegative, and hence
\begin{equation}\label{eq:I+A/2}
\lambda\left(\frac{I+A}{2}\right)=\frac{1+\lambda_2(A)}{2}.
\end{equation}
An application of Theorem~\ref{thm:smoothness mtype} to the matrix $(I+A)/2$ while taking into account the identities~\eqref{eq:I+A/2} and~\eqref{eq:average identity easy gamma} implies that
$$
\gamma(A,\|\cdot\|_X^2)\lesssim \frac{S_q(X)^2}{\left(1-\left(\frac{1+\lambda_2(A)}{2}\right)^{\theta}\right)^{2/q}}.
$$
This yields the desired estimate~\eqref{eq:back to lambda 2} due to the elementary inequality
\begin{equation*}
\forall\, x\in [-1,1],\qquad 1-\left(\frac{1+x}{2}\right)^{\theta}\ge \frac{\theta(1-x)}{2}.\qedhere
\end{equation*}
\end{proof}

We can now complete the proof of Theorem~\ref{thm:p bound gamma},
and consequently also its corresponding dual statement
Corollary~\ref{coro:p av embed}.

\begin{proof}[Proof of Theorem~\ref{thm:p bound gamma}]
Fix $p\in [2,\infty)$ and apply~\eqref{coro:pass to lambda2} to  $X=\ell_p$. Then $X=[\ell_2,\ell_\infty]_{\theta}$ for $\theta=2/p$. Moreover,
by~\cite{Fiegel76,BCL} we have $S_2(\ell_p)=\sqrt{p-1}$, and
therefore the desired estimate~\eqref{eq:p^2 therem} follows
from~\eqref{eq:back to lambda 2} (with $q=2$).
\end{proof}

As in the above proof of Theorem~\ref{thm:p bound gamma}, by
specializing Theorem~\ref{thm:smoothness mtype} to $X=\ell_p$ for
$p\in [2,\infty)$, we obtain the following corollary, which was
stated in the Introduction as inequality~\eqref{lambda p bound in
intro}.

\begin{corollary}\label{coro:gamma p by lambda}
For every $p\in [2,\infty)$, every $n\in \N$ and every $n$ by $n$
symmetric stochastic matrix $A$ we have
$$
\gamma\!\left(A,\|\cdot\|_{\ell_p}^2\right)\lesssim \frac{p}{1-\lambda(A)^{2/p}}.
$$
\end{corollary}

We now proceed to prove Theorem~\ref{thm:smoothness mtype}.

\begin{proof}[Proof of Theorem~\ref{thm:smoothness mtype}] In what follows, given a Banach space $(Y,\|\cdot\|_Y)$ and $n\in \N$ we let $L_2^n(Y)$ denote the Banach space whose underlying vector space is $Y^n$, equipped with the norm
$$
\forall\,y=(y_1,\ldots,y_n)\in Y^n,\quad \|y\|_{L_2^n(Y)}\eqdef \left(\frac{1}{n}\sum_{i=1}^n \|y_i\|_Y^2\right)^{\frac12}.
$$
%The subspace of $L_2^n(Y)$ consisting of all $y\in Y^n$ with $\sum_{i=1}^n y_i=0$ will be denoted $L_2^n(Y)_0$.
If $H$ is a Hilbert space with scalar product $\langle \cdot,\cdot\rangle_H$, then $L_2^n(H)$ is a Hilbert space whose scalar product is always understood to be given by
$
\langle x,y\rangle_{L_2^n(H)}= \frac{1}{n}\sum_{i=1}^n \langle x_i,y_i\rangle_H
$
for every $x,y\in H$.

Let $e_1(A),\ldots,e_n(A)\in L_2^n(\R)$ be an orthonormal eigenbasis of $A$ with $e_1(A)=(1,\ldots,1)$ and $Ae_i(A)=\lambda_i(A)e_i(A)$ for all $i\in \{1,\ldots,n\}$. Define an operator $T:L_2^n(\R)\to L_2^n(\R)$ by setting for every $x\in \R^n$ and $i\in \{1,\ldots,n\}$,
$$
(Tx)_i\eqdef \sum_{j=1}^n a_{ij}\left(x_j-\frac{1}{n}\sum_{k=1}^n x_k\right).
$$
Equivalently,
\begin{equation}\label{eq:spectral T}
Tx=\sum_{i=2}^n \lambda_i(A)\langle x,e_i(A)\rangle_{L_2^n(\R)}e_i(A).
\end{equation}
The operator $T\otimes I_Y:L_2^n(Y)\to L_2^n(Y)$, where $I_Y$ denotes the identity on $Y$, is then given by
\begin{equation}\label{eq:tensor formula}
\forall\, i\in \{1,\ldots,n\},\qquad ((T\otimes I_Y)y)_i\eqdef \sum_{j=1}^n a_{ij}\left(y_j-\frac{1}{n}\sum_{k=1}^n y_k\right).
\end{equation}

Recalling~\eqref{eq:def lambda} and~\eqref{eq:spectral T}, we have
the following operator norm bounds.
\begin{equation}\label{eq:hilbertian operator norm}
\left\|T\otimes I_H\right\|_{L_2^n(H)}=\|T\|_{L_2^n(\R)\to L_2^n(\R)}=\lambda(A).
\end{equation}
The norm of $T\otimes I_Z:L_2^n(Z)\to L_2^n(Z)$ can be bounded crudely by using the fact that $A$ is a symmetric stochastic matrix. Indeed, for every $z\in L_2^n(Z)$ we have
\begin{align}
\nonumber\|Tz\|_{L_2^n(Z)}^2&=\frac{1}{n}\sum_{i=1}^n \left\|\sum_{j=1}^n a_{ij}\left(z_j-\frac{1}{n}\sum_{k=1}^n z_k\right) \right\|_Z^2 \\
&\le \frac{1}{n^2}\sum_{i=1}^n\sum_{j=1}^n\sum_{k=1}^n a_{ij}\|z_j-z_k\|_Z^2\nonumber\\
&\le \frac{2}{n^2}\sum_{i=1}^n\sum_{j=1}^n\sum_{k=1}^n a_{ij}\left(\|z_j\|_Z^2+\|z_k\|_Z^2\right)=4\|z\|_{L_2^n(Z)}^2\label{eq:norm T1}.
\end{align}

An interpolation of~\eqref{eq:hilbertian operator norm} and~\eqref{eq:norm T1} (see~\cite{BL76}) shows that
\begin{equation}\label{eq:interpolated bound}
\|T\otimes I_X\|_{L_2^n(X)\to L_2^n(X)}\le 2^{1-\theta}\lambda(A)^{\theta}.
\end{equation}
For every $x\in L_2^n(X)$ let $\overline{x}\in L_2^n(X)$ be the
vector whose $i$th coordinate equals $x_i-S(x)$, where
$S(x)=\frac{1}{n}\sum_{k=1}^n x_k$. Then
\begin{eqnarray}\label{eq:overline upper}
\nonumber\left\|(I_{L_2^n(X)}-T\otimes I_X)\overline{x}\right\|_{L_2^n(X)}^2&\stackrel{\eqref{eq:tensor formula}}{=}&\frac{1}{n}\sum_{i=1}^n\left\|\sum_{j=1}^n a_{ij}(x_i-x_j)\right\|_X^2\\&\le& \frac{1}{n}\sum_{i=1}^n\sum_{j=1}^n a_{ij}\|x_i-x_j\|_X^2.
\end{eqnarray}
At the same time,
\begin{eqnarray}\label{eq:overline lower}
\nonumber\left\|(I_{L_2^n(X)}-T\otimes I_X)\overline{x}\right\|_{L_2^n(X)}&\ge&
\|\overline{x}\|_{L_2^n(X)}-\left\|(T\otimes I_X)\overline{x}\right\|_{L_2^n(X)}\\
&\stackrel{\eqref{eq:interpolated bound}}{\ge}& \left(1-2^{1-\theta}\lambda(A)^{\theta}\right)\|\overline{x}\|_{L_2^n(X)}.
\end{eqnarray}
Hence, if we suppose that
\begin{equation}\label{eq:suppsoe}
 2^{1-\theta}\lambda(A)^{\theta}< 1 \iff \lambda(A)<\frac{1}{2^{(1-\theta)/\theta}},
\end{equation}
then
%\begin{eqnarray}
%&&\nonumber \!\!\!\!\!\!\!\!\!\!\!\!\!\!\!\!\!\!\!\!\!\!\!\!\!\!\!\!\!\!\!\frac{1}{n^2}\sum_{i=1}^n\sum_{j=1}^n \|x_i-x_j\|_X^2\\&\le& \nonumber %\frac{2}{n^2}\sum_{i=1}^n\sum_{j=1}^n \left(\left\|x_i-S(x)\right\|_X^2+\left\|x_j-S(x)\right\|_X^2\right)
%\\&=&4\|\overline{x}\|_{L_2^n(X)}^2\nonumber\\&\stackrel{\eqref{eq:overline upper}\wedge\eqref{eq:overline lower}}{\le} & %\frac{4}{\left(1-2^{\theta}\lambda(A)^{1-\theta}\right)^2}\cdot \frac{1}{n^2}\sum_{i=1}^n\sum_{j=1}^n a_{ij}\|x_i-x_j\|_X^2.
%\end{eqnarray}
\begin{multline*}
\frac{1}{n^2}\sum_{i=1}^n\sum_{j=1}^n \|x_i-x_j\|_X^2\le \nonumber \frac{2}{n^2}\sum_{i=1}^n\sum_{j=1}^n \left(\left\|x_i-S(x)\right\|_X^2+\left\|x_j-S(x)\right\|_X^2\right)
\\=4\|\overline{x}\|_{L_2^n(X)}^2\stackrel{\eqref{eq:overline upper}
\wedge\eqref{eq:overline lower}}{\le}
\frac{4}{\left(1-2^{1-\theta}\lambda(A)^{\theta}\right)^2}\cdot
\frac{1}{n}\sum_{i=1}^n\sum_{j=1}^n a_{ij}\|x_i-x_j\|_X^2.
\end{multline*}
Equivalently,
\begin{equation}\label{eq:gamma bound theta}
\gamma\!\left(A,\|\cdot\|_X^2\right)\le \frac{4}{\left(1-2^{1-\theta}\lambda(A)^{\theta}\right)^2}.
\end{equation}

We shall now apply a trick that was used by Pisier in~\cite{Pis10-interpolation}, where it is attributed to V. Lafforgue: we can ensure that the condition~\eqref{eq:suppsoe} holds true if we work with a large enough power of $A$. We will then be able to return back to an inequality that involves $A$ rather than its power by using Markov type through Corollary~\ref{cor:smoothness}. Specifically, define
\begin{equation}\label{eq:our m power}
m\eqdef \left\lceil  \frac{(2-\theta)\log 2}{\theta\log(1/\lambda(A))}\right\rceil.
\end{equation}
This choice of $m$ ensures that
$$
2^{1-\theta} \lambda(A^m)^{\theta}=2^{1-\theta}\lambda(A)^{m\theta}\le \frac12,
$$
so we may apply~\eqref{eq:gamma bound theta} with $A$ replaced by $A^m$ to get the estimate
\begin{equation}\label{eq:mth power bound gamma}
\gamma(A^m,\|\cdot\|_X^2)\le 16.
\end{equation}
An application of Corollary~\ref{cor:smoothness} with $p=2$ now implies that
\begin{multline}\label{eq:use mtype}
\gamma(A,\|\cdot\|_X^2)\stackrel{\eqref{eq:gamma comparison
smoothness}\wedge\eqref{eq:mth power bound gamma}} {\lesssim}
m^{\frac{2}{q}}S_q(X)^2\\\stackrel{\eqref{eq:our m power}}{\lesssim}
S_q(X)^2\left(1+\frac{1}{\theta\log\left(1/\lambda(A)\right)}\right)^{\frac{2}{q}}\lesssim
\frac{S_q(X)^2}{\left(1-\lambda(A)^{\theta}\right)^{2/q}},
\end{multline}
where in~\eqref{eq:use mtype} we used the elementary inequality $$\forall\, x\in [0,1],\qquad 1+\frac{1}{\log(1/x)}\le \frac{2}{1-x},$$
which holds true because $\exp\left(-\frac{1-x}{1+x}\right)\ge 1- \frac{1-x}{1+x}=\frac{2x}{1+x}\ge x$. \end{proof}

\subsection{Ramanujan graphs and Alon--Roichman graphs}\label{sec:p
index} Given a connected $n$-vertex graph $G=(\n,E_G)$, let
$d_G(\cdot,\cdot)$ denote the shortest path metric that $G$ induces
on $\n$. The diameter of the metric space $(\n,d_G)$ will be denoted
below by $\diam(G)$. Suppose that $d\in \{3,\ldots,n-1\}$ and that
$G$ is $d$-regular, i.e., for every $i\in \n$ the number of $j\in
\n$ such that $\{i,j\}\in E_G$ equals $d$. The normalized adjacency
matrix of $G$ will be denoted $A_G$, i.e.,
$(A_G)_{ij}=\frac{1}{d}\1_{\{i,j\}\in E_G}$ for every $i,j\in \n$.
Thus $A_G$ is an $n$ by $n$ symmetric stochastic matrix. We denote
$\lambda_i(G)=\lambda_i(A_G)$ for every $i\in \n$, and we
correspondingly set $\lambda(G)=\lambda(A_G)$ and
$\gamma(G,d_X^q)=\gamma(A_G,d_X^q)$ for every metric space $(X,d_X)$
and $q\in [1,\infty)$.

Setting $c_X(G)\eqdef c_X(\n,d_G)$, an important idea of Linial,
London and Rabinovich~\cite{LLR} relates $\gamma(G,d_X^q)$ to a
lower bound on $c_X(G)$ as follows. Let $f:\n\to X$ be a nonconstant
function and apply~\eqref{eq:def gamma} with $A=A_G$ and $x_i=f(i)$
for every $i\in \n$, thus obtaining the estimate
\begin{eqnarray}\label{eq:LLRp}
\nonumber\frac{1}{n^2}\sum_{i=1}^n\sum_{j=1}^n d_X(f(i),f(j))^q&\le& \frac{\gamma(G,d_X^q)}{dn}
\sum_{\substack{(i,j)\in \n^2\\\{i,j\}\in E_G}} d_X(f(i),f(j))^q\\&\le&
\gamma(G,d_X^q)\|f\|_{\Lip}^q.
\end{eqnarray}
Denoting $Av_X^{(q)}(G)\eqdef Av_X^{(q)}(\n,d_G)$, it follows
from~\eqref{eq:LLRp} that
\begin{equation}\label{eq:Avp lower bound gamma}
c_X(G)\ge Av_X^{(q)}(G)\ge \frac{1}{\gamma(G,d_X^q)^{1/q}} \left(\frac{1}{n^2}
\sum_{i=1}^n\sum_{j=1}^n d_G(i,j)^q\right)^{\frac1{q}}.
\end{equation}

For notational simplicity we will assume from now on that $G$ is a vertex-transitive graph, since in this case we have
\begin{equation}\label{eq:p norm vertex transitive}
\frac{\diam(G)}{2^{1+1/q}}\le \left(\frac{1}{n^2}\sum_{i=1}^n\sum_{j=1}^n d_G(i,j)^q\right)^{\frac1{q}}\le \diam(G).
\end{equation}
One verifies the validity of~\eqref{eq:p norm vertex transitive} by
arguing as in the proof of Proposition~3.4 of~\cite{NR09}: for every
$i\in \n$ and $r\in (0,\infty)$ let $B_G(i,r)=\{j\in \n:\
d_G(i,j)\le r\}$ be the closed ball of radius $r$ centered at $i$ in
the metric $d_G$. Since $G$ is vertex-transitive, the cardinality of
$B_G(i,r)$ is independent of $i$. Hence, if we let $r_*$ be the
minimum $r\in \N$ such that $|B_G(i,r)|>n/2$ for every $i\in \n$,
then for every $i\in \n$ we have $|\n\setminus B_G(i,r_*-1)|\ge
n/2$. In other words, for every $i\in \n$ there are at least $n/2$
vertices $j\in \n$ with $d_G(i,j)\ge r_*$. Hence
$\sum_{i=1}^n\sum_{j=1}^n d_G(i,j)^q\ge n^2r_*^q/2$. At the same
time, by the definition of $r_*$ we have $B_G(i,r_*)\cap
B_G(j,r_*)\neq \emptyset$ for every $i,j\in \n$, and therefore
$\diam(G)\le 2r_*$. This proves the leftmost inequality in~\ref{eq:p
norm vertex transitive} (the remaining inequality in~\eqref{eq:p
norm vertex transitive} is trivial).

By combining~\eqref{eq:Avp lower bound gamma} and~\eqref{eq:p norm
vertex transitive}, Matou\v{s}ek's argument in~\cite{Mat97} deduces
from his bound~\eqref{eq:matousek extrapolation} that if
$G=(\n,E_G)$ is a vertex-transitive graph such that $\lambda_2(G)$
is bounded away from $1$ by a universal constant then for every
$p\in [2,\infty)$ we have
\begin{equation}\label{eq:Matousek cp bound}
c_p(G)\gtrsim \frac{\diam(G)}{p}.
\end{equation}
Denote $p(G)\eqdef p(\n,d_G)$, where we recall that in
Section~\ref{sec:embedding application} we defined for a separable
metric space $(X,d_X)$ the quantity $p(X,d_X)$ (or simply $p(X)$ if
the metric is clear from the context) to be the infimum over those
$p\in [2,\infty]$ for which $c_p(X)\le 10$. It follows
from~\eqref{eq:Matousek cp bound} that $p(G)\gtrsim \diam(G)$ (still
under the assumption that $\lambda_2(G)$ is bounded away from $1$).
Using Corollary~\ref{coro:gamma p by lambda}, we now show that it is
possible to improve over this estimate.

\begin{proposition}\label{prop:better matousek lambda}
Fix $p\in [2,\infty)$ and let $G=(\n,E_G)$ be a vertex-transitive
graph. Then
$$
p\le \log\left(\frac{1}{\lambda(G)}\right)\implies c_p(G)\gtrsim \frac{\diam(G)}{\sqrt{p}},
$$
and
$$
p\ge \log\left(\frac{1}{\lambda(G)}\right)\implies c_p(G)\gtrsim \frac{\diam(G)}{p}
\sqrt{\log\left(\frac{1}{\lambda(G)}\right)}.
$$
\end{proposition}

\begin{proof}
By combining~\eqref{eq:Avp lower bound gamma} and~\eqref{eq:p norm
vertex transitive} (for $q=2$) with Corollary~\ref{coro:gamma p by
lambda} we see that
\begin{equation*}
c_p(G)\ge Av_X^{(2)}(G)\gtrsim \frac{\sqrt{1-\lambda(G)^{2/p}}}{\sqrt{p}}\diam(G).\qedhere
\end{equation*}
\end{proof}

By the definition of $p(G)$, the following corollary is a formal
consequence of Proposition~\eqref{prop:better matousek lambda}.

\begin{corollary}
Let $G=(\n,E_G)$ be a vertex-transitive graph. Then
$$
\lambda(G)\le e^{-\diam(G)^2}\implies p(G)\gtrsim \diam(G)^2,
$$
and
$$
\lambda(G)\ge e^{-\diam(G)^2}\implies p(G)\gtrsim \diam(G)\sqrt{\log\left(\frac{1}{\lambda(G)}\right)}.
$$
\end{corollary}

For every $n\in \N$ and $d\in \{3,\ldots,n-1\}$, if $G=(\n,E_G)$ is
a $d$-regular graph then by~\cite{Chu89} we have
\begin{equation*}\label{eq:chung}
\frac{\log n}{\log d}\lesssim \diam(G)\le 1+\frac{\log n}{\log(1/\lambda(G))}.
\end{equation*}
Consequently, if $1/\lambda(G)$ is at least $d^c$ for some universal
constant $c>0$ then $\diam(G)\asymp \log_d n$. This happens in
particular when $G$ is a Ramanujan graph, i.e., $\lambda(G)\le
2\sqrt{d-1}/d$. Such graphs have been constructed
in~\cite{LPS,Mar88}. It is natural to ask for the asymptotic
evaluation of $c_p(G)$ when $G$ is an $n$-vertex $d$-regular
Ramanujan graph. While this question remains open,
Proposition~\ref{prop:ramanujan} below contains a lower bound on
$c_p(G)$ that improves over Matou\v{s}ek's bound.

\begin{proposition}\label{prop:ramanujan}
Fix $n\in \N$ and $d\in \{3,\ldots,n-1\}$. Suppose that $G=(\n,E_G)$
is a Ramanujan graph. Then
\begin{equation}\label{eq:ramanujan1}
2\le p\le \log d\implies c_p(G)\gtrsim \frac{\log n}{\sqrt{p}\log d}.
\end{equation}
\begin{equation}\label{eq:ramanujan2}
p\ge \log d\implies c_p(G)\gtrsim \frac{\log n}{p\sqrt{\log d}}.
\end{equation}
\end{proposition}

\begin{proof}
By combining~\eqref{eq:Avp lower bound gamma} and~\eqref{eq:p norm
vertex transitive} with Corollary~\ref{coro:gamma p by lambda} we
see that
\begin{eqnarray*}\label{eq:use Ramanujan}
\nonumber c_p(G)&\gtrsim&
\frac{\sqrt{1-\lambda(G)^{2/p}}}{\sqrt{p}}
\left(\frac{1}{n^2}\sum_{i=1}^n\sum_{j=1}^n
d_G(i,j)^2\right)^{\frac12}\\&\gtrsim&
\frac{\sqrt{1-\left(2/\sqrt{d}\right)^{2/p}}}{\sqrt{p}}\log_d
n,
\end{eqnarray*}
where we used the fact that $\frac1{n^2}\sum_{i=1}^n\sum_{j=1}^n
d_G(i,j)^2\gtrsim (\log_d n)^2$ and that $\lambda(G)\le 2/\sqrt{d}$.
The latter bound uses the fact that $G$ is a Ramanujan graph (in
fact, weaker bounds on $\lambda(G)$ suffice for our purposes), and
the former bound holds true for any connected $n$-vertex $d$-regular
graph (see~\cite{Mat97} for a simple proof of this).
%The
%inequalities that are asserted in Proposition~\ref{prop:ramanujan}
%follow from~\eqref{eq:use Ramanujan} by considering the cases $p\le
%\log d$ and $p\ge \log d$ separately.
\end{proof}

%Recall that in Section~\ref{sec:embedding application} we defined
%$p(n,d)$ to be the expectation of $p(G)$ when $G$ is a uniformly
%random $n$-vertex $d$-regular graph, and we stated bounds on
%$p(n,d)$ in Proposition~\ref{prop:better than matousek pnd}.

If $G$ is a uniformly random $n$-vertex $d$-regular graph then
by~\cite{BS87} $\lambda(G)\le 2/\sqrt[4]{d}$ with high probability
(for the best known bound on $\lambda(G)$ when $G$ is a random
$d$-regular graph, see~\cite{Fri08}). By arguing identically to the
proof of Proposition~\ref{prop:ramanujan},  we see that with high
probability~\eqref{eq:ramanujan1} and~\eqref{eq:ramanujan2} hold
true for such $G$, implying Proposition~\ref{prop:better than
matousek pnd}.

Corollary~\ref{coro:gamma p by lambda} also implies new distortion
bounds for Abelian Alon--Roichman graphs~\cite{AR94}. These are
graphs that are obtained from the following random construction. Let
$\Gamma$ be a finite Abelian group, and think of $\Gamma$ as the set
$\n$, equipped with an Abelian group operation. Fix $\e\in (0,1/2)$
and set $k=\left\lceil \frac{3}{\e^2} \log n\right\rceil$. Let
$g_1,\ldots,g_k\in \Gamma$ be chosen independently and uniformly at
random. This induces a random Cayley graph $G$ whose generating
multi-set is $\{g_1,g_1^{-1},\ldots,g_k,g_k^{-1}\}$. As explained
in~\cite{AR94}, with probability that tends to $1$ as $n\to \infty$
the graph $G$ is connected. Note that since $G$ is a Cayley graph it
is vertex-transitive. It follows from~\cite{CM08} that provided $n$
is large enough we have $\lambda(G)\le \e$ with probability at least
$\frac12$. Moreover, by~\cite[Prop.~3.5]{NR09} we have
$\diam(G)\gtrsim (\log n)/(\log 1/\e)$. A substitution of these
estimates into Proposition~\ref{prop:better matousek lambda} shows
that
$$
2\le p\le \log(1/\e)\implies c_p(G)\gtrsim \frac{\log n}{\sqrt{p}\log(1/\e)},
$$
and
$$
p\ge \log(1/\e)\implies c_p(G)\gtrsim \frac{\log n}{p\sqrt{\log(1/\e)}}.
$$

\begin{remark}
We warn that there is some subtlety in the definition of the
parameter $p(X)$ for a separable metric space $(X,d_X)$. Given that
$X$ is isometric to a subset of $\ell_\infty$, it is indeed natural to ask
for the smallest $p\in [2,\infty]$ such that $X$ embeds with bounded
distortion, say, distortion $10$, into $\ell_p$. As an example of an application that was shown to us by Yuval Rabani, one can use the methods of~\cite{KOR00} to prove that subsets of $\ell_p$ admit an efficient approximate nearest neighbor data structure with approximation guarantee $e^{O(p)}$, so the parameter $p(X)$ relates to approximate nearest neighbor search in $X$ (it would be very interesting to determine the correct asymptotic dependence on $p$ here). But, understanding the set of $p\in
[2,\infty)$ for which $X$ admits a bi-Lipschitz embedding into
$\ell_p$ can be subtle. In particular, it is not true that if
$X$ embeds into $\ell_p$ then for every $q>p$ it also embeds into
$\ell_q$. In fact, we have the following estimates for every $n\in
\N$ and $p,q\in (2,\infty)$.
\begin{equation}\label{eq:q<p cotype}
2<q<p\implies c_q\left(\ell_p^n\right)\asymp n^{\frac{1}{q}-\frac{1}{p}},
\end{equation}
and
\begin{equation}\label{eq:FJS}
2<p<q\implies c_q\left(\ell_p^n\right)\asymp_{p,q} n^{\frac{(q-p)(p-2)}{(q-2)p^2}}.
\end{equation}
The asymptotic identity~\eqref{eq:q<p cotype} is a standard
consequence of the fact that $L_q$ has Rademacher cotype $q$ (see
e.g.~\cite{Woj91}). The remarkable asymptotic
identity~\eqref{eq:FJS} is due to~\cite{FJS88} (using a computation
of~\cite{GPP80}). The implicit dependence on $p,q$ in~\eqref{eq:FJS}
is unknown, and it would be of interest to evaluate it up to
a universal constant factor. Observe that the exponent of $n$
in~\eqref{eq:FJS} tends to $(p-2)/p^2>0$ as $q\to \infty$, and
therefore the implicit constant in~\eqref{eq:FJS} must tend to $0$
as $q\to \infty$.
\end{remark}

\subsection{Curved Banach spaces in the sense of Pisier}\label{sec:curved}
Motivated by his work on nonlinear spectral gaps~\cite{Laff08}, V. Lafforgue associated the following modulus to a Banach space $(X,\|\cdot\|_X)$, a modulus has been   investigated extensively by Pisier in~\cite{Pis10-interpolation}. Given $\e\in (0,\infty)$ let $\Delta_X(\e)$ denote the infimum over those $\Delta\in (0,\infty)$ such that for every $n\in \N$, every matrix $T=(t_{ij})\in M_n(\R)$ with $$\|T\|_{L_2^n(\R)\to L_2^n(\R)}\le \e\qquad \mathrm{and}\qquad \|\mathrm{abs}(T)\|_{L_2^n(\R)\to L_2^n(\R)}\le 1,$$ where $\mathrm{abs}(T)\eqdef (|t_{ij}|)$ is the entry-wise absolute value of $T$, satisfies
$$
 \|T\otimes I_X\|_{L_2^n(X)\to L_2^n(X)}\le \Delta.
$$

Pisier introduced the following terminology in~\cite{Pis10-interpolation}: $X$ is said to be {\em curved} if $\Delta_X(\e)<1$ for some $\e\in (0,1)$. $X$ is said to be {\em fully curved} if $\Delta_X(\e)<1$ for all $\e\in (0,1)$, and $X$ is said to be {\em uniformly curved} if $\lim_{\e\to 0} \Delta_X(\e)=0$. It is shown in~\cite{Pis10-interpolation} that if $X$ is either fully curved or uniformly curved then it admits an equivalent uniformly convex norm. A remarkable characterization of Pisier~\cite{Pis10-interpolation} shows that $\Delta_X(\e)\lesssim \e^\alpha$ for some $\alpha\in (0,\infty)$ if and only if $X$ arises from complex interpolation with Hilbert space: formally, this happens if and only if $X$ is isomorphic to a quotient of a subspace of an ultraproduct of $\theta$-Hilbertian Banach spaces for some $\theta\in (0,1)$; we refer to~\cite{Pis10-interpolation} for the definition of these notions. A more complicated structural characterization of uniformly curved spaces (based on real interpolation) is also obtained in~\cite{Pis10-interpolation}.

One can use the above notions to give a generalized abstract treatment of results in the spirit of Theorem~\ref{thm:smoothness mtype}. Fix $\e\in (0,1)$ and suppose that $\Delta_X(\e)<\frac12$. Let $A\in M_n(\R)$ be symmetric and stochastic and let $T=(t_{ij})\in M_n(\R)$ be given as in~\eqref{eq:spectral T}.  By~\eqref{eq:hilbertian operator norm}  we have $\|T\|_{L_2^n(\R)\to L_2^n(\R)}= \lambda(A)$. Moreover, since $\mathrm{abs}(T)=(|a_{ij}-1/n|)$ and $A$ is symmetric and stochastic, it is immediate to check that $\|\mathrm{abs}(T)\|_{L_2^n(\R)\to L_2^n(\R)}\le 2$. By the definition of the modulus $\Delta_X(\cdot)$ we therefore have $\|T\otimes I_X\|_{L_2^n(X)\to L_2^n(X)}\le 2\Delta_X(\lambda(A)/2)$. Define
\begin{equation}\label{eq:our m power pisier}
m\eqdef \left\lceil \frac{\log(1/(2\e))}{\log(1/\lambda(A))}\right\rceil,
\end{equation}
so that $\lambda(A^m)/2=\lambda(A)^m/2\le \e$, and apply the above
reasoning with $A$ replaced by $A^m$. Arguing as
in~\eqref{eq:overline upper} and~\eqref{eq:overline lower}, we
obtain the estimate
$$
\gamma\!\left(A^m,\|\cdot\|_X^2\right)\le \frac{4}{\left(1-2\Delta_X(\e)\right)^2}.
$$
We can now use the notion of Markov type through Lemma~\ref{lem:mtype lower decay gap} to deduce the following statement.
\begin{theorem}\label{thm:curved general}
Fix $\e\in (0,1)$ and let $(X,\|\cdot\|_X)$ be a Banach that satisfies $\Delta_X(\e)<\frac12$. Then for every $n\in \N$ and every symmetric stochastic matrix $A\in M_n(\R)$ we have
\begin{multline*}
\gamma\!\left(A,\|\cdot\|_X^2\right)\\\le
\frac{4}{\left(1-2\Delta_X(\e)\right)^2}\cdot M_2\left(X;\left\lceil
\frac{\log(1/(2\e))}{\log(1/\lambda(A))}\right\rceil\right)^2\cdot\left\lceil
\frac{\log(1/(2\e))}{\log(1/\lambda(A))}\right\rceil .
\end{multline*}
In particular, using Theorem~\ref{thm:npss} and arguing as in the proof of Corollary~\ref{coro:pass to lambda2}, if $X$ has modulus of smoothness of power type $2$ then every symmetric stochastic matrix $A\in M_n(\R)$ satisfies
$$
\gamma\!\left(A,\|\cdot\|_X^2\right)\lesssim_X \frac{1}{1-\lambda_2(A)}.
$$
\end{theorem}

In conjunction with Theorem~\ref{thm:duality} we deduce the following geometric embedding result for uniformly curved Banach space that can be renormed so as to have modulus of smoothness of power type $2$.

\begin{corollary}
Suppose that $(X,\|\cdot\|_X)$ is a uniformly curved Banach space
that admits an equivalent norm whose modulus of uniform smoothness
has power type $2$. Then for every $x_1,\ldots,x_n\in X$ there exist
$y_1,\ldots, y_n\in \ell_2$ such that
$$
\sum_{i=1}^n\sum_{j=1}^n \|x_i-x_j\|_X^2=\sum_{i=1}^n\sum_{j=1}^n \|y_i-y_j\|_2^2,
$$
and
$$
\forall\, i,j\in \{1,\ldots,n\},\qquad \|y_i-y_j\|_2\lesssim_X \|x_i-x_j\|_X.
$$
%where the implied constant depends only on $X$ but not on $x_1,\ldots,x_n$.
\end{corollary}

\subsection{An interpolation inequality for nonlinear spectral gaps}\label{sec:interpolation} The {\em modulus of uniform convexity\/} of a Banach space
$(X,\|\cdot\|_X)$ is defined for $\e\in [0,2]$ as
\begin{equation*}
\delta_X(\e)\eqdef\inf\left\{ 1-\frac{\|x+y\|_X}{2}:\ x,y\in B_X\ \wedge\  \|x-y\|_X=\e\right\}.
\end{equation*}
$X$ is said to be {\em uniformly convex} if $\delta_X(\e)>0$ for all
$\e\in (0,2]$. Furthermore, $X$ is said to have modulus of convexity
of power type $p$ if there exists a constant $c\in (0,\infty)$ such
that $\delta_X(\e)\ge c\,\e^p$ for all $\e\in [0,2]$. It is
straightforward to check that in this case necessarily $p\ge 2$. By
Proposition 7 in~\cite{BCL} (see also~\cite{Fiegel76}), $X$ has
modulus of convexity of power type $p$ if and only if there exists a
constant $K\in [1,\infty)$ such that for every $x,y\in X$
\begin{eqnarray}\label{eq:two point convex}
\left(\|x\|_X^p+\frac{1}{K^p}\|y\|_X^p\right)^{\frac{1}{p}}\le
\left(\frac{\|x+y\|_X^p+\|x-y\|_X^p}{2}\right)^{\frac{1}{p}}.
\end{eqnarray}
The infimum over those $K$ for which~\eqref{eq:two point convex}
holds true is called the $p$-convexity constant of $X$, and is
denoted $K_p(X)$. Note that every Banach space satisfies
$K_\infty(X)=1$. Below we shall use the convention $K_p(X)=\infty$ if $p\in [1,2)$. For $1\le q\le 2\le p$, the $p$-convexity constant of $X$ is related to the
$q$-smoothness constant of $X$ (recall~\eqref{eq:two point smooth}) via the
following duality relation~\cite[Lem.~5]{BCL}.
$$
\frac{1}{p}+\frac{1}{q}=1\implies K_p(X)=S_q(X^*).
$$

\begin{theorem}\label{thm:general interpolation min}
Let $(X,Y)$ be a compatible pair of Banach spaces. Fix $\theta\in
[0,1]$ and consider the complex interpolation space
$Z=[X,Y]_\theta$. Fix also $p,q\in [1,\infty]$ and $r\in [1,2]$.
Then for every $n\in \N$ and every $n$ by $n$ symmetric stochastic
matrix $A$ we have
\begin{multline*}
\frac{\gamma(A,\|\cdot\|_Z^s)}{c^s\left(s^{\frac{s}{r}}+S_r(Z)^s\right)}\\\le \left(\min\left\{\frac{(9K_p(X))^{p}\gamma(A,\|\cdot\|_X^p)}{\theta},
\frac{(9K_q(Y))^{q}\gamma(A,\|\cdot\|_Y^q)}{1-\theta}\right\}\right)^{\frac{s}{r}},
\end{multline*}
where $c\in (0,\infty)$ is a universal constant and  $s\in
[2,\infty]$ is given by
$$
\frac{1}{s}=\frac{\theta}{p}+\frac{1-\theta}{q}.
$$
\end{theorem}

Observe that for every $a,b\in (0,\infty)$ and every $\theta\in
(0,1)$ we have
$$
\min\left\{\frac{a}{\theta},\frac{b}{1-\theta}\right\}\le \frac{2}{\frac{\theta}{a}+\frac{1-\theta}{b}}\le a^\theta b^{1-\theta}.
$$
Consequently, the conclusion of Theorem~\ref{thm:general
interpolation min} implies that
$$
\gamma(A,\|\cdot\|_Z^s)\lesssim_{X,Y,Z,s} \gamma(A,\|\cdot\|_X^p)^{\theta s/r}
\gamma(A,\|\cdot\|_Y^q)^{(1-\theta)s/r}.
$$
Such an estimate is in the spirit of the interpolation
inequality~\eqref{eq:interpolation hope}, but it is insufficient for
the purpose of addressing Question~\ref{Q:p ARV}.
Theorem~\ref{thm:general interpolation min} does suffice to prove
Lemma~\ref{lem:p sharp}, so we assume the validity of
Theorem~\ref{thm:general interpolation min} for the moment and
proceed now to prove Lemma~\ref{lem:p sharp}.

\begin{proof}[Proof of Lemma~\ref{lem:p sharp}]
Matou\v{s}ek proved in~\cite{Mat97} that if $(X,d_X)$ is an
$n$-point metric space then for $p\in [2,\infty)$ we have
\begin{equation}\label{eq:matousek bourgain}
c_p(X)\lesssim 1+\frac{\log n}{p}.
\end{equation}
The case $p=2$ of~\eqref{eq:matousek bourgain} is Bourgain's
embedding theorem~\cite{Bourgain-embed}. Now, for every $n\in \N$
let $G_n=(\n,E_{G_n})$ be a $4$-regular graph with $\sup_{n\in
\N}\lambda_2(G_n)<1$, i.e., $\{G_n\}_{n=1}^\infty$ forms an expander
sequence. Fixing $n\ge e^p$, by~\eqref{eq:matousek bourgain} we know
that there exist $x_1,\ldots,x_n\in \ell_p$ such that
\begin{equation}\label{eq:log n /p embedding}
\forall\, i,j\in \n,\qquad       d_{G_n}(i,j)\le \|x_i-x_j\|_{\ell_p}\lesssim \frac{\log n}{p}d_{G_n}(i,j).
\end{equation}

Suppose that $y_1,\ldots,y_n\in \ell_q$ satisfy~\eqref{eq:expander
rs norm}, i.e.,
\begin{equation}\label{eq:remind rs assumption}
\left(\frac{1}{n^2}\sum_{i=1}^n\sum_{j=1}^n\|y_i-y_j\|_{\ell_q}^r\right)^{\frac{1}{r}}=
\left(\frac{1}{n^2}\sum_{i=1}^n\sum_{j=1}^n\|x_i-x_j\|_{\ell_p}^s\right)^{\frac{1}{s}},
\end{equation}
If $\|y_i-y_j\|_{\ell_q}\le D\|x_i-x_j\|_{\ell_p}$ for every $i,j\in
\n$ then we need to show that $D\gtrsim p/(q+r)$. Note that since
$G_n$ is $4$-regular, a constant fraction of the pairs $(i,j)\in
\n^2$ satisfy $d_{G_n}(i,j)\gtrsim \log n$ (the standard argument
showing this appears in e.g.~\cite{Mat97}). Hence, due to the
leftmost inequality in~\eqref{eq:log n /p embedding}, it follows
from~\eqref{eq:remind rs assumption} that
\begin{equation}\label{eq:at least log n in lemma}
\left(\frac{1}{n^2}\sum_{i=1}^n\sum_{j=1}^n\|y_i-y_j\|_{\ell_q}^r\right)^{\frac{1}{r}}\gtrsim \log n.
\end{equation}

\medskip
\noindent{\bf Case 1.} $r\le q$. In this case we have
\begin{multline}\label{r->q}
\left(\frac{1}{n^2}\sum_{i=1}^n\sum_{j=1}^n\|y_i-y_j\|_{\ell_q}^r\right)^{\frac{1}{r}}\le
\left(\frac{1}{n^2}\sum_{i=1}^n\sum_{j=1}^n\|y_i-y_j\|_{\ell_q}^q\right)^{\frac{1}{q}} \\
\le \gamma(G_n,\|\cdot\|_{\ell_q}^q)^{\frac{1}{q}}\left(\frac{1}{4n}
\sum_{\substack{(i,j)\in \n^2\\\{i,j\}\in
E_{G_n}}}\|y_i-y_j\|_{\ell_q}^q\right)^{\frac{1}{q}} \lesssim
\frac{Dq\log n}{p},
\end{multline}
where in~\eqref{r->q} we bounded $\gamma(G_n,\|\cdot\|_{\ell_q}^q)$
using~\eqref{eq:matousek extrapolation}, and we used the fact that
if $\{i,j\}\in E_{G_n}$ then $\|y_i-y_j\|_{\ell_q}\le
D\|x_i-x_j\|_{\ell_p}\lesssim D(\log n)/p$, due to~\eqref{eq:log n
/p embedding}. By contrasting~\eqref{r->q} with~\eqref{eq:at least
log n in lemma} we have $D\gtrsim p/q$, as required.

\medskip
\noindent{\bf Case 2.} $r>q$. Write $\theta\eqdef 1/(r-1)$ and $t\eqdef 2q(r-2)/(2r-q-2)$. Recalling that we are assuming in Lemma~\ref{lem:p sharp} that $q\ge 2$, it follows that $\theta\in (0,1)$ and $t\in [q,2r]$. Consequently, $K_{2r}(\ell_t)\le K_t(\ell_t)\le 1$ (see~\cite{BCL}). Note also that $1/q=\theta/2+(1-\theta)/t$, so $\ell_q=[\ell_2,\ell_t]_\theta$. Since $1/r=\theta/2+(1-\theta)/(2r)$ and $S_2(\ell_q)=\sqrt{q-1}$ (see~\cite{BCL}),  it follows from Theorem~\ref{thm:general interpolation min} that  for every $n\in \N$ we have
$$
\gamma(G_n,\|\cdot\|_{\ell_q}^r)^{1/r}\lesssim \sqrt{r\cdot
\min\left\{r\gamma(G_n,\|\cdot\|_{\ell_2}^2),\gamma(G_n,\|\cdot\|_{\ell_t}^{2r})\right\}}\lesssim r.
$$
Hence, if we argue as in~\eqref{r->q} we see that $D\gtrsim p/r$, as
required.
\end{proof}

We now prove Theorem~\ref{thm:general interpolation
min}.

\begin{proof}[Proof of Theorem~\ref{thm:general
interpolation min}] We may assume that $A$ is ergodic, implying that
$\gamma(A,\|\cdot\|_X^p),\gamma(A,\|\cdot\|_Y^q)<\infty$, since
otherwise the conclusion of Theorem~\ref{thm:general interpolation
min} is vacuous. So, by Lemma~\ref{lem:I+A} we have
\begin{equation}\label{eq:I+A p}
\gamma_+\!\left(\frac{I+A}{2},\|\cdot\|_X^p\right)\le 2^{2p+1}\gamma(A,\|\cdot\|_X^p)<\infty,
\end{equation}
and
\begin{equation}\label{eq:I+A q}
\gamma_+\!\left(\frac{I+A}{2},\|\cdot\|_Y^q\right)\le 2^{2q+1}\gamma(A,\|\cdot\|_X^q)<\infty,
\end{equation}

For a Banach space $(W,\|\cdot\|_W)$ and $t\in [1,\infty]$ let
$L_t^n(W)_0$ be the subspace of $L_t^n(W)$ consisting of mean-zero
vectors, i.e.,
$$
L_t^n(W)_0\eqdef \left\{(w_1,\ldots,w_n)\in L_t^n(W):\ \sum_{i=1}^n w_i=0\right\}.
$$
Let $Q:L_t^n(\R)\to L_t^n(\R)_0$ be the canonical projection, i.e,
for every $v\in L_t^n(\R)$ and $i\in \n$,
$$
(Qv)_i\eqdef v_i-\frac{1}{n}\sum_{j=1}^n v_j.
$$
Then by the triangle inequality, $\|Q\otimes I_W\|_{L_t^n(W)\to
L_t^n(W)_0}\le 2$, and consequently for every $B\in M_n(\R)$ such
that $B(L_t^n(\R)_0)\subset L_t^n(\R)_0$,
\begin{equation}\label{eq:projection norm 2}
\left\|\left(BQ\right)\otimes I_W\right\|_{L_t^n(W)\to L_t^n(W)}\le 2\|B\otimes I_W\|_{L_t^n(\R)_0\to
L_t^n(\R)_0}.
\end{equation}
Note that if $B$ is symmetric and stochastic then
$B(L_t^n(\R)_0)\subset L_t^n(\R)_0$, and
by~\cite[Lem.~6.6]{MN-towards} for every $t\in [1,\infty]$ we have
\begin{equation}\label{eq:quote 6.6}
\|B\otimes I_W\|_{L_t^n(\R)_0\to
L_t^n(\R)_0}\le e^{-2/(t(2K_t(W))^t\gamma_+(B,\|\cdot\|_W^t))}.
\end{equation}
(Observe that we always have $\|B\otimes I_W\|_{L_t^n(\R)_0\to
L_t^n(\R)_0}\le 1$ because $B$ is symmetric and stochastic,
so~\eqref{eq:quote 6.6} is most meaningful when $t\in [2,\infty)$
and $K_t(W)<\infty$.) Consequently, for every $m\in \N$ we have
\begin{eqnarray}\label{eq:B^m}
\nonumber \left\|B^m\otimes I_W\right\|_{L_t^n(\R)_0\to L_t^n(\R)_0}&\le&
\left(\|B\otimes I_W\|_{L_t^n(\R)_0\to
L_t^n(\R)_0}\right)^m \\
&\le&
e^{-2m/(t(2K_t(W))^t\gamma_+(B,\|\cdot\|_W^t))}.
\end{eqnarray}

Define $T\in M_n(\R)$ by
$$
T\eqdef \left(\frac{I+A}{2}\right)^mQ.
$$
An application of~\eqref{eq:B^m} and~\eqref{eq:projection norm 2}
with $W=X$ and $t=p$ while using~\eqref{eq:I+A p} shows that
\begin{equation}\label{eq:S on X}
\left\|T\otimes I_X\right\|_{L_p^n(X)\to L_p^n(X)}\le 2e^{-m/(4(9K_p(X))^p\gamma(A,\|\cdot\|_X^p))}.
\end{equation}
The same reasoning applied to $W=Y$ and $t=q$ while
using~\eqref{eq:I+A q}  shows that
\begin{equation}\label{eq:S on Y}
\left\|T\otimes I_Y\right\|_{L_q^n(Y)\to L_q^n(Y)}\le 2e^{-m/(4(9K_q(Y))^q\gamma(A,\|\cdot\|_Y^q))}.
\end{equation}

Interpolation of~\eqref{eq:S on X} and~\eqref{eq:S on Y} yields the
following estimate.
\begin{multline}\label{eq:interpolation used}
\left\|\left(\frac{I+A}{2}\right)^m\otimes
I_Z\right\|_{L_s^n(Z)_0\to L_s^n(Z)_0}\le \left\|T\otimes
I_Y\right\|_{L_s^n(Z)\to L_s^n(Z)}\\
\le
2e^{-m\theta/(4(9K_p(X))^p\gamma(A,\|\cdot\|_X^p))-m(1-\theta)/(4(9K_q(Y))^q\gamma(A,\|\cdot\|_Y^q))}.
\end{multline}

Let $m$ be given by
\begin{equation*}
m\eqdef \left\lceil 7\min\left\{\frac{(9K_p(X))^p\gamma(A,\|\cdot\|_X^p)}{\theta},
\frac{(9K_q(Y))^q\gamma(A,\|\cdot\|_Y^q)}{1-\theta}\right\} \right\rceil.
\end{equation*}
Then by~\eqref{eq:interpolation used} we have
$$
\left\|\left(\frac{I+A}{2}\right)^m\otimes
I_Z\right\|_{L_s^n(Z)_0\to L_s^n(Z)_0}\le \frac12.
$$
By~\cite[Lem.~6.1]{MN-towards} this implies that
\begin{equation}\label{eq:9s}
\gamma\!\left(\left(\frac{I+A}{2}\right)^m,\|\cdot\|_Z^s\right)\le \gamma_+\!\left(\left(\frac{I+A}{2}\right)^m,\|\cdot\|_Z^s\right)\le 9^s.
\end{equation}
Hence, using Corollary~\ref{cor:smoothness} we deduce that
\begin{eqnarray*}
2\gamma(A,\|\cdot\|_Z^s)&\stackrel{\eqref{eq:average identity easy
gamma}}{=}& \gamma\!\left(\frac{I+A}{2},\|\cdot\|_Z^s\right)\\
&\stackrel{\eqref{eq:gamma comparison smoothness}}{\lesssim}&
C^s\left(s^{\frac{s}{r}}+S_r(Z)^s\right)m^{\frac{s}{r}}\gamma\!\left(\left(\frac{I+A}{2}\right)^m,\|\cdot\|_Z^s\right)
\\&\stackrel{\eqref{eq:9s}}{\le}& (9C)^s\left(s^{\frac{s}{r}}+S_r(Z)^s\right)m^{\frac{s}{r}},
\end{eqnarray*}
where $C$ is the universal constant from
Corollary~\ref{cor:smoothness}. Recalling our choice of $m$, this
completes the proof of Theorem~\ref{thm:general interpolation min}.
\end{proof}

\section{Proof of Theorem~\ref{thm:ozawa-ineq}}\label{sec:ozawa}

Let $(X,\|\cdot\|_X)$ and $(Y,\|\cdot\|_Y)$ be Banach spaces. We first note that the assumptions of Theorem~\ref{thm:ozawa-ineq} are equivalent to the assertion that $B_X$ is uniformly homeomorphic to a subset of $Y$, i.e,  that there is an injection $\f:B_X\to Y$ such that both $\phi$ and $\f^{-1}$ are uniformly continuous. Indeed, since $B_X$ is metrically convex, the modulus of continuity of $\f$, namely the mapping $\omega_\f:[0,\infty)\to [0,\infty)$ given by
$$
\omega_\phi(t)\eqdef \sup\{\|\f(x)-\f(y)\|_Y: x,y\in B_X\ \wedge\ \|x-y\|_X\le t\},
$$
is sub-additive (see for example~\cite[Ch.~1, Sec.~1]{BL}). Consequently, as explained in~\cite[Ch.~1, Sec.~2]{BL}, there exists an increasing concave function $\beta:[0,\infty)\to [0,\infty)$ such that $\omega_\phi\le \beta\le 2\omega_\phi$.

Since $\beta:[0,\infty)\to [0,\infty)$ is concave, increasing, and $\beta(0)=0$, for every $q\in [1,\infty)$ the mapping $t\mapsto \beta(t^{1/q})^q$ is concave. Indeed, it suffices to verify this when $\beta$ is differentiable. Denote $f(t)=\beta(t^{1/q})^q$. Then
\begin{equation}\label{eq:simple derivative}
f'(t)=\left(\frac{\beta(t^{1/q})}{t^{1/q}}\right)^{q-1}\beta'(t^{1/q}).
\end{equation}
By our assumptions $\beta'$ is nonnegative and decreasing, and $s\mapsto \beta(s)/s$ is decreasing on $(0,\infty)$. It therefore follows from~\eqref{eq:simple derivative} that $f'$ is decreasing on $(0,\infty)$, as required.

The canonical radial retraction of $X$ onto $B_X$ is denoted below by $\rho:X\to B_X$, i.e.,
\begin{equation}\label{eq:def rho}
\rho(x)\eqdef \left\{\begin{array}{ll}x&\mathrm{if\ } \|x\|_X\le 1,\\
\frac{x}{\|x\|_X}&\mathrm{if\ }\|x\|_X\ge 1.\end{array}\right.
\end{equation}
It is straightforward to check that $\rho$ is $2$-Lipschitz.

\begin{lemma}\label{lem:minimal radius}
Under the assumptions of Theorem~\ref{thm:ozawa-ineq}, fix $n\in \N$, $q\in [1,\infty)$ and $x_1,\ldots,x_n\in X$ with
\begin{equation}\label{eq:equal to 1 assumption}
\frac{1}{n}\sum_{i=1}^n\sum_{j=1}^n a_{ij} \|x_i-x_j\|_X^q=1.
\end{equation}
%suppose that $x_1,\ldots,x_n\in X$ satisfy
%\begin{equation}\label{eq:equal to 1 assumption}
%\frac{1}{n}\sum_{i=1}^n\sum_{j=1}^n a_{ij} \|x_i-x_j\|_X^q=1.
%\end{equation}
For every $i\in \{1,\ldots,n\}$ let $r_i\in (0,\infty)$ be the smallest $r>0$ such that
\begin{equation}\label{eq: rj def}
\big|\{j\in \{1,\ldots,n\}:\ \|x_j-x_i\|_X\le r\}\big|\ge \frac{n}{2}.
\end{equation}
Then for every $n$ by $n$ symmetric stochastic matrix $A=(a_{ij})\in M_n(\R)$,
\begin{equation}\label{eq:minimal radius bound}
\min_{i\in \{1,\ldots,n\}} r_i\\\le \max\left\{2(8\gamma(A,\|\cdot\|_Y^q))^{1/q},\frac{2}{\beta^{-1}\left(\frac{\alpha(1/4)}{8\gamma\!\left(A,\|\cdot\|_Y^q\right)^{1/q}}\right)}\right\}.
\end{equation}
\end{lemma}
\begin{proof}
Note that the fact that the function $x\mapsto \beta(x^{1/q})^q$ is concave on $[0,\infty)$ implies that for every $\lambda\in (0,\infty)$ we have
\begin{multline}\label{eq:beta concavity ineq}
\left(\frac{1}{n}\sum_{i=1}^n\sum_{j=1}^na_{ij}\beta\left(\lambda\|x_i-x_j\|_X\right)^q\right)^{\frac1{q}}\\\le \beta\left(\lambda \left(\frac{1}{n}\sum_{i=1}^n\sum_{j=1}^n a_{ij} \|x_i-x_j\|_X^q\right)^{\frac1{q}}\right)\stackrel{\eqref{eq:equal to 1 assumption}}{=}\beta(\lambda).
\end{multline}
By relabeling the points if necessary we may also assume without loss of generality that
\begin{equation}\label{eq:r1 minimal}
r_1=\min_{i\in \{1,\ldots,n\}} r_i.
\end{equation}
Denote
\begin{equation}\label{eq:def B}
B\eqdef \{j\in \{1,\ldots,n\}:\ \|x_j-x_1\|_X\le r_1\}.
\end{equation}
So, by definition,
\begin{equation}\label{eq: B big}
|B|\ge \frac{n}{2}.
\end{equation}
For the sake of simplicity we denote below
\begin{equation}\label{eq:gamma notataion}
\gamma\eqdef \gamma\!\left(A,\|\cdot\|_Y^q\right).
\end{equation}

For $i\in \{1,\ldots,n\}$ define
$$
y_i\eqdef x_1+r_1\rho\left(\frac{x_i-x_1}{r_1}\right),
$$
where $\rho:X\to B_X$ is given in~\eqref{eq:def rho}. Since $\rho$ is $2$-Lipschitz,
\begin{equation}\label{eq:2Lip}
\forall\, i,j\in \{1,\ldots,n\},\qquad \|y_i-y_j\|_X\le 2\|x_i-x_j\|_X.
\end{equation}
By definition $y_i\in x_1+r_1B_X$, so we may consider the vectors
\begin{equation}\label{eq:def zi}
z_i\eqdef \f\left(\frac{y_i-x_1}{r_1}\right)\in Y.
\end{equation}
Now,
\begin{multline}\label{eq:p convex}
\frac{1}{n^2}\sum_{i=1}^n\sum_{j=1}^n \|z_i-z_j\|_Y^q \stackrel{\eqref{eq:gamma notataion}}{\le} \frac{\gamma}{n}\sum_{i=1}^n\sum_{j=1}^n a_{ij} \|z_i-z_j\|_Y^q
\\\stackrel{\eqref{eq;alpha beta}\wedge\eqref{eq:def zi}\wedge\eqref{eq:2Lip}}{\le} \frac{\gamma}{n}\sum_{i=1}^n\sum_{j=1}^n a_{ij}\beta\left(\frac{2\|x_i-x_j\|_X}{r_1}\right)^q\stackrel{\eqref{eq:beta concavity ineq}}{\le} \gamma\beta\left(\frac{2}{r_1}\right)^q.
\end{multline}

Denoting
$$
w\eqdef \frac{1}{n}\sum_{j=1}^nz_j\in Y,
$$
by the convexity of $\|\cdot\|_Y^q$ we have
\begin{equation}\label{eq:for markov ineq}
\frac{1}{n}\sum_{i=1}^n \|z_j-w\|_Y^q\le \frac{1}{n^2}\sum_{i=1}^n\sum_{j=1}^n \|z_i-z_j\|_Y^q\stackrel{\eqref{eq:p convex}}{\le} \gamma\beta\left(\frac{2}{r_1}\right)^q.
\end{equation}
It follows from~\eqref{eq:for markov ineq} and Markov's inequality that if we set
\begin{equation}\label{eq:def C}
C\eqdef\left\{i\in \{1,\ldots,n\}:\ \|z_i-w\|_Y\le (4\gamma)^{1/q}\beta\left(\frac{2}{r_1}\right)\right\},
\end{equation}
 then $|C|\ge 3n/4$. By~\eqref{eq: B big} it follows that
 \begin{equation}\label{eq:intersection big}
 |B\cap C|\ge \frac{n}{4}.
 \end{equation}

 Recalling the definitions~\eqref{eq:def B}, \eqref{eq:def C} and~\eqref{eq:def zi}, for every $i\in B\cap C$ we have
 \begin{equation}\label{eq:on B intersection C}
 z_i=\f\left(\frac{x_i-x_1}{r_1}\right)\qquad\mathrm{and}\qquad \|z_i-w\|_Y\le (4\gamma)^{1/q}\beta\left(\frac{2}{r_1}\right).
  \end{equation}
  Hence, for every $i,j\in B\cap C$,
\begin{multline*}
\alpha\left(\frac{\|x_i-x_j\|_X}{r_1}\right)\stackrel{\eqref{eq;alpha beta}}{\le} \left\|\f\left(\frac{x_i-x_1}{r_1}\right)-\f\left(\frac{x_j-x_1}{r_1}\right)\right\|_Y\\
\stackrel{\eqref{eq:on B intersection C}}{=}\|z_i-z_j\|_Y
\le \|z_i-w\|_Y+\|z_j-w\|_Y\stackrel{\eqref{eq:on B intersection C}}{\le} 2(4\gamma)^{1/q}\beta\left(\frac{2}{r_1}\right).
\end{multline*}
Consequently,
\begin{equation}\label{eq;diam of intersection}
i,j\in B\cap C\implies \|x_i-x_j\|_X\le r\eqdef r_1\alpha^{-1}\left(2(4\gamma)^{1/q}\beta\left(\frac{2}{r_1}\right)\right).
\end{equation}

Fix an arbitrary index $k\in B\cap C$ and define
\begin{equation}\label{eq:def S}
S\eqdef \{j\in \{1,\ldots,n\}:\ \|x_k-x_j\|_X\le r\}.
\end{equation}
It follows from~\eqref{eq;diam of intersection} that $S\supseteq B\cap C$, and therefore by~\eqref{eq:intersection big} we have
\begin{equation}\label{eq:S cardinality}
|S|\ge \frac{n}{4}.
\end{equation}
Moreover, by the definition~\eqref{eq:def S},
\begin{equation}\label{eq:sum on S}
\frac{1}{n}\sum_{i\in S} \|x_i-x_k\|_X^q\le \frac{|S|r^q}{n}.
\end{equation}

Now, define for every $i\in \{1,\ldots,n\}$,
\begin{equation}\label{eq:def vi}
v_i\eqdef \max\left\{0,\|x_i-x_k\|_X-r\right\}\in [0,\infty).
\end{equation}
Then for every $i\in \{1,\ldots,n\}\setminus S$ and $j\in S$,
\begin{eqnarray}\label{eq:to sum S and complement}
\nonumber\|x_i-x_k\|_X^q&\le& 2^{q-1} \left(\|x_i-x_k\|_X-r\right)^q+2^{q-1}r^q\\
&\stackrel{\eqref{eq:def vi}}{=}&2^{q-1}|v_i-v_j|^q+2^{q-1}r^q.
\end{eqnarray}
Hence,
\begin{align}\label{real poin}
\nonumber&\frac{1}{n}\sum_{i\in \{1,\ldots,n\}\setminus S} \|x_i-x_k\|_X^q\\\nonumber&\stackrel{\eqref{eq:to sum S and complement}}{\le} \frac{2^{q-1}}{n|S|}\sum_{i=1}^n\sum_{j=1}^n |v_i-v_j|^q+\frac{2^{q-1}(n-|S|)r^q}{n}\\\nonumber
&\stackrel{\eqref{eq:S cardinality}}{\le} \frac{2^{q+1}\gamma(A,d_\R^q)}{n}\sum_{i=1}^n\sum_{j=1}^n a_{ij} |v_i-v_j|^q+\frac{2^{q-1}(n-|S|)r^q}{n}\\&\ \le 2^{q+1}\gamma+\frac{2^{q-1}(n-|S|)r^q}{n},
\end{align}
where the last step of~\eqref{eq:S cardinality} uses the trivial fact $\gamma(A,d_\R^q)\le \gamma$ (since $Y$ contains an isometric copy of $\R$) and
$$
\frac{1}{n}\sum_{i=1}^n\sum_{j=1}^n a_{ij} |v_i-v_j|^q\stackrel{\eqref{eq:def vi}}{\le} \frac{1}{n}\sum_{i=1}^n\sum_{j=1}^n a_{ij} \|x_i-x_j\|_X^q\stackrel{\eqref{eq:equal to 1 assumption}}{=}1.
$$

A combination of~\eqref{eq:sum on S} and~\eqref{real poin} yields the estimate
$$
\frac{1}{n}\sum_{i=1}^n \|x_i-x_k\|_X^q\le 2^{q+1}\gamma+2^{q-1}r^q.
$$
Consequently, an application of Markov's inequality shows that
$$
\left|\left\{i\in \{1,\ldots,n\}:\ \|x_i-x_k\|_X\le 2^{1/q}\left(2^{q+1}\gamma+2^{q-1}r^q\right)^{1/q}\right\}\right|\ge \frac{n}{2}.
$$
Recalling the definition of $r_k$ (see~\eqref{eq: rj def}), it follows that
\begin{equation}\label{eq:r1 upper first}
r_1\stackrel{\eqref{eq:r1 minimal}}{\le} r_k\le 2^{1/q}\left(2^{q+1}\gamma+2^{q-1}r^q\right)^{1/q}.
\end{equation}
If $(2r)^q\le r_1^q/2$ then it follows from~\eqref{eq:r1 upper first} that
$
r_1^q\le 2^{q+3}\gamma$, implying the desired estimate~\eqref{eq:minimal radius bound}. So, suppose that $(2r)^q>r_1^q/2$, which, by recalling the definition of $r$ in~\eqref{eq;diam of intersection}, is the same as
$$
\frac{r_1^q}{2}< \left(2r_1\alpha^{-1}\left(2(4\gamma)^{1/q}\beta\left(\frac{2}{r_1}\right)\right)\right)^q,
$$
or
$$
\alpha\left(\frac14\right)\le \alpha\left(\frac{1}{2^{1+1/q}}\right)\le 2(4\gamma)^{1/q}\beta\left(\frac{2}{r_1}\right)\le 8\gamma^{1/p}\beta\left(\frac{2}{r_1}\right),
$$
implying the validity of~\eqref{eq:minimal radius bound} in this case as well.
\end{proof}

\begin{proof}[Proof of Theorem~\ref{thm:ozawa-ineq}]
We continue to use the notation that was introduced in the statement and the proof of Lemma~\ref{lem:minimal radius}. In particular, the assumptions~\eqref{eq:equal to 1 assumption} and~\eqref{eq:r1 minimal} are (without loss of generality) still in force.

Recalling the definition of $B$ in~\eqref{eq:def B}, we have
\begin{equation}\label{eq:sum on B}
\frac{1}{n}\sum_{i\in B} \|x_i-x_1\|_X^q\le \frac{r_1^q|B|}{n}.
\end{equation}
For $i\in \{1,\ldots,n\}$ define
\begin{equation}\label{eq:def ui}
u_i\eqdef \max\{0,\|x_i-x_1\|_X-r_1\}\in [0,\infty).
\end{equation}
Then for every $i\in \{1,\ldots,n\}\setminus B$ and $j\in B$ we have
\begin{eqnarray*}
\nonumber\|x_i-x_1\|_X^q&\le& 2^{q-1} \left(\|x_i-x_1\|_X-r_1\right)^q+2^{q-1}r_1^q\\
&\stackrel{\eqref{eq:def ui}}{=}&2^{q-1}|u_i-u_j|^q+2^{q-1}r_1^q.
\end{eqnarray*}
Arguing exactly as in~\eqref{real poin} (with the use of~\eqref{eq:S cardinality} replaced by the use of~\eqref{eq: B big}), it follows that
\begin{equation}\label{eq:sum off B}
\frac{1}{n}\sum_{i\in \{1,\ldots,n\}\setminus B}\|x_i-x_1\|_X^q\le 2^{q}\gamma+\frac{2^{q-1}(n-|B|)r_1^q}{n}.
\end{equation}
Consequently,
\begin{multline}\label{eq:gamma plus r1}
\frac{1}{n^2}\sum_{i=1}^n\sum_{j=1}^n \|x_i-x_j\|_X^q\le \frac{2^{q-1}}{n^2}\sum_{i=1}^n\sum_{j=1}^n \left(\|x_i-x_1\|_X^q +\|x_1-x_j\|_X^q\right)\\= \frac{2^q}{n}\sum_{i=1}^n \|x_i-x_1\|_X^q\stackrel{\eqref{eq:sum on B}\wedge \eqref{eq:sum off B}}{\le} 4^{q}\gamma+ \frac{4^qr_1^q}{2}.
\end{multline}
The desired estimate~\eqref{eq:ozawa ineq} now follows substituting~\eqref{eq:minimal radius bound} into~\eqref{eq:gamma plus r1}.
\end{proof}

\subsection{Limitations of Ozawa's method}\label{sec:no ozawa}

In the discussion immediately preceding the estimates~\eqref{eq:p<2} and~\eqref{eq:p>2} we stated that for every $p\in [1,\infty)$ there exists a mapping $\f:\ell_p\to \ell_2$ such that for every $x,y\in \ell_p$ we have
\begin{equation}\label{eq:mazur2 p>2}
p\in (2,\infty)\implies \frac{\|x-y\|_{\ell_p}^{p/2}}{2^{(p-2)/2}}\le \|\f(x)-\f(y)\|_{\ell_2}\le p\|x-y\|_{\ell_p},
\end{equation}
and
\begin{equation}\label{eq:mazur2 p<2}
p\in [1,2)\implies \frac{\|x-y\|_{\ell_p}}{3}\le \|\f(x)-\f(y)\|_{\ell_2}\le 2\|x-y\|_{\ell_p}^{p/2}.
\end{equation}
The estimates~\eqref{eq:mazur2 p>2} and~\eqref{eq:mazur2 p<2} are a special case of the following bounds on the modulus of uniform continuity of the Mazur map~\cite{Maz29} (see also~\cite[Ch.9]{BL}). Let $(\Omega,\mu)$ be a measure space and fix $p,q\in [1,\infty)$. Define $M_{p,q}:L_p(\mu)\to L_q(\mu)$ by
$$
\forall\, f\in L_p(\mu),\qquad M_{p,q}(f)\eqdef|f|^{p/q}\sign(f).
$$
If $p\ge q$ then for every $f,g\in L_p(\mu)$ with $\|f\|_{L_p(\mu)},\|g\|_{L_p(\mu)}\le 1$ we have
\begin{equation}\label{eq:Mazur pq}
\frac{\|f-g\|_{L_p(\mu)}^{p/q}}{2^{(p-q)/q}}\le \left\|M_{p,q}(f)-M_{p,q}(g)\right\|_{L_q(\mu)}\le \frac{2^{\frac{1}{q}-\frac{1}{p}}p\|f-g\|_{L_p(\mu)}}{q}.
\end{equation}
Note that~\eqref{eq:mazur2 p>2} is a special case of~\eqref{eq:Mazur pq}, and~\eqref{eq:mazur2 p<2} is also a consequence of~\eqref{eq:Mazur pq} because $M_{p,q}^{-1}=M_{q,p}$.  While the bounds appearing in~\eqref{eq:Mazur pq} are entirely standard, they seem to have been always  stated in the literature while either using implicit multiplicative constant factors, or with suboptimal constant factors. These constants play a role in our context, so we briefly include the proof of~\eqref{eq:Mazur pq}, following the lines of the proof of~\cite[Prop.~9.2]{BL}. The elementary inequality
$$
\left||u|^\theta\sign(u)-|v|^\theta\sign(v)\right|\ge \frac{|u-v|^\theta}{2^{\theta-1}},
$$
which holds for every $u,v\in \R$  and $\theta\in [1,\infty)$, immediately implies (with $\theta=p/q$) the leftmost inequality in~\eqref{eq:Mazur pq}. To prove the rightmost inequality of~\eqref{eq:Mazur pq}, note the following elementary inequality, which also holds for every $u,v\in \R$  and $\theta\in [1,\infty)$.
$$
\left||u|^\theta\sign(u)-|v|^\theta\sign(v)\right|\le \theta|u-v|\max\left\{|u|^{\theta-1},|v|^{\theta-1}\right\}.
$$
Consequently,
\begin{align}
\nonumber &\left\|M_{p,q}(f)-M_{p,q}(g)\right\|_{L_q(\mu)}^q\\\nonumber&\le \frac{p^q}{q^q}\int_\Omega |f-g|^q\max\left\{|f|,|g|\right\}^{p-q}d\mu\\
&\le \frac{p^q}{q^q}\|f-g\|_{L_p(\mu)}^q\cdot \left\|\max\{|f|,|g|\}\right\|_{L_p(\mu)}^{p-q}\label{eq:holder p/q}\\
&\le \frac{p^q\cdot 2^{(p-q)/p}}{q^q}\|f-g\|_{L_p(\mu)}^q\label{eq;norms are less than 1},
\end{align}
where~\eqref{eq:holder p/q} follows from an application of H\"older's inequality with exponents $p/q$ and $p/(p-q)$ and~\eqref{eq;norms are less than 1} holds true because we have $\|\max\{|f|,|g|\}\|_{L_p(\mu)}^p\le \|f\|_{L_p(\mu)}^p+\|g\|_{L_p(\mu)}^p\le 2$.

Returning to Theorem~\ref{thm:ozawa-ineq} (in particular using the notation and assumptions that were introduced in the statement of Theorem~\ref{thm:ozawa-ineq}), if one wants the bound~\eqref{eq:ozawa ineq} to be compatible with the assumption of Theorem~\ref{thm:duality} one needs~\eqref{eq:ozawa ineq} to yield an upper bound on $\gamma(A,\|\cdot\|_X^q)$ that grows linearly with $\gamma(A,\|\cdot\|_Y^q)$. This is equivalent to the requirement
\begin{equation}\label{eq:beta Kt requirement}
\beta^{-1}\left(\frac{\alpha(1/4)}{8\gamma\!\left(A,\|\cdot\|_Y^q\right)^{1/q}}\right)^q\gtrsim_{X,Y} \frac{1}{\gamma\!\left(A,\|\cdot\|_Y^q\right)}.
\end{equation}
Since~\eqref{eq:beta Kt requirement} is supposed to hold for every $n\in \N$ and every $n$ by $n$ symmetric stochastic matrix $A$, \eqref{eq:beta Kt requirement} is the same as requiring that $\beta(t)\lesssim_{X,Y}t$ for every $t\in (0,\infty]$.

Specializing the above discussion to $Y=\ell_2$ and $q=2$, if $\beta(t)\le Kt$ for some $K\in (0,\infty)$ and every $t\in (0,\infty)$ then~\eqref{eq:ozawa ineq} yields the estimate
$$
\gamma\!\left(A,\|\cdot\|_X^q\right)\lesssim \frac{(K/\alpha(1/4))^2}{1-\lambda_2(A)}.
$$
By Corollary~\ref{coro:linear gap}, this implies that $$Av_{\ell_2}^{(2)}(X)\lesssim \frac{K}{\alpha(1/4)}.$$ In particular, if $p\in (2,\infty)$ then due to~\eqref{eq:mazur2 p>2} we get the estimate
$$A_{\ell_2}^{(2)}(\ell_p)\lesssim p2^{3p/2},$$ which is exponentially worse than~\eqref{eq:pAv}.  The following lemma shows that this exponential loss is inherent to the use of Theorem~\ref{thm:ozawa-ineq} for the purpose of obtaining average distortion embedding of finite subsets of $\ell_p$ into $\ell_2$, i.e., that $K/\alpha(1/4)$ must grow exponentially in $p$ as $p\to \infty$.

\begin{lemma}
Suppose that $p\in (2,\infty)$ and that $\f:B_X\to \ell_2$ satisfies
\begin{equation}\label{eq:assuming alpha in lemma}
 \alpha\left(\|x-y\|_{\ell_p}\right)\le
\|\f(x)-\f(y)\|_{\ell_2}\le K\|x-y\|_{\ell_p},
\end{equation}
for every $x,y\in B_{\ell_p}$, where $\alpha:(0,2]\to (0,\infty)$ is
increasing. Then
\begin{equation}\label{eq:exponential in p lower}
\forall\, \e\in (0,2),\qquad \frac{K}{\alpha(2-\e)}\gtrsim \frac{1}{(1-\e/2)^{p/2}}.
\end{equation}
In particular, for $\e=7/4$ we have $K/\alpha(1/4)\gtrsim 2^{3p/2}$.
\end{lemma}

\begin{proof}
Fix $m,n\in \N$ and $s\in (0,1]$. For every $x\in \Z^n$ define
$\psi(x)\in \ell_p$ to be the vector whose $j$th coordinate equals
$$\frac{s}{n^{1/p}} e^{2\pi x_j i/m}$$ if $j\in \n$, and whose remaining
coordinates vanish. Then we have $\|\psi(x)\|_{\ell_p}\le s\le 1$
for every $x\in \Z^n$.

By the results of Section~3 of~\cite{MN-cotype}, if is $m$ divisible
by $4$ and  $m\ge \frac23 \pi\sqrt{n}$ then we have
\begin{multline}\label{eq:cotype hilber quote}
\frac{1}{m^n} \sum_{j=1}^n \sum_{x\in \{0,\ldots,m-1\}^n}
\left\|\f\left(\psi\left(x+\frac{m}{2}e_j\right)\right)-\f(\psi(x))\right\|_{\ell_2}^2\\\lesssim
\frac{m^2}{(3m)^n} \sum_{\omega\in \{-1,0,1\}^n}\sum_{x\in
\{0,\ldots,m-1\}^n}\left\|\f\left(\psi(x+\omega)\right)-\f(\psi(x))\right\|_{\ell_2}^2.
\end{multline}
Now, by the leftmost inequality in~\eqref{eq:assuming alpha in
lemma} for every $x\in \{0,\ldots,m-1\}^n$ and $j\in \n$ we have
\begin{align}\label{eq:lower phi}
\nonumber\left\|\f\left(\psi\left(x+\frac{m}{2}e_j\right)\right)-\f(\psi(x))\right\|_{\ell_2}&\ge
\alpha\left(\frac{s\left|e^{2\pi (x_j+m/2)i/m}-e^{2\pi x_j i/m}\right|}{n^{1/p}}\right)\\
&=\alpha\left(\frac{2s}{n^{1/p}}\right).
\end{align}
Also, by the rightmost inequality in~\eqref{eq:assuming alpha in
lemma}, for every $x\in \{0,\ldots,m-1\}^n$ and $\omega\in
\{-1,0,1\}^n$ we have
\begin{align}\label{eq:upper phi}
\nonumber\left\|\f\left(\psi(x+\omega)\right)-\f(\psi(x))\right\|_{\ell_2}&\le
\frac{Ks}{n^{1/p}}\left(\sum_{j=1}^n \left|e^{2\pi (x_j+\omega_j)i/m}-e^{2\pi x_j i/m}\right|^p\right)^{\frac{1}{p}}\\
&\lesssim \frac{Ks}{m}.
\end{align}
Choose $n\eqdef \lfloor 1/(1-\e/2)^p\rfloor $ and $s\eqdef
(1-\e/2)n^{1/p}\in (0,1]$. Then, if $m$ is the smallest integer that
is divisible by $4$ and satisfies $m\ge \frac23 \pi\sqrt{n}$, by
substituting~\eqref{eq:upper phi} and~\eqref{eq:lower phi}
into~\eqref{eq:cotype hilber quote} we see that
$$
n\alpha\left(2-\e\right)^2\lesssim m^2\cdot \frac{K^2s^2}{m^2}\lesssim n
\frac{K^2(1-\e/2)^2}{n^{1-2/p}}\lesssim nK^2(1-\e/2)^{p},
$$
which simplifies to give the desired estimate~\eqref{eq:exponential
in p lower}.
\end{proof}

\section{Bourgain--Milman--Wolfson type}\label{sec:type}

Here we study aspects of nonlinear type in the sense of Bourgain,
Milman and Wolfson~\cite{BMW}, proving in particular
Lemma~\ref{lem:no cube}, Lemma~\ref{lem:av type} and
Theorem~\ref{thm:our MP intro} that were stated in the Introduction.

\subsection{On the Maurey--Pisier problem for BMW type}\label{sec:MP
actual} In what follows, for every finite set $\Omega$ and $p\in
[1,\infty)$, the space $L_p(\Omega)$ consists of all mappings
$f:\Omega\to \R$, equipped with the norm
$$
\|f\|_{L_p(\Omega)}\eqdef \left(\frac{1}{|\Omega|}\sum_{\omega\in \Omega} |f(\omega)|^p\right)^{\frac{1}{p}}.
$$

For every $k\in \n$ let $\Omega_k^n\subset \F_2^n\times 2^{\F_2^n}$
be defined by
$$
\Omega_k^n\eqdef \F_2^n\times \binom{\n}{k}=\F_2^n\times \left\{I\subset\n:\ |I|=k\right\}.
$$
Thus $|\Omega_k^n|=2^n\binom{n}{k}$.

Fixing a metric space $(X,d_X)$ and $f:\F_2^n\to X$, define
\begin{equation}\label{eq:def I directional}
\forall(z,I)\in \F_2^n\times
2^{\F_2^n},\qquad \D_f(z,I)\eqdef d_X\left(f\left(z+e_I\right),f(z)\right),
\end{equation}
where for $I\in \n$ we set $$e_I\eqdef \sum_{i\in I} e_i\in
\F_2^n.$$ Thus, using the notation of the Introduction, we have
$e_{\n}=e$.

For $q\in (0,\infty)$ define $E_k^{(q)}(f)\in [0,\infty)$ by
\begin{align}\label{eq:def k-edge average}
\nonumber E_k^{(q)}(f)^q&\eqdef \left\|\D_f\right\|_{L_q(\Omega_k^n)}^q\\ &=
\frac{1}{2^n\binom{n}{k}}\sum_{\substack{I\subseteq \n\\|I|=k}}
\sum_{z\in \F_2^n}d_X\left(f\left(z+e_I\right),f(z)\right)^q.
\end{align}
Note that since for every $I,J\subset \n$ with $|I|=|J|$ the number
of permutations $\sigma\in S_n$ satisfying $\sigma(I)=J$ equals
$|I|!(n-|I|)!$,
\begin{equation}\label{eq:fixed I permutation}
\forall\, I\subset \n,
\qquad E_{|I|}^{(q)}(f)^q=\frac{1}{2^nn!}\sum_{x\in \F_2^n}\sum_{\sigma\in S_n}\D_f(x,\sigma(I))^q.
\end{equation}
We also record for future use the following simple consequence of
the triangle inequality in $(X,d_X)$.
\begin{lemma}\label{lem:E subadditivity}
Fix $n\in \N$ and $q\in [1,\infty)$. Suppose that $(X,d_X)$ is a
metric space and $f:\F_2^n\to X$. Then for every $k,m\in \n$ with
$k+m\le n$ we have
\begin{equation}\label{eq:E subadditivity in lemma}
E^{(q)}_{k+m}(f)\le E^{(q)}_k(f)+E^{(q)}_m(f).
\end{equation}
\end{lemma}

\begin{proof}
Fix $I\subset \n$ with $|I|=k+m$. Recalling~\eqref{eq:def I
directional}, for every $J\subset I$ with $|J|=k$ and every $z\in
\F_2^n$ we have
$$
\D_f(z,I)\le\D_f(z,I\setminus J)+\D_f(z+e_{I\setminus J},J).
$$
Consequently,
\begin{equation}\label{eq:pointwise z,I}
\forall (z,I)\in \Omega_{k+m}^n,\qquad \D_f(z,I)\le U_f(z,I)+V_f(z,I),
\end{equation}
where
$$
U_f(z,I)\eqdef \frac{1}{\binom{k+m}{k}}\sum_{\substack{J\subset I\\ |J|=k}}\D_f(z,I\setminus J).
$$
and
$$
V_f(z,I)\eqdef
\frac{1}{\binom{k+m}{k}}\sum_{\substack{J\subset I\\
|J|=k}}\D_f(z+e_{I\setminus J},J).
$$

 The point-wise estimate~\eqref{eq:pointwise z,I} combined with the
triangle inequality in $L_q(\Omega_{k+m}^n)$ implies that
\begin{equation}\label{eq:E bounded by U+V}
E^{(q)}_{k+m}(f)=\left\|\D_f\right\|_{L_q(\Omega_{k+m}^n)}\le \left\|U_f\right\|_{L_q(\Omega_{k+m}^n)}
+\left\|V_f\right\|_{L_q(\Omega_{k+m}^n)}.
\end{equation}
By the convexity of $t\mapsto |t|^q$ we have
\begin{multline}\label{eq:Uf bound}
\left\|U_f\right\|_{L_q(\Omega_{k+m}^n)}^q\le
\frac{1}{2^n\binom{n}{k+m}} \sum_{\substack{I\subseteq \n\\|I|=k+m}}
\sum_{z\in \F_2^n}\frac{1}{\binom{k+m}{k}}\sum_{\substack{J\subset I\\ |J|=k}}\D_f(z,I\setminus J)^q\\
= \frac{1}{2^n\binom{n}{k+m}\binom{k+m}{k}}\sum_{z\in
\F_2^n}\sum_{\substack{S\subseteq \n\\|S|=m}}
\binom{n-m}{k}\D_f(z,S)^q=E_m^{(q)}(f)^q,
\end{multline}
and
\begin{multline}\label{eq:Vf bound}
\left\|V_f\right\|_{L_q(\Omega_{k+m}^n)}^q\le
\frac{1}{2^n\binom{n}{k+m}} \sum_{\substack{I\subseteq \n\\|I|=k+m}}
\sum_{z\in \F_2^n}\frac{1}{\binom{k+m}{k}}\sum_{\substack{J\subset
I\\ |J|=k}} \D_f(z+e_{I\setminus J},J)^q\\
=
\frac{1}{2^n\binom{n}{k+m}\binom{k+m}{k}}\sum_{\substack{I\subseteq
\n\\|I|=k+m}}\sum_{\substack{J\subset I\\ |J|=k}} \sum_{w\in
\F_2^n}\D_f(w,J)^q=E_k^{(q)}(f)^q.
\end{multline}
The desired estimate~\eqref{eq:E subadditivity in lemma} now follows
from a substitution of~\eqref{eq:Uf bound} and~\eqref{eq:Vf bound}
into~\eqref{eq:E bounded by U+V}.
\end{proof}

For $m,n\in \N$ with $n\ge m$, for every $f:\F_2^m\to X$ denote its
natural lifting to $\F_2^n$ by $f^{\uparrow n}:\F_2^n\to X$, that
is,
$$
\forall\, z\in \F_2^n,\qquad f^{\uparrow n}(z)\eqdef f(z_1,\ldots,z_m).
$$
We then have the following identity for every $k\in \n$.
\begin{align}
E_k^{(q)}\left(f^{\uparrow n}\right)^q & =\frac{1}{2^n\binom{n}{k}}\sum_{\substack{I\subseteq \n\\|I|=k}}
2^{n-m}\sum_{w\in \F_2^m}\D_f(w,I\cap \{1,\ldots,m\})^q \label{eq:use def E and D}
\\
&=\frac{1}{\binom{n}{k}}\sum_{\ell=\max\{1,m+k-n\}}^{\min\{k,m\}}
\binom{n-m}{k-\ell}\binom{m}{\ell}E_\ell^{(q)}(f)^q, \label{eq:E under lifting}
\end{align}
where~\eqref{eq:use def E and D} uses~\eqref{eq:def I directional}
and~\eqref{eq:def k-edge average}, and the identity~\eqref{eq:E
under lifting} follows by observing that if $I\subset\n$ satisfies
$|I\cap\{1,\ldots,m\}|=\ell$ then necessarily $\ell\ge m+k-n$ and
for each $J\subseteq \{1,\ldots,m\}$ with $|J|=\ell$ the number of
subsets $I\subset\n$ with $I\cap\{1,\ldots,m\}=J$ is
$\binom{n-m}{k-\ell}$. Note in particular the following two special cases
of~\eqref{eq:E under lifting}.
\begin{equation}\label{lifting edges diagonals}
E_n^{(q)}\left(f^{\uparrow n}\right)= E_n^{(q)}(f)\qquad\mathrm{and}\qquad E_1^{(q)}\left(f^{\uparrow n}\right)=\left(\frac{m}{n}\right)^{1/q}E_1^{(q)}(f).
\end{equation}

When $q=2$ in~\eqref{eq:def k-edge average} we write $E_k(f)\eqdef
E_k^{(2)}(f)$. With this notation, a metric space $(X,d_X)$ has BMW
type $p\in (0,\infty)$ if and only if there exists $T\in (0,\infty)$
such that for every $n\in \N$ and every $f:\F_2^n\to X$,
\begin{equation}\label{eq:BMW E notation}
E_n(f)\le Tn^{1/p}E_1(f).
\end{equation}
Let $\BMW_p^n(X)$ denote the infimum over those $T\in (0,\infty)$
for which~\eqref{eq:BMW E notation} holds true for every
$f:\F_2^n\to X$. Thus
$$
\BMW_p(X)=\sup_{n\in \N} \BMW_p^n(X).
$$

\begin{remark}
By definition we have $\BMW_p^n(X)n^{1/p}=\BMW_q^n(X)n^{1/q}$ for
every $p,q\in (0,\infty)$. Also, unless $|X|=1$ we have
$\BMW_p^1(X)=1$.
%We shall always tacitly assume that metric spaces
%are not singletons.
\end{remark}

\begin{remark}By Lemma~\ref{lem:E subadditivity} we have $\E_n(f)\le
nE_1(f)$, so
$$
\forall\, n\in \N,\qquad \BMW_p^n(X)\le n^{1-\frac{1}{p}}.
$$
Moreover, given $m,n\in \N$ with $n\ge m$ and $f:\F_2^n\to X$ with
$E_1(f)>0$, by~\eqref{lifting edges diagonals} we have
$$
\frac{E_m(f)}{E_1(f)}=\frac{E_n\left(f^{\uparrow n}\right)}{E_1\left(f^{\uparrow n}\right)}
\left(\frac{n}{m}\right)^{1/p}.
$$
Consequently,
\begin{equation}\label{eq:monotonicity bmw}
\forall\, m,n\in \N,\quad m\le n\implies \BMW_p^m(X)m^{1/p}\le \BMW_p^n(X)n^{1/p}.
\end{equation}
\end{remark}

\begin{remark}\label{rem:p<2 always}
Let $(X,d_X)$ be a metric space and $p\in (1,\infty)$ be such that
$X$ has BMW type $p$. Then necessarily $p\le 2$. Indeed, choose
distinct $x_0,x_1\in X$ and for every $n\in\N$ define $f_n:\F_n\to
X$ by $f_n(z)=x_{z_1}$. Then $E_1(f_n)=d_X(x_0,x_1)/\sqrt{n}$
and $E_n(f_n)=d_X(x_0,x_1)$. This means that $n^{1/p}\BMW_p(X)\ge
\sqrt{n}$ for all $n\in \N$, which implies that $p\le 2$. A
straightforward application of the triangle inequality
(see~\cite{BMW}) implies that every metric space has BMW type $1$
with $\BMW_1(X)=1$.
\end{remark}

Recalling the definition of $p_X\in [1,2]$ in~\eqref{eq:def p_X
BMW}, we have the following lemma that relies on a sub-multiplicativity
argument that was introduced by Pisier~\cite{Pisier-type-1} in the
context of Rademacher type of normed spaces, and has been
implemented in the context of nonlinear type by Bourgain, Milman and
Wolfson~\cite{BMW} (see also~\cite{Pisier-type}).

\begin{lemma}\label{lem:use sub-multiplicativity}
For every metric space $(X,d_X)$ we have
$$
\forall\, n\in \N,\qquad \BMW_{p_X}^n(X)\ge 1.
$$
\end{lemma}

\begin{proof} Write $p=p_X$ and suppose for the sake of obtaining a contradiction that there exists
$m\in \N$ and $\e\in (0,1)$ such that $\BMW_p^m(X)<\e$. We may also
assume without loss of generality that $\e>1/n^{1/p}$. Since
$\BMW_p^1(X)=1$, we have $m\ge 2$. If we define
\begin{equation}\label{eq:def q bigger than p}
q\eqdef \frac{1}{\frac{1}{p}-\frac{\log(1/\e)}{\log m}}
\end{equation}
then $q\in (p,\infty)$. By~\cite[Lem.~2.3]{BMW} (see
also~\cite[Lem.7.2]{Pisier-type}), for every $k,n\in \N$ we have
$$
\BMW_p^{kn}(X)\le \BMW_p^{k}(X)\cdot \BMW_p^{n}(X).
$$
Consequently, for every $i\in \N$ we have
\begin{equation}\label{eq:at m^i}
\BMW_p^{m^i}(X)\le \BMW_p^m(X)^i<\e^i\stackrel{\eqref{eq:def q bigger than p}}{=}
m^{-i\left(\frac{1}{p}-\frac{1}{q}\right)}.
\end{equation}
For every $n\in \N$ choose $i\in \N$ such that $m^{i-1}\le n<m^i$.
Since, by~\eqref{eq:monotonicity bmw}, $\BMW_p^n(X)n^{1/p}$
increases with $n$, it follows from~\eqref{eq:at m^i} that
$$
\BMW_q^n(X)=\frac{\BMW_p^n(X)n^{1/p}}{n^{1/q}}\le \frac{\BMW_p^{m^i}(X)m^{i/p}}{n^{1/q}}<\frac{m^{i/q}}{n^{1/q}}\le m^{1/q}.
$$
Consequently $\BMW_q(X)<\infty$, i.e., $X$ has BMW type $q$. Since
$q>p$, this contradicts the definition of $p_X=p$.
\end{proof}

\begin{lemma}\label{lem:partition ki}
Fix $p\in [1,2]$, $n\in \N$, and $k_1,\ldots,k_m\in \n$ such that
$k_1+\ldots +k_m\le n$. Then for every metric space $(X,d_X)$ and
every $f:\F_2^n\to X$ we have
\begin{equation}\label{eq:index partition in lemma}
 E_{k_1+\ldots+k_m}(f)\le
\BMW_p(X)m^{\frac{1}{p}-\frac12}\left(\sum_{j=1}^m
E_{k_j}(f)^2\right)^{\frac12}.
\end{equation}
\end{lemma}

\begin{proof} Write $k_0=0$ and $I=\{1,\ldots, k_1+\ldots+k_m\}\subset \n$. For every
$j\in \m$ set
$I_j=\{k_1+\ldots+k_{j-1}+1,\ldots,k_1+\ldots+k_{j}\}\subset I$.
Fixing $x\in \F_2^n$  and a permutation $\sigma\in S_n$, define
$\f_x^{\sigma}:\F_2^m\to X$ by
$$
\f_x^\sigma(z)\eqdef f\left(x+\sum_{s=1}^m z_s e_{\sigma(I_s)}\right).
$$
An application of the definition of $\BMW_p(X)$ to $\phi_x^\sigma$
yields the inequality
\begin{multline}\label{use BMW partition}
\frac{1}{2^m}\sum_{z\in \F_2^m}\D_f\left(x+\sum_{s=1}^m z_s
e_{\sigma(I_s)},\sigma(I)\right)^2\\\le
\frac{\BMW_p(X)^2m^{\frac{2}{p}-1}}{2^m}\sum_{z\in
\F_2^m}\sum_{j=1}^m \D_f\left(x+\sum_{s=1}^m z_s
e_{\sigma(I_s)},\sigma(I_j)\right)^2.
\end{multline}
Recalling~\eqref{eq:fixed I permutation}, by averaging~\eqref{use
BMW partition} over $x\in \F_2^n$ and $\sigma\in S_{n}$ we obtain
\begin{align*}
E_{k_1+\ldots+k_m}(f)^2&=\frac{1}{2^nn!}\sum_{y\in \F_2^n} \sum_{\sigma\in S_n}\D_f\left(y,\sigma(I)\right)^2
\\&\le
\frac{\BMW_p(X)^2m^{\frac{2}{p}-1}}{2^nn!} \sum_{j=1}^m\sum_{y\in \F_2^n}
\sum_{\sigma\in S_n} \D_f(y,\sigma(I_j))^2\\&=
\BMW_p(X)^2m^{\frac{2}{p}-1}\sum_{j=1}^m E_{k_j}(f)^2.\qedhere
\end{align*}
\end{proof}

%\begin{theorem} For every metric space $(X,d_X)$ and every $d\in \N$ there exists $N\in \N$ such that
%$$
% c_{\ell_2^N(X)}\left(\F_2^d,\|x-y\|_{p_X}\right)\le \mathrm{BMW}_{p_X}(X)^2.
%$$
%\end{theorem}

\begin{proof}[Proof of Theorem~\ref{thm:our MP intro}] Denote $p_X=p$. Fix $\e\in (0,1/3)$ and define
\begin{equation}\label{eq:n choice as function of d}
n\eqdef \left\lceil\frac{\BMW_p(X)^4}{\e}d\right\rceil.
\end{equation}
Since $n\ge d$, we can consider $\F_p^d$ as a subset of $\F_p^n$
(say, cannonically embedded as the first $d$ coordinates).

By Lemma~\ref{lem:use sub-multiplicativity} we have $\BMW_p^n(X)\ge
1$, and therefore by the definition of $\BMW_p^n(X)$ there exists
$f:\F_p^n\to X$ such that
\begin{equation}\label{eq:f maximizer type}
E_n(f) \ge (1-\e)n^{1/p}E_1(f)>0.
\end{equation}
Write
$$
Y\eqdef \ell_2^{\F_2^n\times S_n}(X)\cong \ell_2^{2^nn!}(X).
$$
Define $F:\F_2^n\to Y$ by setting
$$
\forall (x,z,\pi)\in \F_2^n\times \F_2^n\times S_n,\qquad F(x)_{(z,\pi)}\eqdef
f(\pi(x)+z),
$$
where $\pi(x)\eqdef \sum_{i=1}^n x_{\pi(i)}e_i$.
Recalling~\eqref{eq:def k-edge average}, every $x,y\in
\F_2^n$ satisfy
\begin{equation}\label{eq:F distance constant on levels}
d_Y(F(x),F(y))= \sqrt{2^n n!}\cdot E_{\|x-y\|_1}(f),
\end{equation}

By Lemma~\ref{lem:partition ki} with $m=\|x-y\|_1$ and
$k_1=\ldots=k_m=1$, for every $x,y\in \F_p^n$ we have
\begin{eqnarray}\label{eq:upper dY}
\frac{d_Y(F(x),F(y))}{\sqrt{2^nn!}}\!\!\!&\stackrel{\eqref{eq:index partition in lemma}\wedge \eqref{eq:F distance constant on levels}}
{\le}&\!\!\!
\BMW_p(X)\|x-y\|_1^{\frac{1}{p}-\frac12}\cdot
\sqrt{\|x-y\|_1}E_1(f)\nonumber \\&=&\BMW_p(X)E_1(f)\|x-y\|_p.
\end{eqnarray}
Fixing $x,y\in \F_p^d\subset \F_p^n$, write $n=a\|x-y\|_1+b$ for
appropriate integers $a$ and $b\in [0,\|x-y\|_1)$. By
Lemma~\ref{lem:partition ki} with $m=b$ and $k_1=\ldots=k_m=1$,
\begin{equation}\label{eq:Eb}
E_b(f)\le \BMW_p(X)b^{1/p}E_1(f)\le \BMW_p(X)\|x-y\|_1^{1/p}E_1(f).
\end{equation}
Using Lemma~\ref{lem:partition ki} once more, this time with
$m=a+1$, $k_{1}=b$ and $k_2=\ldots=k_{a+1}=\|x-y\|_1$, and noting
that since $\|x-y\|_1\le d\le \e n$ we have $m\le
(1+\e)n/\|x-y\|_1$, we conclude that
\begin{equation*}
E_n(f)\le
\BMW_p(X)\left(\frac{(1+\e)n}{\|x-y\|_1}\right)^{\frac{1}{p}-\frac12}\sqrt{\frac{n}{\|x-y\|_1}E_{\|x-y\|_1}(f)^2+E_b(f)^2}.
\end{equation*}
In combination with~\eqref{eq:Eb} and our assumption~\eqref{eq:f
maximizer type}, this implies
\begin{multline}\label{eq:to subtract}
(1-\e)^2n^{2/p}E_1(f)^2\le
(1+\e)\BMW_p(X)^2\left(\frac{n}{\|x-y\|_1}\right)^{2/p}E_{\|x-y\|_1}(f)^2\\+(1+\e)n^{2/p}
\BMW_p(X)^4\frac{\|x-y\|_1}{n}E_1(f)^2.
\end{multline}
Recalling~\eqref{eq:n choice as function of d},  we have
$\|x-y\|_1/n\le d/n\le \e \BMW_p(X)^4$ (since $x,y\in \F_2^d$), and
it therefore follows from~\eqref{eq:to subtract} that
\begin{equation}\label{eq:lower dY}
\frac{d_Y(F(x),F(y))}{\sqrt{2^n n!}}\stackrel{\eqref{eq:F distance constant on levels}}{=} E_{\|x-y\|_1}(f)\ge
\frac{\|x-y\|_pE_1(f)}{\BMW_p(X)}\sqrt{\frac{1-3\e}{1+\e}}.
\end{equation}
Since $\e\in (0,1/3)$ can be taken to be arbitrarily small, by
combining~\eqref{eq:upper dY} and~\eqref{eq:lower dY} we conclude
that
\begin{equation*}
c_Y(\F_2^d,\|x-y\|_p)\le \BMW_p(X)^2.\qedhere
\end{equation*}
\end{proof}
\subsection{Obstructions to average distortion embeddings of cubes}\label{sec:no
average} We start by proving Lemma~\ref{lem:av type}, whose proof is
very simple.

\begin{proof}[Proof of Lemma~\ref{lem:av type}]
Fix $D>Av_Y^{(2)}(X)$ and $n\in \N$. If $f:\F_2^n\to X$ then there
exists a nonconstant mapping $g:f(\F_2^n)\to Y$ such that
$$
\sum_{(x,y)\in \F_2^n\times F_2^n} d_Y(g(f(x)),g(f(y)))^2\ge \frac{\|g\|_{\Lip}^2}{D^2}
\sum_{(x,y)\in \F_2^n\times F_2^n}
d_X(f(x),f(y))^2.
$$
Consequently,
\begin{align}\label{eq:f circ g}
\nonumber&\frac{1}{2^n}\sum_{x\in \F_2^n}d_X(f(x),f(x+e))^2\\\nonumber&\le
\frac{2}{4^n}\sum_{(x,y)\in \F_2^n\times \F_2^n}\left(d_X(f(x),f(y))^2+d_X(f(y),f(x+e))^2\right)\\\nonumber&
=\frac{1}{4^{n-1}}\sum_{(x,y)\in \F_2^n\times F_2^n}
d_X(f(x),f(y))^2\\
&\le \frac{D^2}{4^{n-1}\|g\|_{\Lip}^2}\sum_{(x,y)\in \F_2^n\times F_2^n} d_Y(g(f(x)),g(f(y)))^2.
\end{align}
Now,
\begin{multline}\label{eq:BMW on each Ek}
\frac{1}{4^n}\sum_{(x,y)\in \F_2^n\times F_2^n}
d_Y(g(f(x)),g(f(y)))^2 \stackrel{\eqref{eq:def k-edge
average}}{=}\frac{1}{2^n}\sum_{k=1}^n \binom{n}{k}
E_k(g\circ f)^2\\
\stackrel{\eqref{eq:index partition in lemma}}{\le}
\frac{1}{2^n}\sum_{k=1}^n \binom{n}{k} \BMW_p(Y)^2k^{2/p} E_1(g\circ
f)^2.
\end{multline}
Since $E_1(g\circ f)\le \|g\|_{\Lip}E_1(f)$, it follows
from~\eqref{eq:f circ g} and~\eqref{eq:BMW on each Ek} that
$\BMW_p(X)\le 2Av_Y^{(2)}(X)\BMW_p(Y)$.
\end{proof}

Recall that (see e.g.~\cite{Mau03}) a Banach space $(X,\|\cdot\|_X)$
is said to have Rademacher type $p$ constant $T\in (0,\infty)$ if
for every $m\in \N$ and every $x_1,\ldots,x_m\in X$ we have
\begin{equation}\label{eq:rad type}
\left(\E_\e\left[\left\|\sum_{i=1}^m\e_ix_i\right\|_X^2\right]\right)^{\frac12}
\le T\left(\sum_{i=1}^m \|x_i\|_X^p\right)^{\frac{1}{p}},
\end{equation}
where $\E_\e[\cdot]$ is the expectation with respect to i.i.d.
 $\pm 1$ Bernoulli random variables $\e_1,\ldots,\e_n$. The
infimum over those $T\in (0,\infty)$ for which $X$ has Rademacher
type $p$ constant $T$ is denoted $T_p(X)$. If no such $T$ exists
then we write $T_p(X)=\infty$.

\begin{lemma}\label{lem:use Pisier ineq}
Assume that $p\in [1,2)$ and fix $q\in (p,2]$ and $r,s\in
[1,\infty)$. Suppose that $(Y,\|\cdot\|_Y)$ is a Banach space with
$T_q(Y)<\infty$ and that $f:\F_2^m\to Y$ satisfies
\begin{multline}\label{eq:f assumption pisier ineq}
\left(\frac{1}{4^m}\sum_{(x,y)\in \F_2^m\times
\F_2^m}\|f(x)-f(y)\|_Y^r\right)^{\frac{1}{r}}\\=
\left(\frac{1}{4^m}\sum_{(x,y)\in \F_2^m\times
\F_2^m}\|x-y\|_{\ell_p}^s\right)^{\frac{1}{s}}\asymp m^{1/p}.
\end{multline}
Then there exists $x\in \F_2^m$ and $i\in \n$ such that
\begin{equation}\label{eq:lower average cube with logm term}
\|f(x)-f(x+e_i)\|_Y \gtrsim \frac{1}{\sqrt{r}T_q(Y)}\cdot \frac{m^{\frac{1}{p}-\frac{1}{q}}}{\log m}.
\end{equation}
\end{lemma}

\begin{proof}
By Pisier's inequality~\cite{Pisier-type} we have
\begin{multline}\label{eq:pisier's ineq first use}
 m^{1/p}\stackrel{\eqref{eq:f assumption pisier
ineq}}\lesssim \left(\frac{1}{4^m}\sum_{(x,y)\in \F_2^m\times
\F_2^m}\|f(x)-f(y)\|_Y^r\right)^{\frac{1}{r}}\\\lesssim \log m
\left(\frac{1}{2^m}\sum_{x\in
\F_2^n}\E_\e\left[\left\|\sum_{i=1}^m\e_i(f(x+e_i)-f(x))\right\|_Y^r\right]\right)^{\frac{1}{r}}.
\end{multline}
For every fixed $x\in \F_2^m$ it follows from Kahane's inequality
(with asymptotically optimal dependence on $r$; see
e.g.~\cite{Tal88}) that
\begin{align*}
&\left(\E_\e \left[\left\|\sum_{i=1}^m\e_i(f(x+e_i)-f(x))\right\|_Y^r\right]\right)^{\frac{1}{r}}
\\&\lesssim \sqrt{r} \left(\E_\e \left[\left\|\sum_{i=1}^m\e_i(f(x+e_i)-f(x))\right\|_Y^2\right]\right)^{\frac12}
\\&\le
\sqrt{r}T_q(Y)\left(\sum_{i=1}^m \|f(x+e_i)-f(x)\|_Y^q\right)^{\frac{1}{q}}.
%\\&\le \sqrt{r}T_q(X)m^{1/q}\cdot \max_{i\in \{1,\ldots,m\}} \|\partial_if(\e)\|_Y.
\end{align*}
Combined with~\eqref{eq:pisier's ineq first use}, this implies that
\begin{equation*}
\max_{\substack{x\in \F_2^m\\i\in \{1,\ldots,m\}}} \|f(x+e_i)-f(x)\|_Y\gtrsim \frac{1}{\sqrt{r}T_q(Y)}\cdot \frac{m^{\frac{1}{p}-\frac{1}{q}}}{\log m}.\qedhere
\end{equation*}
\end{proof}

There are classes of Banach spaces $Y$, including Banach lattices of nontrivial type and UMD spaces, for which it is known that
Pisier's inequality~\eqref{eq:pisier's ineq first use} holds true
with the $\log m$ factor replaced by a constant that may depend on
$Y$ and $r$ but not on $m$; see~\cite{NS02,HN12}. For such spaces we
therefore obtain~\eqref{eq:lower average cube with logm term}
without the $\log m$ term.

\begin{lemma}\label{lem:uniformly smooth type}
Assume that $p\in [1,2)$ and fix $q\in (p,2]$ and $r,s\in
[1,\infty)$. Suppose that $(Y,\|\cdot\|_Y)$ is a Banach space with
$S_q(Y)<\infty$, i.e., $Y$ has modulus of uniform smoothness of
power type $q$. If $f:\F_2^m\to Y$ satisfies~\eqref{eq:f assumption
pisier ineq} then there exists $x\in \F_2^m$ and $i\in \n$ such that
\begin{equation}\label{eq:lower average cube smoothness}
\|f(x)-f(x+e_i)\|_Y \gtrsim \frac{m^{\frac{1}{p}-\frac{1}{q}}}{r^{1/q}+S_{q}(Y)}.
\end{equation}
\end{lemma}

\begin{proof} Due to~\eqref{eq:f assumption pisier ineq}, in order to prove~\eqref{eq:lower average cube smoothness} it
suffices to show that for every $h:\F_2^m\to Y$,
\begin{multline}\label{mtype to cube type}
\left(\frac{1}{4^m}\sum_{(x,y)\in \F_2^m\times
\F_2^m}\|h(x)-h(y)\|_Y^r\right)^{\frac{1}{r}}\\ \lesssim
\left(r^{\frac{1}{q}}+S_{q}(Y)\right)m^{\frac{1}{q}}\left(\frac{1}{m2^m}\sum_{i=1}^m\sum_{x\in
\F_2^m} \|h(x+e_i)-h(x)\|_Y^r\right)^{\frac{1}{r}}.
\end{multline}
Note that it suffices to prove~\eqref{mtype to cube type} when $r\ge
2$, since otherwise we could replace $q$ by $r$ and use the fact
that $S_r(Y)\le S_q(Y)$.

By considering the standard random walk on the Hamming cube $\F_2^n$
and arguing mutatis mutandis as in~\cite[Sec.~5]{NS02}, \eqref{mtype
to cube type} is a formal consequence of the Markov type estimate of
Theorem~\ref{thm:npss}. Alternatively, once can deduce~\eqref{mtype
to cube type} directly via the martingale argument
in~\cite[Sec.~5]{KN06}, the only difference being the use of the
martingale inequality~\eqref{eq:martingale with constant} in place
of Pisier's inequality~\cite{Pisier-martingales}.
\end{proof}

\begin{proof}[Proof of Lemma~\ref{lem:no cube}]
Since for $q\in [2,\infty)$ we have $S_2(\ell_q)\le \sqrt{q-1}$,
Lemma~\ref{lem:no cube} is a special case  of
Lemma~\ref{lem:uniformly smooth type} with $Y=\ell_q$, $n=2^m$ and
$x_1,\ldots,x_{2^m}$ being an arbitrary enumeration of
$\F_2^m\subseteq Y$.
\end{proof}

\begin{remark}\label{rem:sqrt[4]n}
As promised in the Introduction, here we justify~\eqref{eq:av cube
fourth root}. The fact that $Av_\R(\F_2^n,\|\cdot\|_2)\lesssim
\sqrt[4]{n}$ is simple: consider the mapping $\f:\F_2^n\to \R$ given
by $\f(x)=\sqrt{\max\{\|x\|_1-n/2,0\}}$. Then $\f$ is $1$-Lipschitz with
respect to the metric induced on $\F_2^n$ by the Euclidean norm
$\|\cdot\|_2$. By the central limit
theorem the average of $|\f(x)-\f(y)|^2$ over $(x,y)\in \F_2^n\times
\F_2^n$ is of order $\sqrt{n}$.

The corresponding lower bound $Av_\R(\F_2^n,\|\cdot\|_2)\gtrsim
\sqrt[4]{n}$ is an example of a lower bound on the average
distortion of the cube $\F_2^n$ that is not proved through the use
on nonlinear type. Suppose that $f:\F_2^n\to \R$ satisfies
$|f(x)-f(y)|\le \|x-y\|_2$ for every $x,y\in \F_2^n$. Suppose also that
$S\subset \F_2^n$ satisfies $|S|\ge 2^{n-1}$. Then by Harper's
inequality~\cite{Har66} (see also~\cite[Thm.~2.11]{Led01}), for
every $t\in (0,\infty)$ we have
\begin{multline*}
\frac{\left|\left\{x\in \F_2^n:\ \forall\, y\in S,\ \|x-y\|_2\ge
t\right\}\right|}{2^n}\\= \frac{\left|\left\{x\in \F_2^n:\ \forall\,
y\in S,\ \|x-y\|_1\ge t^2\right\}\right|}{2^n}\le e^{-2t^4/n}.
\end{multline*}
Consequently, if $M_f\in \R$ is a median of $f$ then
by~\cite[Prop.1.3]{Led01} we have $\left|\left\{x\in \F_2^n:\
|f(x)-M_f|\ge t\right\}\right|/2^n\le 2e^{-2t^4/n}$.  Hence,
\begin{align*}
\left(\frac{1}{4^n}\sum_{(x,y)\in \F_2^n\times
\F_2^n}|f(x)-f(y)|^2\right)^{\frac12} &\le
2\left(\frac{1}{2^n}\sum_{(x,y)\in \F_2^n\times
\F_2^n}|f(x)-M_f|^2\right)^{\frac12}\\  &\le \left(\int_0^\infty
4te^{-2t^4/n} dt\right)^{\frac12}\lesssim \sqrt[4]{n}.
\end{align*}
\end{remark}

\begin{remark} Additional obstructions to average distortion
embeddings that do not fall into the framework described in this
section have been obtained in the context of integrality gap lower
bounds for the Goemans--Linial semidefinite relaxation for the
Uniform Sparsest Cut Problem. The best known result in this
direction is due to~\cite{KM13} (improving over the
works~\cite{DKSV06,KR09}), where it is shown that for arbitrarily
large $n\in \N$ there exists an $n$-point metric space $(X,d_X)$
such that the metric space $(X,\sqrt{d_X})$ emebds isometrically
into $\ell_2$, yet $Av_{\ell_1}^{(1)}(X)\ge
\exp\left(c\sqrt{\log\log n}\right)$,
%$$
%Av_{\ell_1}^{(1)}(X)\ge e^{c\sqrt{\log\log n}},
%$$
where $c\in (0,\infty)$ is a universal constant. Finding the correct asymptotic dependence here remains open.
\end{remark}

\section{Existence of average distortion embeddings}\label{sec:average}

The main purpose of this section is to state criteria for the
existence of average distortion embeddings. In what follows we often
discuss probability distributions over random subsets or random
partitions of metric spaces. To avoid measurability issues we focus
our discussion on finite metric spaces. Such topics can be treated
for infinite spaces as well, as done in~\cite{LN-extension}.

\subsection{Random zero sets}\label{sec:random zero}
Fix $\Delta,\zeta\in (0,\infty)$ and $\d\in (0,1)$.
Following~\cite{ALN08}, a
finite metric space $(X,d_X)$ is said to admit a {\em random zero
set} at scale $\Delta$ which is $\zeta$-spreading with probability
$\d$ if there exists a probability distribution $\mu$ over $2^X$
such that every $x,y\in X$ satisfy
\begin{equation}\label{eq:def zero set}
d_X(x,y)\ge \Delta\implies
\mu\left(\left\{Z\in 2^X:\ x\in Z\ \wedge \ d_X(y,Z)\ge \frac{\Delta}{\zeta}\right\}\right)\ge \d.
\end{equation}
We denote by $\zeta(X;\d)$ the infimum over those $\zeta\in
(0,\infty)$ such that for every scale $\Delta\in (0,\infty)$ the
finite metric space $(X,d_X)$ admits a random zero set at scale
$\Delta$ which is $\zeta$-spreading with probability $\d$. If
$(X,d_X)$ is an infinite metric space then we write
\begin{equation}\label{eq:zeta for infinite}
\zeta(X;\d)\eqdef \sup_{\substack{S\subseteq X\\ |S|<\infty}}\zeta(X;\d).
\end{equation}

The following proposition asserts that random zero sets can be used to
obtain embeddings into the real line $\R$ with low average
distortion.
\begin{proposition}\label{prop:into R}
Fix $n\in \N$ and $\d\in (0,1)$. Suppose that $(X,d_X)$ is a metric
space with $\zeta(X;\d)<\infty$. Then for every $p\in [1,\infty)$
and every $x_1,\ldots,x_n\in X$  there exists a $1$-Lipschitz
function $f:X\to \R$ such that
$$
\sum_{i=1}^n\sum_{j=1}^n |f(x_i)-f(x_j)|^p\ge \frac{\d}{2^{11p}\zeta(X;\d)^p} \sum_{i=1}^n\sum_{j=1}^n d_X(x_i,x_j)^p.
$$
Thus, using the notation of Section~\ref{sec:MP},  $
Av_\R^{(p)}(X)\lesssim \zeta(X;\d)/\d^{1/p}$.
\end{proposition}

Proposition~\ref{prop:into R} will be proven in
Section~\ref{sec:prop proofs} below. We will now explain how
Proposition~\ref{prop:into R} can be applied to a variety of metric
spaces. Due to the discussion preceding Theorem~\ref{thm:duality},
such spaces will satisfy the spectral inequality~\eqref{eq:Psi
euclidean} with $\Psi$ linear.

\subsubsection{Random partitions}\label{sec:part} Many spaces are known to admit good random zero sets. Such examples often (though not always) arise from metric spaces for which one can construct {\em random padded partitions}. If $(X,d_X)$ is a finite metric space  let $\mathscr{P}(X)$ denote the set of all partitions of $X$. For $P\in \mathscr{P}(X)$ and $x\in X$, the unique element of $P$ to which $x$ belongs is denoted $P(x)\subseteq X$. Given $\e,\d\in (0,1)$, the metric space $(X,d_X)$ is said to admit an $\e$-padded random partition with probability $\d$ if for every $\Delta\in (0,\infty)$ there exists a probability distribution $\mu_\Delta$ over partitions of $X$ with the following properties.
\begin{itemize}
\item $\forall\, P\in \mathscr{P}(X),\quad \mu_\Delta(P)>0\implies \max_{x\in X} \diam P(x)\le \Delta$.
\item For every $x\in X$ we have
$$
\mu_\Delta\left(\left\{P\in \mathscr{P}(X):\ B_X(x,\e\Delta)\subseteq P(x)\right\}\right)\ge \d,
$$
where $B_X(x,r)\eqdef \{y\in X:\ d_X(x,y)\le r\}$ for every
$r\in [0,\infty)$.
\end{itemize}
Let $\e(X;\d)$ denote the supremum over those $\e\in (0,1)$ for
which $(X,d_X)$ admits an $\e$-padded random partition with
probability $\d$. As in~\eqref{eq:zeta for infinite}, we extend this
definition to infinite metric spaces $(X,d_X)$ by setting
\begin{equation*}
\e(X;\d)\eqdef \inf_{\substack{S\subseteq X\\ |S|<\infty}}\e(X;\d).
\end{equation*}

Fact~3.4 in~\cite{ALN08} (which itself uses an idea of~\cite{Rao99})
asserts that for every $\d\in (0,1)$, if $(X,d_X)$ is a finite
metric space then
\begin{equation}\label{eq:eps zeta}
\e(X;\d)\cdot\zeta\left(X;\frac{\d}{4}\right)\le 1.
\end{equation}
(\cite[Fact~3.4]{ALN08} states this for the arbitrary choice
$\d=\frac12$, but its proof does not use this specific value of $\d$
in any way.) One should interpret~\eqref{eq:eps zeta} as asserting
that a lower bound on $\e(X;\d)$ implies an upper bound on
$\zeta(X,\d/4)$. The following classes of metric spaces $(X,d_X)$
are known to satisfy $\e(X;\d)>0$ for some $\d\in (0,1)$: doubling
metric spaces, compact Riemannian surfaces, Gromov hyperbolic spaces
of bounded local geometry, Euclidean buildings, symmetric spaces,
homogeneous Hadamard manifolds, and forbidden-minor (edge-weighted)
graph families. The case of doubling spaces goes back
to~\cite{Ass83}, with subsequent improved bounds on $\e(X;\d)>0$
obtained in~\cite{GKL03}. The case of forbidden-minor graph families
is due to~\cite{KPR93}, with subsequent improved bounds on
$\e(X;\d)>0$ obtained in~\cite{FT03}. The case of compact Riemannian
surfaces is due to~\cite{LN-extension}, with subsequent improved
bounds on $\e(X;\d)>0$ obtained in~\cite{LS10}. The remaining cases
follow from the general fact~\cite{NS11} that if $(X,d_X)$ has
bounded Nagata dimension then $\e(X;\d)>0$ for some $\d\in (0,1)$
(see~\cite{LS05} for more information on Nagata dimension of metric
spaces). We single out the following two consequences of
Proposition~\ref{prop:into R} and (the easy direction of) Theorem~\ref{thm:duality}, with
explicit quantitative bounds arising from the estimates on
$\e(X;\d)$ obtained in~\cite{GKL03,LS10}.

\begin{corollary}\label{cor:doubling}
Suppose that $(X,d_X)$ is a metric space that is doubling with
constant $K\in [2,\infty)$. Then for every $n\in \N$ and every
symmetric stochastic matrix $A\in M_n(\R)$ we have
$$
\gamma\!\left(A,d_X^2\right)\lesssim \frac{(\log K)^2}{1-\lambda_2(A)}.
$$
\end{corollary}

\begin{corollary}\label{cor:surface}
Suppose that $(X,d_X)$ is a two dimensional Riemannian manifold of
genus $g\in \N\cup \{0\}$. Then for every $n\in \N$ and every
symmetric stochastic matrix $A\in M_n(\R)$ we have
$$
\gamma\!\left(A,d_X^2\right)\lesssim \frac{(\log(g+1))^2}{1-\lambda_2(A)}.
$$
\end{corollary}

The fact that the conclusion of Proposition~\ref{prop:into R} holds
true under the assumption that $\e(X;\d)>0$ for every $\d\in
(0,\infty)$ (as follows by combining Proposition~\ref{prop:into R}
with~\eqref{eq:eps zeta}) was proved by Rabinovich in~\cite{Rab08}
in the case $p=1$. It has long been well known to experts (and
stated explicitly in~\cite{BLR10}), that the original proof of
Rabinovich extends mutatis mutandis to every $p\in [1,\infty)$. The
(simple) proof of Proposition~\ref{prop:into R} below builds on the ideas
of Rabinovich in~\cite{Rab08}.

An example of a class of metric spaces that admits good random zero
sets for reasons other than the existence of random padded
partitions is the class of spaces that admit a quasisymmetric
embedding into Hilbert space. We refer to~\cite{Hei01} and the
references therein for more information on quasisymmetric
embeddings; it suffices to say here that $L_1(\mu)$ spaces provide
such examples (see~\cite{DZ}). It follows from~\cite{ALN08} (using in part
ideas of~\cite{ARV09,NRS05,Lee05,CGR08}) that if $(X,d_X)$ is a metric
space that admits a quasisymmetric embedding into Hilbert space then
there exist $\e,\d\in (0,1)$ (depending only on the modulus of
quasisymmetry of the implicit embedding) such that for every $n\in
\N$, any $n$-point subset $S\subseteq X$ satisfies $\e(S;\d)\ge
\e/\sqrt{\log n}$. Consequently we have the following statement.

\begin{corollary}\label{cor:quasisymmetric}
Suppose that $(X,d_X)$ is a metric space that admits a
quasisymmetric embedding into a Hilbert space. Then there exists a
constant $C\in (0,\infty)$ (depending only on the modulus of
quasisymmetry of the implicit embedding) such that for every $n\in
\N$ and every symmetric stochastic matrix $A\in M_n(\R)$ we have
\begin{equation}\label{eq:quasi}
\gamma\!\left(A,d_X^2\right)\le \frac{C\log n}{1-\lambda_2(A)}.
\end{equation}
\end{corollary}
Note that Bourgain's embedding theorem~\cite{Bourgain-embed} implies
that~\eqref{eq:quasi} holds true for {\em every} metric space
$(X,d_X)$ if one replaces the term $\log n$ by $(\log n)^2$ (in
which case $C$ can be taken to be a universal constant).

\subsection{Localized weakly bi-Lipschitz embeddings}\label{sec:weak} Following the terminology of~\cite{NPSS06},
for $D\in [1,\infty)$ say that a metric space $(X,d_X)$ admits a
weakly bi-Lipschitz embedding with distortion $D$ into a metric
space $(Y,d_Y)$ if for every $\Delta\in (0,\infty)$ there exists a
non-constant Lipschitz mapping $f_\Delta:X\to Y$ such that for every
$x,y\in X$,
\begin{equation}\label{eq:weak def}
d_X(x,y)\ge \Delta\implies d_Y\left(f_\Delta(x),f_{\Delta}(y)\right)\ge \frac{\left\|f_{\Delta}\right\|_{\Lip}}{D} \Delta.
\end{equation}

The origin of this terminology is that such embeddings preserve (by
design) {\em weak} $(p,q)$ metric Poincar\'e inequalities.
Specifically, a standard way by which one rules out the existence of
bi-Lipschitz embeddings is via generalized Poincar\'e-type
inequalities as follows. Suppose that $n\in \N$ and $p,q,K\in
(0,\infty)$, and there exist two measures $\mu,\nu$ on
$\{1,\ldots,n\}^2$ such that every $y_1,\ldots,y_n\in Y$ satisfy
\begin{multline}\label{eq:mu nu}
\left(\sum_{i=1}^n\sum_{j=1}^n
d_Y(y_i,y_j)^p\mu(i,j)\right)^{\frac{1}{p}}\le
\left(\sum_{i=1}^n\sum_{j=1}^n
d_Y(y_i,y_j)^q\nu(i,j)\right)^{\frac{1}{q}}.
\end{multline}
Clearly if $f:X\to Y$ is a bi-Lipschitz embedding then the
inequality~\eqref{eq:mu nu} holds for $(X,d_X)$ as well, with the
right hand side of~\eqref{eq:mu nu} multiplied by
$\|f\|_{\Lip}\|f^{-1}\|_{\Lip}$.  Thus a {\em strong $(p,q)$}
inequality such as~\eqref{eq:mu nu} are bi-Lipschitz invariants
that can be used to show that certain spaces $(X,d_X)$ must incur
large distortion in any bi-Lipschitz embedding into $(Y,d_Y)$. The
obvious {\em weak $(p,q)$} variant of~\eqref{eq:mu nu} is the
assertion that for every $u\in (0,\infty)$ and every $y_1,\ldots,y_n\in Y$ we have
\begin{multline}\label{eq:weak pq}
\mu\left(\left\{(i,j)\in \{1,\ldots,n\}^2:\ d_Y(y_i,y_j)\ge
u\right\}\right)^{\frac{1}{p}}\\\le
\frac{1}{u}\left(\sum_{i=1}^n\sum_{j=1}^n
d_Y(y_i,y_j)^q\nu(i,j)\right)^{\frac{1}{q}}.
\end{multline}
%It is immediate from the definition that weak $(p,q)$ metric
%inequalities such as~\eqref{eq:weak pq} are invariant under weakly
%bi-Lipschitz embeddings, i.e., families of mappings $\{f_\Delta:X\to
%Y\}_{\Delta\in (0,\infty)}$ satisfying~\eqref{eq:weak def}.
%Specifically,
By definition, if a metric space $(X,d_X)$ admits a weakly bi-Lipschitz embedding with distortion $D$ into a metric space
$(Y,d_Y)$ satisfying~\eqref{eq:weak pq} then for every $n\in \N$,
any $x_1,\ldots,x_n\in X$ satisfy
\begin{multline*}\label{eq:weak pq-X weak embed}
\mu\left(\left\{(i,j)\in \{1,\ldots,n\}^2:\ d_X(x_i,x_j)\ge
u\right\}\right)^{\frac{1}{p}}\\\le
\frac{D}{u}\left(\sum_{i=1}^n\sum_{j=1}^n
d_X(x_i,x_j)^q\nu(i,j)\right)^{\frac{1}{q}}.
\end{multline*}

We will see below that one can prove nonlinear spectral gap
inequality such as~\eqref{eq:Q meta} with $\Psi$ linear by showing
that $(X,d_X)$ admits  a weakly bi-Lipschitz embedding into
$(Y,d_Y)$. To this end it suffices to localize the
condition~\eqref{eq:weak def} to balls of proportional scale, as
follows.  For $D\in [1,\infty)$ say that a metric space $(X,d_X)$
admits a {\em localized weakly bi-Lipschitz embedding} with
distortion $D$ into a metric space $(Y,d_Y)$ if for every $z\in X$
and $\Delta\in (0,\infty)$ there exists a non-constant Lipschitz
mapping $f^z_\Delta:X\to Y$ such that for every $x,y\in
B_X(z,32\Delta)$ we have
\begin{equation}\label{eq:localized weak def}
d_X(x,y)\ge \Delta\implies d_Y\left(f^z_\Delta(x),f^z_{\Delta}(y)\right)\ge \frac{\left\|f^z_{\Delta}\right\|_{\Lip}}{D} \Delta.
\end{equation}
The factor $32$ here was chosen to be convenient for the ensuing
arguments, but it is otherwise arbitrary.
\begin{proposition}\label{prop:weak embedding into Y}
Fix $n\in \N$ and $p,D\in [1,\infty)$. Suppose that $(X,d_X)$ is a
metric space that admits a localized weakly bi-Lipschitz embedding
with distortion $D$ into a metric space $(Y,d_Y)$. Then for every $x_1,\ldots,x_n\in X$ there is a nonconstant mapping $f:\{x_1,\ldots,x_n\}\to Z$, where $Z\in \{Y,\R\}$, such that
$$
\left(\frac{1}{n^2}\sum_{i=1}^n\sum_{j=1}^n d_Z(f(x_i),f(x_j))^p\right)^{\frac{1}{p}}\gtrsim \frac{\|f\|_{\Lip}}{D}\left(\frac{1}{n^2}\sum_{i=1}^n\sum_{j=1}^n d_X(x_i,x_j)^p\right)^{\frac{1}{p}}.
$$
%\begin{equation}\label{eq:Y or R}
%Av_Y^{(p)}(X)\le 2^{11}D\quad \mathrm{or}\quad Av_\R^{(p)}(X)\le 2^5.
%\end{equation}
\end{proposition}
Observe that if $(Y,d_Y)$ contains an isometric copy of an interval
$[a,b]\subseteq \R$ (in particular if $Y$ is a Banach space) then
the conclusion of Proposition~\ref{prop:weak embedding into Y} can
be taken to be $Av_Y^{(p)}(X)\lesssim D$.
%In general, the conclusion of Proposition~\ref{prop:weak embedding into Y} can
%be taken to be $Av_W^{(p)}(X)\lesssim D$ where $W=(Y\oplus \R)_p$ is the $\ell_p$-direct sum of $Y$ and $\R$.

Lemma~3.5 in~\cite{ALN08} asserts that for every $\d\in (0,1)$,
every finite metric space $(X,d_X)$ admits a weakly bi-Lipschitz
embedding into $\ell_2$ with distortion $\zeta(X;\d)/\sqrt{\d}$.
Consequently, all the examples that arise from random padded
partitions as described in Section~\ref{sec:part} fall into the
framework of Proposition~\ref{prop:weak embedding into Y}, with the
only difference being that an application of
Proposition~\ref{prop:weak embedding into Y} rather than
Proposition~\ref{prop:into R} yields an embedding  into Hilbert
space rather than into the real line. This difference is discussed
further in Section~\ref{sec:dim} below. The following lemma shows
that Proposition~\ref{prop:weak embedding into Y} has wider
applicability than Proposition~\ref{prop:into R}: in combination
with Proposition~\ref{prop:weak embedding into Y} it yields a
different proof of the case $p\in (2,\infty)$ of~\eqref{eq:p>2} that
avoids the use of Theorem~\ref{thm:duality}.

\begin{lemma}\label{lem:no duality}
Suppose that $(X,\|\cdot\|_X)$ and $(Y,\|\cdot\|_Y)$ are Banach
spaces that satisfy the assumptions of Theorem~\ref{thm:ozawa-ineq}.
Suppose furthermore that there exists $K\in (0,\infty)$ such that
$\beta(t)=Kt$ for all $t\in [0,\infty)$. Then $(X,\|\cdot\|_X)$
admits a localized weakly bi-Lipschitz embedding with distortion
$2K/\alpha(1/32)$ into $(Y,\|\cdot\|_Y)$.
\end{lemma}

\begin{proof}
Fix $\Delta\in (0,\infty)$ and a mapping $f:B_X\to Y$ that
satisfies~\eqref{eq;alpha beta}.  For $z\in X$ define
$f^z_\Delta:X\to Y$ by
$$
f^z_\Delta(x)\eqdef 32\Delta f\left(\rho\left(\frac{x-z}{32\Delta}\right)\right),
$$
where $\rho$ is given as in~\eqref{eq:def rho}. Since $\rho$ is
$2$-Lipschitz, $\|f^z_\Delta\|_{\Lip}\le 2K$. If $x,y\in z+32\Delta
B_X$ satisfy $\|x-y\|_X\ge \Delta$ then
\begin{equation*}
\left\|f^z_\Delta(x)-f^z_\Delta(y)\right\|_Y\ge \Delta\alpha\left(\frac{\|x-y\|_X}{32\Delta}\right)\ge \frac{\alpha\left(\frac{1}{32}\right)\|f^z_\Delta\|_{\Lip}}{2K}\Delta.\qedhere
\end{equation*}
\end{proof}

For $p\in [1,2)$, due to Lemma~\ref{lem:no cube} and
Proposition~\ref{prop:weak embedding into Y}, $\ell_p$ does not
admit a localized weakly bi-Lipschitz embedding into Hilbert
space (this can also be proved directly via a shorter argument).
Since finite subsets of $\ell_p$ embed isometrically into $\ell_1$
(see e.g.~\cite{DZ}), it follows from~\cite{ARV09} that every
$n$-point subset of $\ell_p$ admits a weakly bi-Lipschitz embedding
into $\ell_2$ with distortion $O(\sqrt{\log n})$. By~\cite{Lee05},
it is also true that every $n$-point subset of $\ell_p$ admits a
weakly bi-Lipschitz embedding into $\ell_2$ with distortion $O((\log
n)^{(2-p)/p^2})$, which is better than the $O(\sqrt{\log n})$ bound
of~\cite{ARV09} if $\sqrt{5}-1<p\le 2$. Therefore for every $n\in
\N$ and every $n$ by $n$ symmetric stochastic matrix $A$,
$$
p\in \left[1,\sqrt{5}-1\right]\implies \gamma(A,\|\cdot\|_{\ell_p}^2)\lesssim
\frac{\log n}{1-\lambda_2(A)},
$$
and
$$
p\in \left[\sqrt{5}-1,2\right]\implies \gamma(A,\|\cdot\|_{\ell_p}^2)\lesssim
\frac{(\log n)^{\frac{2(2-p)}{p^2}}}{1-\lambda_2(A)}.
$$
These are the currently best known bounds towards Question~\ref{Q:p
ARV}.

\subsection{Dimension reduction}\label{sec:dim}

As discussed in Section~\ref{sec:weak},  Proposition~\ref{prop:into
R} yields an average distortion embedding into the real line, while
Proposition~\ref{prop:weak embedding into Y}, when applied in the
context of spaces with random zero sets, yields an average
distortion embedding into Hilbert space. Here we briefly compare
these two notions. The following lemma is a simple application of
the classical Johnson-Lindenstrauss dimension reduction
lemma~\cite{JL84}.

\begin{lemma}\label{lem:JL}
If $(X,d_X)$ is an $n$-point metric space then
\begin{equation}\label{eq:apply JL}
\frac{Av_{\R}^{(2)}(X)}{\sqrt{\log n}}\le Av_{\ell_2}^{(2)}(X)\lesssim  Av_{\R}^{(2)}(X).
\end{equation}
\end{lemma}

\begin{proof}
The rightmost inequality in~\eqref{eq:apply JL} is trivial. Write
$D=Av_{\ell_2}^{(2)}(X)$ and take $x_1,\ldots,x_m\in X$. By the
Johnson-Lindenstrauss lemma~\cite{JL84} there exists $k\in \N$ such
that $k\lesssim \log n$ and there exists a $1$-Lipschitz function
$f=(f_1,\ldots,f_k):\{x_1,\ldots,x_m\}\to \R^k$ such that
$$
\sum_{i=1}^m\sum_{j=1}^m \|f(x_i)-f(x_j)\|_{2}^2\ge \frac{1}{2D^2}\sum_{i=1}^m\sum_{j=1}^m d_X(x_i,x_j)^2.
$$
Therefore there exists $s\in \{1,\ldots,k\}$ such that
$$
\sum_{i=1}^m\sum_{j=1}^m |f_s(x_i)-f_s(x_j)|^2\ge \frac{1}{2kD^2}\sum_{i=1}^m\sum_{j=1}^m d_X(x_i,x_j)^2.
$$
Since $f_s:\{x_1,\ldots,x_m\}\to \R$ is also $1$-Lipschitz, we
conclude that $Av_{\R}^{(2)}(X)\le\sqrt{2k}D\lesssim \sqrt{\log
n}D$.
\end{proof}

The following lemma shows that Lemma~\ref{lem:JL} is almost
asymptotically sharp.

\begin{lemma}\label{lem:sharp in l2}
For arbitrarily large $n\in \N$ there exists an $n$-point metric
space $X_n$ such that
\begin{equation}\label{eq:log log}
Av_{\R}^{(2)}(X_n)\gtrsim  Av_{\ell_2}^{(2)}(X_n)\cdot \sqrt{\frac{\log n}{\log\log n}}.
\end{equation}
\end{lemma}

\begin{proof}
Fix $\e\in (0,1)$ and an integer $m\ge 2$. Let $\sigma$ denote the
normalized surface measure on the unit sphere $S^{m-1}\subseteq
\ell_2^m$. By Lemma~21 in~\cite{FS02} there exists a partition
$\{C_1,\ldots,C_n\}$ of $S^{m-1}$ into nonempty measurable sets such
that $\sigma(C_i)=1/n$ and  $\diam(C_i)\le \e$ for all $i\in
\{1,\ldots,n\}$, and $n\le (\kappa/\e)^m$ for some universal
constant $\kappa\in (0,\infty)$.

Choose an arbitrary point $x_i\in C_i$ and set
$X_n=\{x_1,\ldots,x_n\}\subseteq \ell_2^{m-1}$. Since $X_n$ is
isometric to a subset of Hilbert space, $Av_{\ell_2}^{(2)}(X_n)=1$.
Suppose that $f:\{x_1,\ldots,x_n\}\to \R$ is a $1$-Lipschitz
function. By the nonlinead Hahn-Banach theorem (see~\cite{BL}) we
can think of $f$ as the restriction to $X_n$ of a $1$-Lipschitz
function defined on all of $S^{m-1}$. The Poincar\'e inequality on
the Sphere $S^{m-1}$ (see e.g.~\cite{Cha84,Led01}) asserts that
\begin{multline}\label{eq:use poincare sphere}
\int_{S^{m-1}}\int_{S^{m-1}}
|f(x)-f(y)|^2d\sigma(x)d\sigma(y)\\\le
\frac{2}{m-1}\int_{S^{m-1}} \left\|\nabla
f(x)\right\|_2^2d\sigma(x)\le \frac{4}{m}.
\end{multline}
For every $i,j\in \{1,\ldots,n\}$ and every $(x,y)\in C_i\times C_j$
we have
\begin{align}\label{eq:use lip1}
\nonumber\frac{|f(x_i)-f(x_j)|^2}{3}&\le |f(x_i)-f(x)|^2+|f(x)-f(y)|^2+|f(y)-f(x_j)|^2\\\nonumber&\le |f(x)-f(y)|^2+\diam(C_i)^2+\diam(C_j)^2\\&\le |f(x)-f(y)|^2+2\e^2,
\end{align}
and similarly,
\begin{equation}\label{eq:use lip2}
\|x-y\|_{2}^2\le 3\|x_i-x_j\|_2^2+6\e^2.
\end{equation}
Consequently,
\begin{multline}\label{eq:integrate on product}
\frac{|f(x_i)-f(x_j)|^2-6\e^2}{n^2}=\sigma(C_i)\sigma(C_j)\left(|f(x_i)-f(x_j)|^2-6\e^2\right)\\
\stackrel{\eqref{eq:use lip1}}{\le}
3\int_{C_i}\int_{C_j}|f(x)-f(y)|^2d\sigma(x)d\sigma(y).
\end{multline}
and
\begin{equation}\label{eq:big variance}
\frac{\|x_i-x_j\|_2^2}{n^2}\ge \frac{1}{3}\int_{C_i}\int_{C_j}\|x-y\|_2^2d\sigma(x)d\sigma(y)-\frac{2\e^2}{n^2}.
\end{equation}
Hence,
\begin{multline*}
\frac{1}{n^2}\sum_{i=1}^n\sum_{j=1}^n|f(x_i)-f(x_j)|^2\\
\stackrel{\eqref{eq:integrate on product}}{\le}
3\int_{S^{m-1}}\int_{S^{m-1}}
|f(x)-f(y)|^2d\sigma(x)d\sigma(y)+6\e^2\stackrel{\eqref{eq:use
poincare sphere}}{\le} \frac{12}{m}+6\e^2,
\end{multline*}
and
\begin{eqnarray*}
\frac{1}{n^2}\sum_{i=1}^n\sum_{j=1}^n\|x_i-x_j\|_2^2&\stackrel{\eqref{eq:big variance}}{\ge} & \frac{1}{3}\int_{S^{m-1}}\int_{S^{m-1}}\|x-y\|_2^2d\sigma(x)d\sigma(y)-2\e^2\\&=&\frac23-2\e^2.
\end{eqnarray*}
By choosing $\e=1/(2\sqrt{m})$, we have shown that every
$1$-Lipschitz  function $f:X_n\to \R$ satisfies
$$
\sum_{i=1}^n\sum_{j=1}^n|f(x_i)-f(x_j)|^2\lesssim \frac{1}{m}\sum_{i=1}^n\sum_{j=1}^n\|x_i-x_j\|_2^2.
$$
Recalling that $n\le (\kappa/\e)^m=(2\kappa\sqrt{m})^m$, or
$m\gtrsim \sqrt{\log n/\log\log n}$, the proof of~\eqref{eq:log log}
is complete.
\end{proof}

\begin{remark}\label{rrem:spread constant} Given an $n$-point metric space $(X,d_X)$ let $\mathfrak{S}(X)$ denote the maximum of $\frac{1}{2n^2}\sum_{(x,y)\in X\times X} |f(x)-f(y)|^2$ over all $1$-Lipschitz functions $f:X\to \R$. The quantity $\mathfrak{S}(X)$ was introduced by Alon, Boppana and Spencer in~\cite{ABS98}, where they called it the {\em spread constant of $X$}. They proved that the spread constant of $X$ governs the asymptotic isoperimetric behavior of $\ell_1^n(X)$ as $n\to \infty$. They also state that ``The spread constant appears to be new and may well be of independent interest." We agree with this assertion. In particular, it would be worthwhile to investigate the computational complexity of the problem that takes as input an $n$-point metric space $(X,d_X)$ and is supposed to output in polynomial time a number that is guaranteed to be a good approximation of its spread constant. We are not aware of hardness of approximation results for this question. Let $\mathfrak{S}_{\ell_2}(X)$ denote the maximum of $\frac{1}{2n^2}\sum_{(x,y)\in X\times X} \|f(x)-f(y)\|_2^2$ over all $1$-Lipschitz functions $f:X\to \ell_2$. The quantity $\mathfrak{S}_{\ell_2}(X)$ can be computed in polynomial time with arbitrarily good precision, since (by definition) it can be cast as a semidefinite program (see~\cite{GLS93}). The proof of Lemma~\ref{lem:JL} can be viewed as a simple approximation algorithm to the spread constant, achieving an approximation guarantee of $O(\log n)$. Lemma~\ref{lem:sharp in l2} can be viewed as yielding an almost matching integrality gap lower bound for the semidefinite program. Note that the parameter $\mathfrak{S}_{\ell_2}(X)$ itself has also been studied in the literature in the context of the problem of finding the fastest mixing Markov process on a given graph; see~\cite{SBXD06}. See also the works~\cite{Fie89,GHW08,GHR12} that study this quantity in the context of the absolute algebraic connectivity of a graph. Clearly $(Av_{\R}^{(2)}(X))^2$ is closely related to $\mathfrak{S}(X)$: it amounts to finding the (multi)subset of $X$ with largest spread constant. The same can be said for the relation between $(Av_{\ell_2}^{(2)}(X))^2$ and $\mathfrak{S}_{\ell_2}(X)$.
\end{remark}

\subsection{Proofs of Proposition~\ref{prop:into R} and Proposition~\ref{prop:weak embedding into Y}}\label{sec:prop proofs}
We start by recording the following very simple lemma, whose proof is a
straightforward application of the triangle inequality.
\begin{lemma}\label{lem:D triangle}
Fix $p\in [1,\infty)$ and $n\in \N$. Let $(X,d_X)$ be a metric space
and $x_1,\ldots,x_n,\in X$. Define
\begin{equation}\label{eq:defD}
r\eqdef \min_{i\in \{1,\ldots,n\}} \left(\frac{1}{n}\sum_{j=1}^nd_X(x_i,x_j)^p\right)^{\frac1{p}}.
\end{equation}
Then
\begin{equation}\label{eq:Dis the moment up to 2}
r\le \left(\frac{1}{n^2}\sum_{i=1}^n\sum_{j=1}^n d_X(x_i,x_j)^p\right)^{\frac1{p}}\le 2r.
\end{equation}
\end{lemma}
\begin{proof}
Choose $k\in \{1,\ldots,n\}$ such that
\begin{equation}\label{eq:k choice}
r=\left(\frac{1}{n}\sum_{j=1}^nd_X(x_k,x_j)^p\right)^{\frac1{p}}.
\end{equation}
 The rightmost inequality in~\eqref{eq:Dis the moment up to 2} follows from averaging the estimate $d_X(x_i,x_j)^p\le 2^{p-1}d_X(x_i,x_k)^p+2^{p-1}d_X(x_k,x_j)^p$
 over $i,j\in \{1,\ldots,n\}$. The leftmost inequality in~\eqref{eq:Dis the moment up to 2} follows from the definition of $r$.
\end{proof}
\begin{lemma}\label{lem:AB dichotomy}
Continuing with the notation of the statement and proof of
Lemma~\ref{lem:D triangle}, in particular choosing $k\in
\{1,\ldots,n\}$ so as to satisfy~\eqref{eq:k choice}, write
\begin{equation}\label{eq:4D ball is B}
B\eqdef \left\{i\in \{1,\ldots,n\}:\ d_X(x_i,x_k)\le 4r\right\},
\end{equation}
and
\begin{equation}\label{eq:A choice}
M\eqdef \left\{(i,j)\in B\times B:\ d_X(x_i,x_j)\ge \frac{r}{8}\right\}.
\end{equation}
Then,
\begin{equation}\label{eq:AB implication}
|M|\le \frac{n^2}{2^{7p}}\implies \left(\frac{1}{n}\sum_{j\in \{1,\ldots,n\}\setminus B} d_X(x_j,x_k)^p\right)^{\frac1{p}}\ge \frac{r}{8}.
\end{equation}
\end{lemma}
\begin{proof}
Suppose that
\begin{equation}\label{eq:A small assumption}
|M|\le \frac{n^2}{2^{7p}}.
\end{equation}
 Recalling~\eqref{eq:k choice} and~\eqref{eq:4D ball is B}, it follows from Markov's inequality that
\begin{equation}\label{eq:markov for B}
\frac{n-|B|}{n}\le \frac{1}{4^p}.
\end{equation}
Hence,
\begin{eqnarray}\label{eq:sum off B lower}
&&\!\!\!\!\!\!\!\!\!\!\!\!\!\!\!\!\!\!\!\!\!\!\!\!\!\!\!\!\nonumber\frac{2}{n^2}\sum_{i=1}^n \sum_{j\in \{1,\ldots,n\}\setminus B} d_X(x_i,x_j)^p\\&\le& \nonumber \frac{2^p}{n^2}\sum_{i=1}^n \sum_{j\in \{1,\ldots,n\}\setminus B}\left(d_X(x_i,x_k)^p+d_X(x_k,x_j)^p\right)\\
&\stackrel{\eqref{eq:k choice}}{=}&\frac{2^p(n-|B|)}{n}r^p+\frac{2^p}{n}\sum_{j\in \{1,\ldots,n\}\setminus B}d_X(x_k,x_j)^p\nonumber\\
&\stackrel{\eqref{eq:markov for B}}{\le}& \frac{r^p}{2^p}+\frac{2^p}{n}\sum_{j\in \{1,\ldots,n\}\setminus B}d_X(x_k,x_j)^p.
\end{eqnarray}
Since  every $i,j\in B$ satisfy $d_X(x_i,x_j)\le
d_X(x_i,x_k)+d_X(x_j,x_k)\le 8r$,
\begin{eqnarray}\label{eq:use A}
\nonumber\frac{1}{n^2}\sum_{(i,j)\in B\times B}d_X(x_i,x_j)^p&\stackrel{\eqref{eq:4D ball is B}\wedge \eqref{eq:A choice}}{\le} & \frac{8^pr^p|M|}{n^2}+\frac{|B|^2-|M|}{n^2}\cdot\frac{r^p}{8^p}\\&\stackrel{\eqref{eq:A small assumption}}{\le}& \frac{r^p}{2^{4p}}+\frac{r^p}{2^{3p}}.
\end{eqnarray}
It follows that
\begin{eqnarray*}
r^p&\stackrel{\eqref{eq:Dis the moment up to 2}}{\le}&\frac{1}{n^2}\sum_{i=1}^n\sum_{j=1}^n d_X(x_i,x_j)^p\\&\stackrel{\eqref{eq:sum off B lower}\wedge\eqref{eq:use A}}{\le}& \frac{r^p}{2^{4p}}+\frac{r^p}{2^{3p}}+\frac{r^p}{2^p}+\frac{2^p}{n}\sum_{j\in \{1,\ldots,n\}\setminus B}d_X(x_k,x_j)^p,
\end{eqnarray*}
which (since $p\ge 1$) implies that
\begin{equation*}
\frac{1}{n}\sum_{j\in \{1,\ldots,n\}\setminus B}d_X(x_k,x_j)^p\ge \frac{r^p}{8^p}.\qedhere
\end{equation*}
\end{proof}

\begin{lemma}\label{lem:si good}
Continuing with the notation of the statements and proofs of
Lemma~\ref{lem:D triangle} and Lemma~\ref{lem:AB dichotomy}, define
$$
\forall\, i\in \{i,\ldots,n\},\qquad s_i\eqdef \max\left\{0,d_X(x_i,x_k)-2r\right\}.$$ If $|M|\le n^2/2^{7p}$ then
$$
\sum_{i=1}^n\sum_{j=1}^n |s_i-s_j|^p\ge \frac{1}{2^{5p}}\sum_{i=1}^n\sum_{j=1}^n d_X(x_i,x_k)^p.
$$
\end{lemma}
\begin{proof}
Due to Lemma~\ref{lem:AB dichotomy} we know that
\begin{equation}\label{eq:use dichotomy lemma}
\frac{1}{n}\sum_{j\in \{1,\ldots,n\}\setminus B}d_X(x_k,x_j)^p\ge \frac{r^p}{8^p}
\end{equation}
Define $ B'\eqdef \left\{i\in \{1,\ldots,n\}:\ d_X(x_i,x_k)\le
2r\right\}$. If $i\in B'$ then $s_i=0$, and if $j\in
\{1,\ldots,n\}\setminus B$ then  $s_j=d_X(x_j,x_k)-2r\ge \frac12
d_X(x_j,x_k)$. Also, recalling~\eqref{eq:k choice}, it follows from
Markov's inequality that
\begin{equation}\label{eq:B' big}
\frac{n-|B'|}{n}\le \frac{1}{2^p}.
\end{equation}
Consequently,
\begin{align*}
&\!\!\!\!\!\!\!\!\!\!\!\!\!\!\!\!\!\!\!\!\!\sum_{i=1}^n\sum_{j=1}^n |s_i-s_j|^p\ge 2|B'|\sum_{j\in \{1,\ldots,n\}\setminus B} \frac{d_X(x_j,x_k)^p}{2^p}\\
&\stackrel{\eqref{eq:use dichotomy lemma}\wedge\eqref{eq:B' big}}{\ge} \frac{n\left(1-\frac{1}{2^p}\right)}{2^{p-1}}\cdot \frac{nr^p}{8^p}\stackrel{\eqref{eq:Dis the moment up to 2}}{\ge}
\frac{1}{2^{5p}}\sum_{i=1}^n\sum_{j=1}^n d_X(x_i,x_j)^p.
\qedhere
\end{align*}
\end{proof}

\begin{proof}[Proof of Proposition~\ref{prop:into R}]
Fix $n\in \N$. Suppose that $x_1,\ldots,x_n\in X$ and write
$S=\{x_1,\ldots,x_n\}\subseteq X$. Define $r\in (0,\infty)$ as
in~\eqref{eq:defD} and let $\mu$ be a probability distribution over
$2^X$ that satisfies~\eqref{eq:def zero set} with $\Delta=r/8$,
$\zeta=\zeta(S;\d)$ and $\delta=r$ . Let $M$ be defined as
in~\eqref{eq:A choice} and suppose that $|M|> n^2/2^{7p}$. Then
\begin{align}
\nonumber &\frac{1}{n^2}\sum_{i=1}^n\sum_{j=1}^n \int_{2^X}\left|d_X(x_i,Z)-d_X(x_j,Z)\right|^pd\mu(Z)\\
&\ge \frac{1}{n^2}\sum_{(i,j)\in M} \left(\frac{r}{8\zeta}\right)^p\mu\left(\left\{Z\in 2^X:\ x_i\in Z\ \wedge \ d_X(x_j,Z)\ge \frac{r}{8\zeta}\right\}\right)\nonumber\\
&
\ge\frac{|M|}{n^2}\cdot \frac{\d r^p}{8^p\zeta^p}\label{eq:use zero set axioms}\\
&\ge
\frac{\d}{2^{11p}\zeta^p}\cdot \frac{1}{n^2}\sum_{i=1}^n\sum_{j=1}^n
d_X(x_i,x_j)^p,\label{eq:use lower M}
\end{align}
where~\eqref{eq:use zero set axioms} uses~\eqref{eq:def zero set} and~\eqref{eq:A choice}, and~\eqref{eq:use lower M} uses~\eqref{eq:Dis the moment up to 2} and the assumption $|M|> n^2/2^{7p}$. It follows that there exists $Z\subseteq X$ such that the
$1$-Lipschitz function  $f:X\to \R$ given by $f(x)=d_X(x,Z)$
satisfies
$$
\sum_{i=1}^n\sum_{j=1}^n |f(x_i)-f(x_j)|^p\ge \frac{\d}{2^{11p}\zeta^p}\sum_{i=1}^n\sum_{j=1}^n d_X(x_i,x_j)^p.
$$
If, on the hand, $|M|\le n^2/2^{7p}$ then the desired estimate
follows by choosing $f(x)=\max\{0,d_X(x,x_k)-2r\}$ and applying
Lemma~\ref{lem:si good}.
\end{proof}

\begin{proof}[Proof of Proposition~\ref{prop:weak embedding into Y}]
Fix $n\in \N$ and choose $x_1,\ldots,x_n\in X$. Define $r\in
(0,\infty)$ as in~\eqref{eq:defD} and $k\in \{1,\ldots,n\}$ as
in~\eqref{eq:k choice}. Let $M$ be defined as in~\eqref{eq:A choice}
and suppose that $|M|> n^2/2^{7p}$. An application
of~\eqref{eq:localized weak def} with $\Delta=r/8$ and $z=x_k$ shows
that
\begin{multline*}
\sum_{i=1}^n\sum_{j=1}^n d_Y\left(f_\Delta^{x_k}(x_i),f_\Delta^{x_k}(x_j)\right)^p \ge \sum_{(i,j)\in M}d_Y\left(f_\Delta^{x_k}(x_i),f_\Delta^{x_k}(x_j)\right)^p\\
\ge |M|\frac{\|f\|_{\Lip}^p}{D^p}\cdot\frac{r^p}{8^p}\ge
\frac{n^2r^p\|f\|_{\Lip}^p}{2^{10p}D^p}\ge
\frac{\|f\|_{\Lip}^p}{2^{11p}D^p}\sum_{i=1}^n\sum_{j=1}^n
d_X(x_i,x_j)^p.
\end{multline*}
This yields the desired average distortion embedding into $Y$. If,
on the hand, $|M|\le n^2/2^{7p}$ then the existence of the desired
embedding into $\R$ follows by choosing
$f(x)=\max\{0,d_X(x,x_k)-2r\}$ and applying Lemma~\ref{lem:si good}.
\end{proof}

It is natural to ask how the quantities $Av_Y^{(p)}(X)$ and  $Av_Y^{(q)}(X)$ are related to each other for distinct $p,q\in [1,\infty)$ and two metric space $(X,d_X)$ and $(Y,d_Y)$. We shall now briefly address this matter.

Suppose that $D> Av_Y^{(q)}(X)$ and fix $x_1,\ldots,x_n\in X$. Then there exists a nonconstant mapping $f:\{x_1,\ldots,x_n\}\to Y$ such that
\begin{equation}\label{eq:av dist assumption q}
\sum_{i=1}^n\sum_{j=1}^n d_Y(f(x_i),f(x_j))^q\ge
\frac{\|f\|_{\Lip}^q}{D^q}\sum_{i=1}^n\sum_{j=1}^n d_X(x_i,x_j)^q.
\end{equation}
Suppose first that $q<p$, and continue using the notation of
Lemma~\ref{lem:D triangle} and Lemma~\ref{lem:AB dichotomy}. If $|M|> n^2/2^{7p}$ then
\begin{multline}\label{eq:on M q}
\left(\frac{1}{n^2}\sum_{i=1}^n\sum_{j=1}^n d_X(x_i,x_j)^q\right)^{\frac{1}{q}}\ge \left(\frac{1}{n^2}\sum_{(i,j)\in M} d_X(x_i,x_j)^q\right)^{\frac{1}{q}}\\\stackrel{\eqref{eq:A choice}}{\ge} \left(\frac{1}{2^{7p}}\cdot \frac{r^q}{8^q}\right)^{\frac{1}{q}}\stackrel{\eqref{eq:Dis the moment up to 2}}{\gtrsim} \frac{1}{2^{7p/q}}\left(\frac{1}{n^2}\sum_{i=1}^n\sum_{j=1}^n d_X(x_i,x_j)^p\right)^{\frac{1}{p}}.
\end{multline}
Consequently,
\begin{multline}\label{eq:q->p p>q}
\left(\frac{1}{n^2}\sum_{i=1}^n\sum_{j=1}^n d_Y(f(x_i),f(x_j))^p\right)^{\frac{1}{p}}\ge \left(\frac{1}{n^2}\sum_{i=1}^n\sum_{j=1}^n d_Y(f(x_i),f(x_j))^q\right)^{\frac{1}{q}}\\
\stackrel{\eqref{eq:av dist assumption q}\wedge \eqref{eq:on M q}}{\gtrsim} \frac{\|f\|_{\Lip}}{2^{7p/q}D}\left(\frac{1}{n^2}\sum_{i=1}^n\sum_{j=1}^n d_X(x_i,x_j)^p\right)^{\frac{1}{p}}.
\end{multline}
By using Lemma~\ref{lem:si good} if $|M|\le n^2/2^{7p}$, we deduce from~\eqref{eq:q->p p>q} that
\begin{equation}\label{eq:transference p>q}
q<p\implies Av^{(p)}_{Y\times \R}(X)\lesssim 2^{7p/q} Av^{(q)}_Y(X),
\end{equation}
where $Y\times \R$ is understood to be equipped with, say, the $\ell_1$-sum metric, i.e., $d_{Y\times \R}((x,s),(y,t))=d(x,y)+|s-t|$ for every $(x,s),(y,t)\in Y\times \R$ (as in Proposition~\ref{prop:weak embedding into Y}, we can conclude here that there exists an average distortion embedding into either $Y$ or $\R$, but we choose to work with $Y\times \R$ for notational simplicity).

If $p<q$ then the following argument establishes  an estimate analogous to~\eqref{eq:transference p>q}, but under an additional assumption. The Lipschitz extension constant for the pair of metric spaces $(X,Y)$, denoted $e(X,Y)$, is the infimum over those $K\in [1,\infty)$ such that for every $S\subset X$ every Lipschitz function $f:S\to Y$ admits an extension $F:X\to Y$ with $\|F\|_{\Lip}\le K\|f\|_{\Lip}$. If no such $K$ exists then set $e(X,Y)=\infty$. Suppose that $p<q$ and that $|M|> n^2/2^{7p}$. By the definition of $D$ there exists a nonconstant mapping $\f:B\to Y$ such that
\begin{equation}\label{eq:B times B q}
\sum_{(i,j)\in B\times B} d_Y(\f(x_i),\f(x_j))^q\ge
\frac{\|\f\|_{\Lip}^q}{D^q}\sum_{(i,j)\in B\times B} d_X(x_i,x_j)^q.
\end{equation}
Note that since for every $i,j\in B$ we have $d_X(x_i,x_j)\le 8r$,
\begin{multline}\label{eq:max out of sum}
\sum_{(i,j)\in B\times B} d_Y(\f(x_i),\f(x_j))^q\\
\le\left(8r\|\f\|_{\Lip}\right)^{q-p}\sum_{(i,j)\in B\times B} d_Y(\f(x_i),\f(x_j))^p.
\end{multline}
Also, arguing as in~\eqref{eq:on M q} we have
\begin{equation}\label{eq:q on B times B}
\frac{1}{n^2}\sum_{(i,j)\in B\times B} d_X(x_i,x_j)^q\gtrsim (cr)^q,
\end{equation}
where $c\in (0,\infty)$ is a universal constant. By substituting~\eqref{eq:max out of sum} and~\eqref{eq:q on B times B} into~\eqref{eq:B times B q} we therefore have
\begin{multline}\label{eq:average on B times B}
\left(\frac{1}{n^2}\sum_{(i,j)\in B\times B} d_Y(\f(x_i),\f(x_j))^p\right)^{\frac{1}{p}}\ge \frac{(c/8)^{q/p}\|\f\|_{\Lip}r}{D^{q/p}}\\
\stackrel{\eqref{eq:Dis the moment up to 2}}{\ge}\frac{(c/8)^{q/p}\|\f\|_{\Lip}}{D^{q/p}}\left(\frac{1}{n^2}\sum_{i=1}^n\sum_{j=1}^n d_X(x_i,x_j)^p\right)^{\frac{1}{p}}.
\end{multline}
By extending $\f$ to a mapping $\Phi:X\to Y$ with $\|\Phi\|_{\Lip}\lesssim e(X,Y)\|\f\|_{\Lip}$ we see that
\begin{align}\label{eq:case p<q}
\nonumber \left(\frac{1}{n^2}\sum_{i=1}^n\sum_{j=1}^n d_Y(\Phi(x_i),\Phi(x_j))^p\right)^{\frac{1}{p}}&\ge \left(\frac{1}{n^2}\sum_{(i,j)\in B\times B} d_Y(\f(x_i),\f(x_j))^p\right)^{\frac{1}{p}}\\
&\!\!\!\!\!\!\!\!\!\!\!\!\!\!\!\!\!\!\!\!\!\!\!\!\!\! \! \! \! \! \! \! \! \! \! \! \! \! \! \! \! \! \! \! \! \! \! \! \! \! \! \! \! \! \! \! \!  \stackrel{\eqref{eq:average on B times B}}{\ge} \frac{\|\Phi\|_{\Lip}}{(CD)^{q/p}e(X,Y)}\left(\frac{1}{n^2}\sum_{i=1}^n\sum_{j=1}^n d_X(x_i,x_j)^p\right)^{\frac{1}{p}},
\end{align}
where $C\in (0,\infty)$ is a universal constant. By using Lemma~\ref{lem:si good} if $|M|\le n^2/2^{7p}$, we deduce from~\eqref{eq:case p<q} that
\begin{equation}\label{eq:av a>p}
p<q\implies Av^{(p)}_{Y\times \R}(X)\lesssim e(X,Y) \left(CAv^{(q)}_Y(X)\right)^{q/p}.
\end{equation}

By~\cite{Ball,NPSS06} for $p\in [2,\infty)$ we have $e(\ell_p,\ell_2)\lesssim \sqrt{p}$. It therefore follows from~\eqref{eq:transference p>q} and~\eqref{eq:av a>p} combined with Corollary~\ref{coro:p av embed} that for every $p\in [2,\infty)$ and $q\in [1,\infty)$ we have
\begin{equation}\label{eq:av other powers}
Av_{\ell_2}^{(q)}(\ell_p)\lesssim \left\{\begin{array}{ll} 2^{4q}p&\mathrm{if}\ q\ge 2,\\
p^{\frac{2}{q}+\frac12} &\mathrm{if}\ q< 2. \end{array}\right.
\end{equation}
It seems unlikely that~\eqref{eq:av other powers} is sharp.

\section*{Appendix: a refinement of Markov type}

Below is an application of Theorem~\ref{thm:p bound gamma} that I
found in collaboration with Yuval Peres. I thank him for agreeing to
include it here.

Fix $n\in \N$ and an $n$ by $n$ symmetric stochastic matrix
$A=(a_{ij})$. Then for every $m\in \N$ and $x_1,\ldots,x_n\in \R$ we
have
\begin{equation}\label{eq:geometric sum markov type}
\frac{\sum_{i=1}^n\sum_{j=1}^n (A^m)_{ij}(x_i-x_j)^2}{\sum_{i=1}^n\sum_{j=1}^n
a_{ij}(x_i-x_j)^2} \le
\sum_{t=0}^{m-1}\lambda_2(A)^t.
\end{equation}
\eqref{eq:geometric sum markov type} becomes evident when one
expresses the vector $(x_1,\ldots,x_n)\in \R^n$ in an orthonormal
eigenbasis of $A$, showing also that the multiplicative factor
$1+\lambda_2(A)+\ldots+\lambda_2(A)^{m-1}$ is sharp. By using the
estimate $|\lambda_2(A)|\le 1$, it follows from~\eqref{eq:geometric
sum markov type} that Hilbert space has Markov type $2$ with
$M_2(\ell_2)=1$; this was Ball's original proof of this fact
in~\cite{Ball}.

Suppose that $p\in (2,\infty)$. In~\cite{NPSS06} it was shown that
$\ell_p$ has Markov type $2$, and in fact that $M_2(\ell_p)\lesssim
\sqrt{p}$. This is the same as asserting the following inequality,
which holds true for every $n$ by $n$ symmetric stochastic matrix
$A=(a_{ij})$ and every $x_1,\ldots,x_n\in \ell_p$.
\begin{equation}\label{NPSS quote again}
\forall\, m\in \N,\qquad \frac{\sum_{i=1}^n\sum_{j=1}^n (A^m)_{ij}\|x_i-x_j\|_{\ell_p}^2}{\sum_{i=1}^n\sum_{j=1}^n
a_{ij}\|x_i-x_j\|_{\ell_p}^2}\lesssim pm.
\end{equation}
It is natural to ask whether or not one can refine
inequality~\eqref{NPSS quote again} in the spirit
of~\eqref{eq:geometric sum markov type} so as to yield an estimate
in terms of $\lambda_2(A)$ that becomes~\eqref{NPSS quote again} if
one uses the a priori bound $|\lambda_2(A)|\le 1$. Below we will
show how a combination of~\cite{NPSS06} and Theorem~\ref{thm:p bound
gamma} yields the following estimate, thus answering this question
positively.
\begin{equation}\label{eq:refined lp our}
\frac{\sum_{i=1}^n\sum_{j=1}^n
(A^m)_{ij}\|x_i-x_j\|_{\ell_p}^2}{\sum_{i=1}^n\sum_{j=1}^n
a_{ij}\|x_i-x_j\|_{\ell_p}^2} \lesssim p\sum_{t=0}^{m-1}
\left(1-\frac{2}{p}\left(1-\lambda_2(A)\right)\right)^{t}.
\end{equation}

To explain the similarity of~\eqref{eq:refined lp our}
to~\eqref{eq:geometric sum markov type}, note that if $m\ge p/2$
then~\eqref{eq:refined lp our} has the following equivalent form.
\begin{equation}\label{eq:m ge p}
\frac{\sum_{i=1}^n\sum_{j=1}^n
(A^m)_{ij}\|x_i-x_j\|_{\ell_p}^2}{\sum_{i=1}^n\sum_{j=1}^n
a_{ij}\|x_i-x_j\|_{\ell_p}^2} \lesssim p^2\sum_{t=1}^{\lceil 2m/p\rceil}\lambda_2(A)^t.
\end{equation}
Also, if $\lambda_2(A)$ is positive and bounded away from $0$, say,
if $\lambda_2(A)\ge 1/2$, then inequality~\eqref{eq:refined lp
our} has the following equivalent form.
\begin{equation}\label{eq:lambda2 ge 1/2}
\frac{\sum_{i=1}^n\sum_{j=1}^n
(A^m)_{ij}\|x_i-x_j\|_{\ell_p}^2}{\sum_{i=1}^n\sum_{j=1}^n
a_{ij}\|x_i-x_j\|_{\ell_p}^2} \lesssim p\sum_{t=0}^{m-1} \lambda_2(A)^{2t/p}.
\end{equation}

The proof of Lemma~\ref{lem:almost monotonicity} below is a simple
variant of an unpublished argument of Mark Braverman (2009), who
proved the same statement for $p=1$; see Exercise~13.10
in~\cite{LP13}. The factor $2^p$ in inequality~\eqref{eq:2^p
monotonicity} below is asymptotically sharp as $n\to \infty$: this
follows mutatis mutandis from an unpublished argument of Oded
Schramm (2007), who proved the same statement when $p=1$; see
Exercise~13.10 in~\cite{LP13}.

\begin{lemma}\label{lem:almost monotonicity}
Fix $p\in [1,\infty)$ and a metric space $(X,d_X)$. Fix also $n\in \N$ and an $n$ by $n$ symmetric stochastic matrix $A=(a_{ij})$. Then for every $s,t\in \N$ with $t\ge s$ and every $x_1,\ldots,x_n\in X$ we have
\begin{equation}\label{eq:2^p monotonicity}
\sum_{i=1}^n\sum_{j=1}^n (A^{2s})_{ij}d_X(x_i,x_j)^p\le 2^p\sum_{i=1}^n\sum_{j=1}^n (A^t)_{ij}d_X(x_i,x_j)^p.
\end{equation}
and
\begin{equation}\label{eq:2^p monotonicity odd}
\sum_{i=1}^n\sum_{j=1}^n (A^{2s+1})_{ij}d_X(x_i,x_j)^p\le 3^p\sum_{i=1}^n\sum_{j=1}^n (A^t)_{ij}d_X(x_i,x_j)^p.
\end{equation}
\end{lemma}

\begin{proof}
We have $d_X(x_i,x_j)^p\le 2^{p-1}d_X(x_i,x_k)^p+2^{p-1}d_X(x_j,x_k)^p$ for every $i,j,k\in \n$. For every $\ell\in \n$ multiply this inequality by $(A^s)_{i\ell}(A^s)_{j\ell}(A^{t-s})_{k\ell}$ and sum the resulting inequality over $i,j,k,\ell\in \n$, thus obtaining the following estimate.
\begin{align}\label{eq:st ell trick}
\nonumber &\sum_{i=1}^n\sum_{j=1}^n\sum_{k=1}^n\sum_{\ell=1}^n (A^s)_{i\ell}(A^s)_{j\ell}(A^{t-s})_{k\ell} d_X(x_i,x_j)^p\\\nonumber
&\le 2^{p-1} \sum_{i=1}^n\sum_{j=1}^n\sum_{k=1}^n\sum_{\ell=1}^n (A^s)_{i\ell}(A^s)_{j\ell}(A^{t-s})_{k\ell} d_X(x_i,x_k)^p\\&\quad +2^{p-1}\sum_{i=1}^n\sum_{j=1}^n\sum_{k=1}^n\sum_{\ell=1}^n (A^s)_{i\ell}(A^s)_{j\ell}(A^{t-s})_{k\ell} d_X(x_j,x_k)^p.
\end{align}
Since $A$ is symmetric and stochastic, \eqref{eq:st ell trick} is the same as~\eqref{eq:2^p monotonicity}.

To deduce~\eqref{eq:2^p monotonicity odd}, observe that the
convexity of $u\mapsto |u|^p$ implies that for every $i,j,k\in \n$
we have
\begin{equation}\label{eq:2/3-1/3}
 d_X(x_i,x_j)^p\le \left(\frac32\right)^{p-1}d_X(x_i,x_k)^p+3^{p-1} d_X(x_j,x_k)^p.
 \end{equation}
\eqref{eq:2^p monotonicity odd} now follows by multiplying~\eqref{eq:2/3-1/3} by $(A^{2s})_{ik}a_{kj}$, summing the resulting inequality over $i,j,k\in \n$, and using~\eqref{eq:2^p monotonicity}.
\end{proof}
\begin{corollary}\label{coro:compare to infinity} Fix $p\in [1,\infty)$ and $m,n\in \N$. Suppose that $A=(a_{ij})$ is an $n$ by $n$ symmetric stochastic matrix. Then for every metric space $(X,d_X)$ and every $x_1,\ldots,x_n\in X$ we have
$$
\sum_{i=1}^n\sum_{j=1}^n (A^m)_{ij} d_X(x_i,x_j)^p\le 3^p\gamma(A,d_X^p)\sum_{i=1}^n\sum_{j=1}^n a_{ij} d_X(x_i,x_j)^p.
$$
\end{corollary}

\begin{proof}
We may assume without loss of generality that $\gamma(A,d_X^p)<\infty$, i.e., that $A$ is ergodic. In this case we have $\lim_{t\to \infty} (A^t)_{ij}=1/n$ for every $i,j\in \n$. Therefore by Lemma~\ref{lem:almost monotonicity} (with $t\to\infty$),
\begin{align*}
\sum_{i=1}^n\sum_{j=1}^n (A^m)_{ij} d_X(x_i,x_j)^p&\le \frac{3^p}{n}\sum_{i=1}^n\sum_{j=1}^n d_X(x_i,x_j)^p\\&\stackrel{\eqref{eq:def gamma}}{\le}  3^p\gamma(A,d_X^p)\sum_{i=1}^n\sum_{j=1}^n a_{ij} d_X(x_i,x_j)^p.\qedhere
\end{align*}
\end{proof}

The following corollary is an immediate consequence of
Corollary~\ref{coro:compare to infinity} and the definition of the Markov type $p$ constant $M_p(X;m)$.

\begin{corollary}\label{coro:truncated Markov type} Fix $p\in [1,\infty)$ and let $(X,d_X)$ be a
metric space with Markov type $p$, i.e., $M_p(X)<\infty$. Then for
every $m,n\in \N$, every $n$ by $n$ symmetric stochastic matrix
$A=(a_{ij})$ and every $x_1,\ldots,x_n\in X$,
\begin{equation*}
\frac{\sum_{i=1}^n\sum_{j=1}^n (A^m)_{ij}
d_X(x_i,x_j)^p}{\sum_{i=1}^n\sum_{j=1}^n a_{ij} d_X(x_i,x_j)^p} \le
\min\big\{M_p(X;m)^pm,3^p\gamma(A,d_X^p)\big\}.
\end{equation*}
\end{corollary}

Assume from now on that $n\ge 3$, so that $\lambda_2(A)\ge -1/2$
(recall Lemma~\ref{lem:lower a priori}). Since
$1+(1-u)+\ldots+(1-u)^{m-1}\asymp \min\{m,1/u\}$ for every $u\in
(0,3/2]$ and $m\in \N$, inequality~\eqref{eq:refined lp our} is a
consequence of Corollary~\ref{coro:truncated Markov type} (with
$X=\ell_p$ and $p=2$), Theorem~\ref{thm:p bound gamma},
and~\eqref{NPSS quote again}.

To state two additional examples of consequences of this type, fix
$K\in [2,\infty)$ and let $(X,d_X)$ be a metric space that is
doubling with constant $K$. In~\cite{DLP13} it is shown that
$M_2(X)\lesssim \log K$, so in combination with
Corollary~\ref{cor:doubling} we deduce from
Corollary~\ref{coro:truncated Markov type} that
\begin{equation}\label{eq:doubling refined}
\frac{\sum_{i=1}^n\sum_{j=1}^n (A^m)_{ij}
d_X(x_i,x_j)^2}{\sum_{i=1}^n\sum_{j=1}^n a_{ij} d_X(x_i,x_j)^2}
\lesssim (\log K)^2\sum_{t=0}^{m-1}\lambda_2(A)^t.
\end{equation}
Similarly, it was shown that if $G$ is a connected planar graph then
$M_2(G,d_G)\lesssim 1$, so in combination with
Corollary~\ref{cor:doubling} we deduce from
Corollary~\ref{coro:truncated Markov type} that
\begin{equation}\label{eq:planar refined}
\frac{\sum_{i=1}^n\sum_{j=1}^n (A^m)_{ij}
d_X(x_i,x_j)^2}{\sum_{i=1}^n\sum_{j=1}^n a_{ij} d_X(x_i,x_j)^2}\lesssim
\sum_{t=0}^{m-1}\lambda_2(A)^t.
\end{equation}

Despite the validity of satisfactory spectral estimates such
as~\eqref{eq:refined lp our}, \eqref{eq:m ge p}, \eqref{eq:lambda2
ge 1/2}, \eqref{eq:doubling refined} and~\eqref{eq:planar refined},
they do not follow automatically only from the fact that the metric
space in question has Markov type $2$. Specifically, there exists a
metric space $(X,d_X)$ that has Markov type $2$, yet it is not true
that $\gamma(A,d_X^2)\lesssim_X 1/(1-\lambda_2(A))$ for every
symmetric stochastic matrix $A$. To see this, by
Theorem~\ref{thm:duality} it suffices to prove the following result.

\begin{theorem}\label{thm:direct sum tori}
There is a metric space $(X,d_X)$ with
$Av_{\ell_2}^{(2)}(X)=\infty$, yet $M_2(X)<\infty$, i.e., $X$ has
Markov type $2$.
\end{theorem}

\begin{proof} If $\Lambda\subset \R^n$ is a lattice of rank $n$ then denote the
length of the shortest nonzero vector in $\Lambda$ by $N(\Lambda)$.
Also, let  $r(\Lambda)$ denote the infimum over those $r\in
(0,\infty)$ such that Euclidean balls of radius $r$ centered at
$\Lambda$ cover $\R^n$. The dual lattice of $\Lambda$ is denoted
$\Lambda^*$; thus $\Lambda^*$ is the set of all $x\in \R^n$ such
that $\sum_{i=1}^n x_iy_i$ is an integer for every $y\in \Lambda$.

For every $n\in \N$ choose an arbitrary rank $n$ lattice
$\Lambda_n\subseteq \R^n$ that satisfies $r(\Lambda_n)\lesssim
N(\Lambda_n)$. See~\cite{Rog50} for the existence of such lattices.
Let $\R^n/\Lambda_n^*$ be the corresponding flat torus, equipped
with the natural Riemannian quotient metric
$d_{\R^n/\Lambda_n^*}(\cdot,\cdot)$. Also, let $\mu_n$ denote the
normalized Riemannian volume measure on $\R^n/\Lambda_n^*$.

Consider the $\ell_2$ product
$$
X\eqdef \left(\bigoplus_{n=1}^\infty \left(\R^n/\Lambda_n^*\right)\right)_2.
$$
 Thus $X$ consists of all the sequences $x= (x_n)_{n=1}^\infty\in \prod_{n=1}^\infty \left(\R^n/\Lambda_n^*\right)$
 such that $\sum_{n=1}^\infty d_{\R^n/\Lambda_n^*}(x_n,0)^2<\infty$, equipped with the
 metric given by $d_X\left(x,y\right)^2=\sum_{n=1}^\infty
 d_{\R^n/\Lambda_n^*}(x_n,y_n)^2$ for every $x,y\in X$.
%$$
%\forall\, x,y\in X,\qquad d_X\left(x,y\right)\eqdef \left(\sum_{n=1}^\infty d_{\R^n/\Lambda_n^*}(x_n,y_n)^2\right)^\frac12.
%$$
Since $\R^n/\Lambda_n^*$ has vanishing sectional curvature, it is an
Aleksandrov space of nonnegative curvature
(see~\cite[Sec.~3]{Oht09-markov}), and therefore by a theorem of
Ohta~\cite{Oht09-markov} it has Markov type $2$ constant at most
$1+\sqrt{2}$. Being an $\ell_2$ product of spaces with uniformly
bounded Markov type $2$ constant, $X$ also has Markov type $2$.

%To prove that $Av_2^{(2)}(X)=\infty$, we will show that
%$Av_2^{(2)}\left(\R^n/\Lambda_n^*\right)\gtrsim \sqrt{n}$, relying
%on results of~\cite{KN06}.

Fix $\e \in (0,1)$. By covering the fundamental parallelepiped of
$\Lambda_n^*$ by homothetic copies of itself, we see that there
exists a finite measurable partition $\{U_1,\ldots,U_k\}$ of the
torus $\R^n/\Lambda_n^*$ into sets of diameter at most $\e$ and
$\mu_n(U_i)=1/k$ for every $i\in \{1,\ldots,k\}$.

Fix an arbitrary point $x_i\in U_i$ and suppose that
$f:\{x_1,\ldots,x_k\}\to \ell_2$ is $1$-Lipschitz. Since
$M_2(\R^n/\Lambda_n^*)\le 1+\sqrt{2}\le 3$, by Ball's extension
theorem~\cite{Ball} there exists $F: \R^n/\Lambda_n^*\to \ell_2$
that extends $f$ and satisfies $\|F\|_{\Lip}\le 3$. By arguing as in
the proof of Lemma~\ref{lem:sharp in l2}, we have
\begin{multline}\label{eq:to contrast F on torus}
\frac{1}{k^2}\sum_{i=1}^k\sum_{j=1}^k\|f(x_i)-f(x_j)\|_{\ell_2}^2\\\le
3\int_{\R^n/\Lambda_n^*}\int_{\R^n/\Lambda_n^*}
\|F(x)-F(y)\|_{\ell_2}^2 d\mu_n(x)d\mu_n(y)+54\e^2,
\end{multline}
and
\begin{multline}\label{eq:average xi on torus}
\frac{1}{k^2}\sum_{i=1}^kd_{\R^n/\Lambda_n^*}(x_i,x_j)\\\ge
\frac13\int_{\R^n/\Lambda_n^*}\int_{\R^n/\Lambda_n^*}
d_{\R^n/\Lambda_n^*}(x,y)^2 d\mu_n(x)d\mu_n(y)-2\e^2,
\end{multline}

By~\cite[Lem.~11]{KN06}, for every Lipschitz mapping
$g:\R^n/\Lambda_n^*\to \ell_2$,
\begin{equation}\label{eq:upper torus}
\int_{\R^n/\Lambda_n^*}\int_{\R^n/\Lambda_n^*}
\|g(x)-g(y)\|_{\ell_2}^2 d\mu_n(x)d\mu_n(y)\lesssim
\frac{n\|g\|_\Lip^2}{N(\Lambda_n^*)^2}.
\end{equation}
Also, by~\cite[Lem.~10]{KN06} we have
\begin{equation}\label{eq:lower torus}
\int_{\R^n/\Lambda_n^*}\int_{\R^n/\Lambda_n^*}
d_{\R^n/\Lambda_n^*}(x,y)^2 d\mu_n(x)d\mu_n(y)\gtrsim \frac{n^2}{r(\Lambda^*_n)^2}.
\end{equation}
Hence, by letting $\e\to 0$ in~\eqref{eq:to contrast F on torus}
and~\eqref{eq:average xi on torus}, it follows from~\eqref{eq:upper
torus} and~\eqref{eq:lower torus} that
\begin{equation}\label{eq:torus average lower each n}
 Av_{\ell_2}^{(2)}\left(\R^n/\Lambda_n^*\right)\gtrsim \frac{N(\Lambda_n^*)}{r(\Lambda^*_n)}\sqrt{n}\gtrsim \sqrt{n},
\end{equation}
where we used the assumption $r(\Lambda_n)\lesssim N(\Lambda_n)$.
Now $Av_{\ell_2}^{(2)}(X)=\infty$ follows from~\eqref{eq:torus
average lower each n} and the definition of $X$.
\end{proof}

\begin{remark}
The use of Ball's extension theorem in the proof of
Theorem~\ref{thm:direct sum tori} can be replaced by the use of
Kirszbraun's extension theorem~\cite{Kirsz34} combined with the
fact~\cite[Thm.~6]{KN06} (see also~\cite{HR13}) that
$c_2(\R^n/\Lambda_n^*)\lesssim_n 1$;
 in this case the factor $54$ in~\eqref{eq:to
contrast F on torus} would be replaced by a factor that depends on
$n$, but since we let $\e\to 0$ this does not affect the rest of the
proof (however, if one desires to bound $k$ as a function of $n$,
this approach would yield an inferior estimate).
\end{remark}

\bibliographystyle{alphaabbrvprelim}
\bibliography{spectral}

 \end{document}